\documentclass{article}
\usepackage{Macros}
\usepackage{needspace}

\hypersetup{
  pdfauthor={Shunzhi Zhang, Shichen Liao, Congying Han, Tiande Guo},
  pdftitle={Safeguarded Accelerated Adaptive Search under Unreliable Gradient Oracles},
  pdfsubject={Optimization and Control},
  pdfkeywords={stochastic optimization, adaptive step search, Nesterov acceleration, unreliable oracles}
}

\title{ \bfseries Safeguarded Accelerated Adaptive Search\\
under Unreliable Gradient Oracles}
\author{Shunzhi Zhang \quad Shichen Liao \quad Congying Han\thanks{%
  \protect\raggedright Corresponding author: Congying Han.~~
  \protect
  E-mail addresses: \authoremail{zhangshunzhi23@mails.ucas.ac.cn} (Shunzhi Zhang);
  \authoremail{liaoshichen20@mails.ucas.ac.cn} (Shichen Liao);
  \authoremail{hancy@ucas.ac.cn} (Congying Han);
  \authoremail{tdguo@ucas.ac.cn} (Tiande Guo).}\quad Tiande Guo\\[0.5em]
  {\large School of Mathematical Sciences}\\
  {\large University of Chinese Academy of Sciences}}
\date{}

\begin{document}
\maketitle
\begin{abstract}
From the probabilistic-model viewpoint, stochastic adaptive search can require any fraction $\theta\in(0,1)$ of the predicted model reduction for acceptance. Classical backtracking with Nesterov acceleration pairs its momentum construction with a test certifying $1/2$ of this reduction. In stochastic search, however, rejection entails a fresh oracle query without updating the iterates. We therefore separate the fraction required by the search from that used to construct momentum, allowing smaller search fractions to ease acceptance and potentially reduce the oracle cost of reaching a target accuracy. Safeguarded Accelerated Adaptive Search (\texttt{SAAS}) addresses smooth convex optimization with unreliable gradients and inexact function value oracles. We establish high-probability trial bounds under finite-moment gradient errors allowing bias and infinite variance, or bounded corruption subject to relative-accuracy conditions. Numerical experiments demonstrate that the preferred balance between search and momentum varies with gradient reliability, and that decoupling the scales enables stable and rapid progress.
\end{abstract}

\section{Introduction}\label{sec:Introduction}

Nesterov-type methods accelerate gradient and proximal-gradient iterations by adding an extrapolation step, commonly interpreted as momentum. They attain optimal first-order worst-case complexity for smooth convex optimization and for composite convex optimization with a proximable nonsmooth term 
\cite{Nes1983-1st,NemirovskiYudin1983,FISTA,POGM}.
 Classical formulations set the step size using a global smoothness constant $L$. Backtracking variants instead adapt the step size through a local test: the method accepts a trial if a local smoothness condition holds at the trial point; otherwise, it rejects the trial and reduces the step size for the next trial \cite{BKTR,FISTA,localLip,Nestrov2014book,mu>0adaptive}. With exact gradients and function values, $L$-smoothness guarantees acceptance once the trial step size is sufficiently small, typically on the scale of $1/L$. Backtracking can therefore adapt to local geometry while retaining the accelerated worst-case order.

With inexact gradients and possibly inexact function values, however, a rejected trial
may result from an excessive step size, inaccurate information, or both
\cite{stoFISTA,Bonettini2025Linesearch}. When a trial is accepted,
the gradient error affects both the current step and subsequent
momentum updates.
 We model access to these evaluations through unreliable first- and zeroth-order oracles, which return approximate gradients and function values, respectively. We ask whether adaptive search can preserve Nesterov-type acceleration with the information supplied by these oracles.
Specifically, for an objective function $\phi:\mathbb R^n\to\mathbb R$ satisfying \Cref{assump:SC-Lsmooth,assump:global-min}, we consider the unconstrained problem
\[
\min_{\bm{x}\in\mathbb R^n}\,\phi(\bm{x}).
\]
Here \(\phi\) is differentiable and convex, its gradient is
\(L\)-Lipschitz continuous, and its minimum is attained.
We treat both the general convex case \(\mu=0\) and the
strongly convex case with modulus \(\mu>0\).

A closely related method is the {stochastic \texttt{FISTA} step search} of~\cite{stoFISTA}, which we denote by \texttt{FISTA-SS}. This method extends full-backtracking \texttt{FISTA} ~\cite{BKTR} to stochastic first-order information for the composite objective $F:=\phi+h:\mathbb R^n\to\mathbb R\cup\{+\infty\}$, where $h$ is proper, closed, convex, and proximable. At each trial, momentum determines an extrapolated point $\bm{y}$. Given a trial step size $\gamma>0$ and an estimate $\bf{G}$ of $\nabla\phi(\bm{y})$, the method computes $\bm{x}^+:=\operatorname{prox}_{\gamma h}(\bm{y}-\gamma\mathbf{G})$\footnote{%
The proximal operator is defined by
\(\operatorname{prox}_{\gamma h}(\bm{v})
:=\operatorname*{arg\,min}_{\bm{u}\in\mathbb{R}^n}
\{h(\bm{u})+\|\bm{u}-\bm{v}\|^2/(2\gamma)\}\).} and the step $\bm{s}(\gamma;\bm{y}):=\bm{x}^+-\bm{y}$. It accepts the trial when
\begin{equation}\label{eq:intro_fista_ss_acceptance}
F(\bm{x}^+)
\le
\phi(\bm{y})+\langle\mathbf{G},\bm{s}(\gamma;\bm{y})\rangle+h(\bm{x}^+)
+{\|\bm{s}(\gamma;\bm{y})\|^2}/{(2\gamma)}.
\end{equation}
If this condition fails, the method rejects the trial, reduces the next trial step size, and obtains a new gradient estimate. Successive trials can therefore change both the direction and the scale of the step.
The first-order oracle in \cite{stoFISTA} is required to satisfy the relative-accuracy condition
\[
\|\mathbf{G}-\nabla\phi(\bm{y})\|
\le
(\kappa/\gamma)\|\bar {\bm{s}}(\gamma;\bm{y})\|
\qquad\text{where}\qquad
\bar {\bm{s}}(\gamma;\bm{y})
\coloneqq 
\operatorname{prox}_{\gamma h}
\bigl(\bm{y}-\gamma\nabla\phi(\bm{y})\bigr)-\bm{y}
\]
with conditional probability at least $p>1/2$, given the information available before the trial. With exact function and proximal evaluations and a suitable decay bound on the conditional mean-square gradient error, \texttt{FISTA-SS} establishes an accelerated bound on the expected number of trials needed to reach any accuracy.

More generally, model-based adaptive search decides whether to accept a trial by comparing the objective decrease with the reduction predicted by a local model 
\cite{Cartis2018,
scheinberg2025stochasticadaptiveoptimizationunreliable}.
At a base point $\bm{x}$, a local model $m_{\bm{x}}$ of $\phi$ and a positive search parameter $\gamma$ determine a step $\bm{s}(\gamma)$ and a candidate $\bm{x}^+=\bm{x}+\bm{s}(\gamma)$. The predicted reduction is $\Delta m_{\bm{x}}(\bm{s}):=m_{\bm{x}}(\bm{x})-m_{\bm{x}}(\bm{x}+\bm{s})$. When function values are exact and the predicted reduction is positive, the search accepts a candidate satisfying
\begin{equation}\label{eq:intro_model_reduction}
\phi(\bm{x}^+)-\phi(\bm{x})
\le
-\theta\,\Delta m_{\bm{x}}(\bm{s}),
\qquad
\theta\in(0,1).
\end{equation}
Thus $\theta$ specifies the fraction of the predicted model reduction that the trial must realize.

For a model built from unreliable oracle outputs, the probabilistic-model framework assumes that a prescribed accuracy condition holds with conditional probability at least a fixed $p\in(0,1]$ at each trial \cite{Cartis2018}. 
Its search analysis uses a \emph{small-step acceptance} property:
there exists \(\bar\gamma>0\) such that a trial satisfying
the accuracy condition is accepted whenever
\(\gamma\leq\bar\gamma\). The search reduces $\gamma$ after rejection and may increase it after acceptance. If accurate trials occur sufficiently often, these updates yield enough accepted trials with $\gamma$ bounded away from zero. With inexact function values, the test instead compares the observed function-value difference with the relaxed threshold $-\theta\Delta m_{\bm{x}}(\bm{s})+\tau_{\mathrm f}$.\footnote{Here $f(\cdot)$ denotes the zeroth-order oracle output approximating $\phi(\cdot)$, and the observed difference is $f(\bm{x}^+)-f(\bm{x})$. The additive relaxation term $\tau_{\mathrm f}\geq0$ accommodates the error $\bigl|\phi(\bm{x})-f(\bm{x})-(\phi(\bm{x}^+)-f(\bm{x}^+))\bigr|$ in this difference.} Small-step acceptance is then established when both the model and the function-value comparison satisfy their accuracy conditions~\cite{scheinberg2025stochasticadaptiveoptimizationunreliable}. Within this model-reduction framework, different choices of \(m_{\bm{x}}(\cdot)\), \(\bm{s}(\cdot)\), and the update rule for \(\gamma\) give rise to stochastic line/step-search, trust-region, adaptive-regularization, and related probabilistic adaptive methods
\cite{Cartis2018,gratton2018complexity,blanchet2019convergence,
bellavia2022adaptive,scheinberg2023cubic}. 

For the gradient methods studied in~\cite{Cartis2018,courtny,Albert2021,Xie2024},  let $\bf{G}$ approximate $\nabla\phi(\bm{x})$ at the model base point $\bm{x}$, and use the linear model and gradient step $m_{\bm{x}}(\bm{x}+\bm{s})=m_{\bm{x}}(\bm{x})+\langle\bf{G},\bm{s}\rangle$, 
$\bm{s}(\gamma)=-\gamma\bf{G}$. Here $\gamma$ is the step size, and the predicted reduction is $\Delta m_{\bm{x}}(\bm{s})=-\langle\mathbf{G},\bm{s}\rangle=\|\bm{s}\|^2/\gamma$. Condition
\eqref{eq:intro_model_reduction} therefore becomes the \emph{Armijo-type}~\cite{Armijo1966} condition
\begin{equation}\label{eq:intro_armijo_mechanism}
\phi(\bm{x}^+)-\phi(\bm{x})
\le
-({\theta}/{\gamma})\|\bm{x}^+-\bm{x}\|^2
=
-\theta\gamma\|\mathbf{G}\|^2,
\qquad
\theta\in(0,1).
\end{equation} Unlike classical Armijo backtracking along a fixed direction, these
stochastic adaptive-search methods may obtain a new gradient realization after each rejection 
\cite{courtny,Albert2021,Xie2024,
scheinberg2025stochasticadaptiveoptimizationunreliable}. This  mechanism  is termed as \emph{adaptive step search} in~\cite{Xie2024}
and extended to stochastic proximal-gradient methods in~\cite{stoFISTA}.

Small-step acceptance follows directly from \(L\)-smoothness in this specialization. Suppose
that function values are exact and that 
$\|\mathbf{G}-\nabla\phi(\bm{x})\|
\le
\kappa\|\bm{s}(\gamma)\|
=
\kappa\gamma\|\mathbf{G}\|$.
Under this accuracy condition, 
\[
\begin{aligned}
\phi(\bm{x}^+)
\le
\phi(\bm{x})
+\langle\nabla\phi(\bm{x}),\bm{s}(\gamma)\rangle
+({L}/{2})\|\bm{s}(\gamma)\|^2
\le
\phi(\bm{x})
-\gamma\bigl[1-(L/2+\kappa)\gamma\bigr]
\|\mathbf{G}\|^2.
\end{aligned}
\]
Consequently, \eqref{eq:intro_armijo_mechanism} holds whenever
$\gamma
\le
\bar\gamma
\coloneqq 
{(1-\theta)}/{(L/2+\kappa)}$.
When accurate trials occur sufficiently often, this acceptance
property and the step-size update rule allow line-search based  stochastic adaptive 
methods to recover deterministic gradient-method complexity orders,
in expectation or with high probability, for suitable target
accuracies~\cite{Cartis2018,courtny,Albert2021,Xie2024,
scheinberg2025stochasticadaptiveoptimizationunreliable}.

\texttt{FISTA-SS} uses the same connection between oracle accuracy and small-step acceptance. Its test also has a (composite)  model-reduction interpretation as above, but with a particular reduction fraction of $1/2$:
Suppose $\bm{y}\in\operatorname{dom}h$ and write $\bm{s}_\gamma:=\bm{x}^+-\bm{y}$. Define the local composite model and its predicted reduction by
\[
\Delta m_{\bm{y}}(\bm{s}_\gamma;\mathbf{G})
\coloneqq 
m_{\bm{y}}(\bm{y};\mathbf{G})-m_{\bm{y}}(\bm{x}^+;\mathbf{G}),\qquad m_{\bm{y}}(\bm{u};\mathbf{G})
\coloneqq 
\phi(\bm{y})+\langle\mathbf{G},\bm{u}-\bm{y}\rangle+h(\bm{u}).
\]
Proximal optimality gives $\Delta m_{\bm{y}}(\bm{s}_\gamma;\bm{G})\geq\|\bm{s}_\gamma\|^2/\gamma$.\footnote{%
The proximal step satisfies
\(-\bf{G}-\bm{s}_\gamma/\gamma\in\partial h(\bm{x}^+)\);
the bound then follows from the subgradient inequality
with  comparison point \(\bm{y}\).} Hence the condition
\eqref{eq:intro_fista_ss_acceptance} implies
\[
\begin{aligned}
F(\bm{x}^+)-F(\bm{y})
\le
-\Delta m_{\bm{y}}(\bm{s}_\gamma;\mathbf{G})
+{\|\bm{s}_\gamma\|^2}/({2\gamma})
\le
-\tfrac12\,\Delta m_{\bm{y}}(\bm{s}_\gamma;\mathbf{G}),
\end{aligned}
\] leading to the interpretation.  Canceling \(h(\bm{x}^+)\) from both sides of condition
\eqref{eq:intro_fista_ss_acceptance} gives the equivalent condition
\begin{equation}\label{FISTA-SS_model_smooth}
\phi(\bm{x}^+)
\le
\phi(\bm{y})+\langle\mathbf{G},\bm{s}_\gamma\rangle
+{\|\bm{s}_\gamma\|^2}/{(2\gamma)}.
\end{equation} 
In the smooth case $h\equiv0$, the step is $\bm{s}_\gamma=-\gamma\bf{G}$, so \eqref{FISTA-SS_model_smooth} is exactly the Armijo condition at $\bm{y}$ with $\theta=1/2$.
The smooth specializations of several classical accelerated backtracking methods have this same half-scale interpretation 
\cite{BKTR,FISTA,localLip,Nestrov2014book,mu>0adaptive}, in contrast to
the general probabilistic-model framework, which allows any
\(\theta\in(0,1)\).

Reducing $\theta$ makes the acceptance test easier to satisfy, but certifies less decrease for a given trial step.  With exact gradients,
\(L\)-smoothness guarantees \eqref{eq:intro_armijo_mechanism}
whenever \(\gamma\le2(1-\theta)/L\). This threshold increases from $1/L$ at $\theta=1/2$ toward $2/L$ as $\theta$ decreases to zero. At the threshold, however, the certified decrease is
\[
\frac{2\theta(1-\theta)}{L}\|\nabla\phi(\bm{x})\|^2.
\]
The factor $2\theta(1-\theta)$ is maximized at $\theta=1/2$. Thus the half-scale choice maximizes this local decrease guarantee. Stochastic adaptive search, however, queries a fresh gradient at every trial, including after rejection. Under the accuracy condition above, reducing $\theta$ enlarges the acceptance threshold $\bar\gamma=(1-\theta)/(L/2+\kappa)$ and may reduce the queries spent on rejected trials. We therefore retain $\theta\in(0,1)$ as a free search parameter.

Another observation is that classical Nesterov momentum is matched to this half-scale decrease. A smaller search fraction \(\theta<1/2\) therefore calls for a corresponding adjustment of the momentum construction. Drawing on the additional freedom in the momentum update of \texttt{AGNES}~\cite{gupta2024nesterovaccelerationdespitenoisy}, we introduce separate search and momentum scales, \(0<\theta_{\mathrm m}\le\theta<1\). The search requires the $\theta$-scaled reduction, while momentum uses the $\theta_{\mathrm m}$-scaled part. Of particular interest is
\[
0<\theta_{\mathrm m}<\theta<1/2,
\]
which combines a search fraction below the classical half-scale with a positive margin $\theta\!-\!\theta_{\mathrm m}$ for controlling oracle errors. The same parameterization recovers classical Nesterov momentum as a special case.

We combine this momentum parameterization with the
\(\theta\)-scaled Armijo test, using zeroth-order outputs when
function values are inexact. When enabled, a model-consistency test must also pass for acceptance;
it compares the first-order model with function values at the current
and extrapolated points.
After acceptance, a gradient-norm safeguard regulates the next trial step size: it reduces the step size when the gradient estimate has norm below a prescribed threshold and permits an increase otherwise. Together, these
choices define \emph{Safeguarded Accelerated Adaptive Search}
(\texttt{SAAS}). 

\subsection{Problem setting and oracle models}
\label{subsec:prob_set}

\begin{assumption}[Convexity and smoothness]
\label{assump:SC-Lsmooth}
The function \(\phi\colon \mathbb R^n\to\mathbb R\) is differentiable, and
there exist \(L>0\) and \(0\le\mu\le L\) such that, for every
\(\bm{x},\bm{y}\in\mathbb R^n\),
\begin{equation}\label{eq:SC-Lsmooth}
\frac{\mu}{2}\|\bm{y}-\bm{x}\|^2
\le
\phi(\bm{y})-\phi(\bm{x})-\langle\nabla\phi(\bm{x}),\bm{y}-\bm{x}\rangle
\le
\frac{L}{2}\|\bm{y}-\bm{x}\|^2.
\end{equation}
\end{assumption}

\begin{assumption}[Existence of a minimizer]
\label{assump:global-min}
The value
\(\phi^\star\coloneqq \inf_{\bm{x}\in\mathbb R^n}\phi(\bm{x})\) is finite and attained.
\end{assumption}
The first- and zeroth-order oracles supply the gradient and function-value information used to form and test a trial. Given an extrapolated point $\bm{y}$ and a trial step size $\gamma>0$, the first-order oracle returns an estimate $\mathbf{G}(\bm{y},\gamma)$ of $\nabla\phi(\bm{y})$. This estimate defines the step $\bm{s}(\gamma):=\bm{x}^+-\bm{y}=-\gamma\mathbf{G}(\bm{y},\gamma)$. The zeroth-order outputs are used in the Armijo test \eqref{Check_1} and, when enabled, the model-consistency test \eqref{Check_2}.

\begin{assumption}[First-order oracles (\textbf{FO})]
\label{FO}
For every \(\bm{y}\in\mathbb R^n\) and \(\gamma>0\), the first-order oracle
returns
\(\mathbf{G}(\bm{y},\gamma)\sim\mathbb P_{\bm{y},\gamma}\), where
\((\bm{y},\gamma)\mapsto\mathbb P_{\bm{y},\gamma}\) is a measurable probability
kernel\footnote{For every Borel set \(\mathcal{A}\subseteq\mathbb R^n\),
the probability \(\mathbb P_{\bm{y},\gamma}(\mathcal{A})\) is a measurable
function of the inputs \((\bm{y},\gamma)\).} on \(\mathbb R^n\). For \(\bm{g}\in\mathbb R^n\), define
\[
w_{\bm{y}}(\bm{g})\coloneqq \|\bm{g}-\nabla\phi(\bm{y})\|,
\qquad
\bar w(\bm{y},\gamma)\coloneqq 
\int_{\mathbb R^n}w_{\bm{y}}(\bm{g})\,\mathbb P_{\bm{y},\gamma}(\mathrm{d}\bm{g}).
\]
Uniformly over \(\bm{y}\in\mathbb R^n\) and \(\gamma>0\), one of the
following conditions holds.

\begin{itemize}
\item
\emph{Stochastic first-order oracle (\(\textbf{sFO(}\epsilon_{\mathrm{g}},\upsilon,q \textbf{)}\)).}
There exists \(\epsilon_{\mathrm{g}}\ge0\) such that
\begin{equation}\label{eq:SFO_bias}
\bar w(\bm{y},\gamma)\le\epsilon_{\mathrm{g}}.
\end{equation}
In addition, the centered error magnitude satisfies one of two tail alternatives.
\begin{enumerate}[label=\textnormal{(\alph*)},leftmargin=*]
\item
\emph{Finite-moment tail.}
There exist \(q>1\) and \(\upsilon\ge0\) such that
\begin{equation}\label{eq:SFO_moment}
\int_{\mathbb R^n}
\left|w_{\bm{y}}(\bm{g})-\bar w(\bm{y},\gamma)\right|^q\,
\mathbb P_{\bm{y},\gamma}(\mathrm{d}\bm{g})
\le\upsilon.
\end{equation}

\item
\emph{Sub-Weibull tail.}%
\footnote{For every \(s\ge0\), exponential Markov's inequality gives
\(\mathbb P_{\bm{y},\gamma}
\{\bm{g}:\lvert w_{\bm{y}}(\bm{g})-\bar w(\bm{y},\gamma)\rvert>s\}
\le\min\{1,\exp(\upsilon-s^q)\}\).}
There exist \(q\in(1,2]\) and \(\upsilon>0\) such that
\begin{equation}\label{eq:SFO_subweibull}
\int_{\mathbb R^n}
\exp\!\left\{
\left|w_{\bm{y}}(\bm{g})-\bar w(\bm{y},\gamma)\right|^q
\right\}\,
\mathbb P_{\bm{y},\gamma}(\mathrm{d}\bm{g})
\le e^{\upsilon}.
\end{equation}
\end{enumerate}

\item
\emph{Corrupted first-order oracle (\(\textbf{cFO}\textbf{(}\epsilon_{\mathrm{g}},\kappa_{\mathrm{g}},\delta_{\mathrm{g}},\epsilon_{\mathrm{c}} \textbf{)}\)).}
There exist
\(\epsilon_{\mathrm{g}},\epsilon_{\mathrm{c}},\kappa_{\mathrm{g}}\ge0\) and
\(\delta_{\mathrm{g}}\in[0,1)\) such that
\begin{equation}\label{eq:CFO}
\begin{aligned}
\mathbb P_{\bm{y},\gamma}\!\left\{
\bm{g}:w_{\bm{y}}(\bm{g})\le\epsilon_{\mathrm{g}}+\kappa_{\mathrm{g}}\gamma\|\bm{g}\|
\right\}
\ge1-\delta_{\mathrm{g}},\qquad
\mathbb P_{\bm{y},\gamma}\!\left\{
\bm{g}:w_{\bm{y}}(\bm{g})\le\epsilon_{\mathrm{g}}+\epsilon_{\mathrm{c}}
\right\}
=1.
\end{aligned}
\end{equation}
\end{itemize}
\end{assumption}

\begin{remark}[ First-order errors]
\label{rem:FO_interpretation}
Both first-order models allow biased estimates and persistent gradient errors. 

In sFO, \eqref{eq:SFO_bias} controls the mean error magnitude, while
\eqref{eq:SFO_moment} and \eqref{eq:SFO_subweibull} describe its centered fluctuations. The finite-moment alternative permits any $q>1$, including infinite-variance errors when $1<q<2$. The Sub-Weibull alternative has $\exp(-s^q)$-type tails, with the sub-Gaussian endpoint at $q=2$ and the sub-exponential boundary approached as $q\downarrow1$ \cite{vladimirova2020subweibull}.
Jensen's inequality shows that \eqref{eq:SFO_subweibull}
implies \eqref{eq:SFO_moment} with the same \((q,\upsilon)\). A uniform $q$-moment bound on the full gradient error also implies sFO after adjusting the constants, covering fixed-batch stochastic gradients and biased approximations whenever this bound holds uniformly.

The two bounds in cFO serve different purposes. The probabilistic bound controls the error relative to the trial step, with $\epsilon_{\mathrm g}$ allowing a nonzero error even as that step becomes small. The almost-sure bound $\epsilon_{\mathrm g}+\epsilon_{\mathrm c}$ also applies when the probabilistic condition fails. It limits the error that an accepted trial can carry into subsequent momentum updates.
\end{remark}

\begin{assumption}[Bounded zeroth-order oracle (\(\textbf{bZO(}\epsilon_{\mathrm f}\textbf{)}\))]
\label{BZO}
For every ordered triple
\(\mathcal X\coloneqq (\bm{x},\bm{y},\bm{x}^+)\in(\mathbb R^n)^3\) and \(\gamma>0\), the
zeroth-order oracle jointly returns
\[
\mathbf{f}(\mathcal X,\gamma)
=
\bigl(f(\bm{x}),f(\bm{y}),f(\bm{x}^+)\bigr)
\sim\mathbb P_{\mathcal X,\gamma},
\]
where
\((\mathcal X,\gamma)\mapsto\mathbb P_{\mathcal X,\gamma}\) is a
measurable probability kernel on \(\mathbb R^3\), and \(f(\cdot)\)
estimates \(\phi(\cdot)\). There exists \(\epsilon_{\mathrm{f}}\ge0\) such that
\begin{equation}\label{eq:ZO_bound}
D_{\mathcal X,\gamma}
\coloneqq 
\max_{\bm{\chi}_1,\bm{\chi}_2\in\{\bm{x},\bm{y},\bm{x}^+\}}
\left|
\bigl(f(\bm{\chi}_1)-\phi(\bm{\chi}_1)\bigr)
-
\bigl(f(\bm{\chi}_2)-\phi(\bm{\chi}_2)\bigr)
\right|
\le
\gamma\epsilon_{\mathrm{f}}
\quad
\mathbb P_{\mathcal X,\gamma}\text{-a.s.}
\end{equation}
\end{assumption}

\begin{remark}[Errors in function-value differences]
\label{rem:BZO_interpretation}
\Cref{BZO} controls errors in pairwise function-value differences
and allows dependence among the three outputs. 
For example, suppose \(f(\bm\chi)=\phi(\bm\chi)+\xi+e(\bm\chi)\),
with the same \(\xi\) at \(\bm x\), \(\bm y\), and \(\bm x^+\).
The common term cancels in the differences, so
\eqref{eq:ZO_bound} holds whenever
\(\max_{\bm\chi\in\{\bm x,\bm y,\bm x^+\}}|e(\bm\chi)|
\leq\gamma\epsilon_{\mathrm f}/2\). Similar discussion can be found in \cite{Xie2024,scheinberg2025stochasticadaptiveoptimizationunreliable}.
The joint evaluation counts as one \textsc{bZO} call. 
\end{remark}

The parameters in the oracle names, such as
$(\epsilon_{\mathrm g},\upsilon,q)$ for \textsc{sFO}, specify
the accuracy guarantees and remain fixed across calls.
The query points $\bm y$ or $\mathcal X$, together with the
trial step size $\gamma$, are supplied at each call.
We omit the intrinsic parameters from the notation when
the oracle setting is clear from the context.
The algorithm takes available upper bounds
$\widetilde\epsilon_{\mathrm g}\ge\epsilon_{\mathrm g}$ and
$\widetilde\epsilon_{\mathrm f}\ge\epsilon_{\mathrm f}$ as inputs,
while the intrinsic parameter values themselves enter only
the analysis.
\Cref{subsection_3.1} specifies the conditional laws of the
oracle outputs along the adaptive run and defines the stopping
times used to measure trial complexity.
With persistent errors, the admissible target accuracies depend
on both the oracle and search parameters.

\subsection{Contributions}

\begin{itemize}[leftmargin=0.4cm]

\item
We introduce \texttt{SAAS} with separate search and momentum scales, recovering classical Nesterov momentum and allowing $0<\theta_{\mathrm m}<\theta<1/2$. An optional model-consistency test compares the gradient model with function values at the current and extrapolated points; failure triggers rejection and a smaller trial step size, thereby reducing the next extrapolation.

\item
We establish high-probability trial-complexity bounds using a Lyapunov recursion that controls oracle errors through the decrease proportional to $\theta-\theta_{\mathrm m}$. The analysis covers finite-moment gradient errors allowing bias and infinite variance, or bounded corruption, together with step-size-scaled bounds on function-value difference errors. For strongly convex objectives, the bounds retain the $\sqrt{L/\mu}$ dependence under explicit oracle, search, and accuracy conditions. For general convex objectives, finite-moment errors and a bounded auxiliary trajectory yield an $O((\epsilon_\phi-\epsilon_{\mathrm{res}})^{-1/2})$ trial bound at fixed confidence for an objective-gap or gradient-norm target, where $\epsilon_{\mathrm{res}}$ is an error-dependent residual.

\item
We compare SAAS with accelerated and nonaccelerated methods on quadratics, least-squares and real-data ridge-logistic objectives, and neural-network training. Ablations vary $(\theta,\theta_{\mathrm m})$ relative to the classical pair $(0.5,0.5)$ and examine the gradient-norm safeguard, showing how noise and minibatch size affect the momentum allocation favored by late performance.
\end{itemize}

\section{Algorithm: Safeguarded Accelerated Adaptive Search}\label{section_algo}
\Cref{algo} uses either \(\textbf{sFO(}\epsilon_{\mathrm g},\upsilon,q\textbf{)}\) or \(\textbf{cFO(}\epsilon_{\mathrm g},\kappa_{\mathrm g},\delta_{\mathrm g},\epsilon_{\mathrm c}\textbf{)}\) as its first-order oracle, denoted uniformly by \textbf{FO}, together with \(\textbf{bZO(}\epsilon_{\mathrm f}\textbf{)}\). The method maintains an auxiliary point $\bar{\bm x}_t$ alongside the main iterate $\bm x_t$. Each trial extrapolates from $\bm x_t$ along $\bar{\bm x}_t-\bm x_{t-1}$, queries FO to form a candidate, and uses bZO to test acceptance.  \Cref{para,Momentum_para}~specify the extrapolation and auxiliary-update coefficients.

\begin{algorithm}[t]
\small
\caption{\texttt{SAAS}: Safeguarded Accelerated Adaptive Search}
\label{algo}
\begin{algorithmic}[1]
    \State \textbf{Input:} Initial point \(\bm{x}_1\); oracles
\textbf{\textbf{FO}} and \textbf{\textbf{bZO}} with available bounds
\(\widetilde\epsilon_{\mathrm{g}}\ge\epsilon_{\mathrm{g}}\) and
\(\widetilde\epsilon_{\mathrm{f}}\ge\epsilon_{\mathrm{f}}\); problem parameter \(\mu\ge0\);
search parameters \(\theta\in(0,1)\), \(\epsilon_{\mathrm{rel}}\ge0\),
\(\nu_{\mathrm{inc}}>1\), \(\nu_{\mathrm{dec}}\in(0,1)\), and
\(0<\gamma_0\le\gamma_{\max}\); the momentum rule in
\Cref{para,Momentum_para}; the active acceptance conditions and their relaxation rules \(\tau_{\mathrm{des}}^{(t)}\),
        \(\tau_{\mathrm{mod}}^{(t)}\).
    \State \textbf{Initialize:}
\(\bm{x}_0\gets \bm{x}_1\), \(\bar {\bm{x}}_1\gets \bm{x}_1\), and
\(\hat\gamma_1\gets\gamma_0/\nu_{\mathrm{inc}}\).

    \vspace{0.1cm}
    \For{$t=1,2,\ldots$}
        \vspace{0.05cm}
        \State \textbf{Phase 1: Momentum and Prediction}
        \State Compute \(\hat\rho_t\) according to
        \eqref{eq:momentum_parameterization}.
        \State \refstepcounter{equation}\label{update_y_t}Compute the trial extrapolation point
        \(\;\hat {\bm{y}}_t=\bm{x}_t+\hat\rho_t(\bar {\bm{x}}_t-\bm{x}_{t-1})\).
        \hfill \textup{(\theequation)}

        \State \textbf{Phase 2: Oracle Interaction}
        \State Query \textbf{\textbf{FO}} at
        \((\hat {\bm{y}}_t,\hat\gamma_t)\) to obtain
        \(\mathbf{G}_t=\mathbf{G}(\hat {\bm{y}}_t,\hat\gamma_t)\), and compute
        \(\hat {\bm{x}}_{t+1}=\hat {\bm{y}}_t-\hat\gamma_t\mathbf{G}_t\).
        \State Form
        \(\mathcal X_t\gets(\bm{x}_t,\hat {\bm{y}}_t,\hat {\bm{x}}_{t+1})\).
        \State Query \textbf{\textbf{bZO}} at
        \((\mathcal X_t,\hat\gamma_t)\) to obtain
        \(\mathbf{f}_t=
        \bigl(f_t(\bm{x}_t),f_t(\hat {\bm{y}}_t),f_t(\hat {\bm{x}}_{t+1})\bigr)\).
        \vspace{0.1cm}
        \State \textbf{Phase 3: Acceptance and Step-Size Update}
        \State Compute \(\tau_{\mathrm{des}}^{(t)}\) and, when enabled,
        \(\tau_{\mathrm{mod}}^{(t)}\).
        \State \textbf{Acceptance Test:} Test condition~\eqref{Check_1}
        and, when enabled, condition~\eqref{Check_2}:
        \Statex \hspace{\algorithmicindent}%
        \begin{minipage}{0.94\linewidth}
        \begin{align}
            \tag{\textbf{I}}
            f_t(\hat {\bm{x}}_{t+1})
            &\le
            f_t(\hat {\bm{y}}_t)
            -\hat\gamma_t\theta\|\mathbf{G}_t\|^2
            +\tau_{\mathrm{des}}^{(t)},
            \label{Check_1}
            \\[1mm]
            \tag{\textbf{II}}
            \textnormal{Optional:}\qquad
            f_t(\hat {\bm{y}}_t)
            &\le
            f_t(\bm{x}_t)
            +\langle\mathbf{G}_t,\hat {\bm{y}}_t-\bm{x}_t\rangle
            +\tau_{\mathrm{mod}}^{(t)}.
            \label{Check_2}
        \end{align}
        \end{minipage}

        \If{the active condition(s) hold}
        \hfill {\footnotesize $\triangleright$ Accepted Trial}
            \State \textbf{Accept Trial:}
            \(\bm{x}_{t+1}\gets\hat {\bm{x}}_{t+1}\),
            \(\gamma_t\gets\hat\gamma_t\); compute \(\gamma_t'\) according to
            \eqref{eq:momentum_parameterization}, and set
            \(\bar {\bm{x}}_{t+1}\gets\hat {\bm{y}}_t-\gamma_t'\mathbf{G}_t\).
           \State {\renewcommand{\theequation}{\textbf{III}}\refstepcounter{equation}\label{Check_3}}%
\textbf{Reliability Test:}
Test the condition \(\|\mathbf{G}_t\|\ge\epsilon_{\mathrm{rel}}\).
\hfill\textup{\textbf{(III)}}%
\hspace*{\dimexpr0.06\linewidth-\algorithmicindent\relax}

            \State \textbf{Update Trial Step-Size:} \[
            \hat\gamma_{t+1}
            \gets
            \left\{\begin{array}{@{}lll@{}}
            \min\{\nu_{\mathrm{inc}}\hat\gamma_t,\gamma_{\max}\},
            &\textnormal{if condition~\eqref{Check_3} holds},
            &\text{\footnotesize $\triangleright$ Reliable Trial}\\[1mm]
            \nu_{\mathrm{dec}}\hat\gamma_t,
            &\textnormal{otherwise}.
            &\text{\footnotesize $\triangleright$ Unreliable Trial}
            \end{array}\right.
            \]
        \Else
        \hfill {\footnotesize $\triangleright$ Rejected Trial}
            \State \textbf{Retain State:}
            \(\bar {\bm{x}}_{t+1}\gets \bm{x}_t+(\bar {\bm{x}}_t-\bm{x}_{t-1})\),
            \(\bm{x}_{t+1}\gets \bm{x}_t\), and
            \(\gamma_t\gets\gamma_{t-1}\).
            \State \textbf{Decrease Trial Step-Size:}
            \(\hat\gamma_{t+1}\gets\nu_{\mathrm{dec}}\hat\gamma_t\).
        \EndIf
    \EndFor
\end{algorithmic}
\noindent\textit{Notes.}  \(f_t(\bm{x}_t)\) is used only when condition~\eqref{Check_2} is enabled.
On accepted trials, the analysis writes $\bm{y}_t$ and $\rho_t$ for the extrapolated point and coefficient. 
\Cref{ass:mu_pos_sfo_bzo,ass:mu_pos_cfo_bzo,ass:mu_zero_sfo_bzo} specify which acceptance conditions are used and how their relaxation terms are chosen.
\end{algorithm}

\subsection{Algorithmic conventions}
\label{Notation_2.1}

\paragraph{Trial conventions.}
The algorithm is indexed by trials, each of which incurs oracle
calls. Hats mark trial quantities; on acceptance, their unhatted
counterparts denote the accepted values. An accepted trial is
called \emph{reliable} if it also satisfies the gradient-norm
condition~\eqref{Check_3}. Let \(\Theta_t\) denote the acceptance indicator
and \(\Lambda_t\) the indicator of reliable trials.
On an accepted trial, \(\Theta_t=1\), and we write
\[
(\bm{y}_t,\rho_t,\bm{x}_{t+1},\gamma_t)
\coloneqq 
(\hat {\bm{y}}_t,\hat\rho_t,\hat {\bm{x}}_{t+1},\hat\gamma_t),
\qquad
\bar {\bm{x}}_{t+1}\coloneqq \bm{y}_t-\gamma_t'\mathbf{G}_t.
\]
On a rejected trial, \(\Theta_t=0\), and the main iterate,
auxiliary displacement, and retained step size satisfy
\[
\bigl(\bm{x}_{t+1},\bar {\bm{x}}_{t+1}-\bm{x}_t,\gamma_t\bigr)
=
\bigl(\bm{x}_t,\bar {\bm{x}}_t-\bm{x}_{t-1},\gamma_{t-1}\bigr).
\] The next trial therefore uses the same iterate and momentum direction, with the extrapolation recomputed at a smaller step size. The retained and trial step sizes satisfy
\begin{equation}\label{eq:gamma-dyn}
\gamma_t=
\begin{cases}
\hat\gamma_t,&\Theta_t=1,\\
\gamma_{t-1},&\Theta_t=0,
\end{cases}
\qquad \text{and}\qquad
\hat\gamma_{t+1}=
\begin{cases}
\min\{\nu_{\mathrm{inc}}\hat\gamma_t,\gamma_{\max}\},
&\Lambda_t=1,\\
\nu_{\mathrm{dec}}\hat\gamma_t,
&\Lambda_t=0.
\end{cases}
\end{equation}
Together with the initialization, this rule gives $0<\hat\gamma_t/\gamma_{t-1}\leq\nu_{\mathrm{inc}}$ at every trial, even after repeated rejections.

\paragraph{Conditions and oracle information.}
The first-order output $\bf{G}_t$ forms the candidate $\hat{\bm{x}}_{t+1}$ and the auxiliary update on acceptance.
The joint zeroth-order call supplies the function values for the tests. 

Condition~\eqref{Check_1} compares the observed decrease from $\hat{\bm{y}}_t$ to $\hat{\bm{x}}_{t+1}$ with the $\theta$-scaled predicted reduction $\theta\hat\gamma_t\|\mathbf{G}_t\|^2$, allowing the relaxation $\tau_{\mathrm{des}}^{(t)}$.
Condition~\eqref{Check_2} tests a linear inequality between $\bm{x}_t$ and $\hat{\bm{y}}_t$, using $\mathbf{G}_t$ in place of the exact gradient at $\hat{\bm{y}}_t$. With exact information,
this inequality follows from convexity. Its implementation requires
\(f_t(\bm{x}_t)\), which is otherwise unused.
The relaxation terms in \eqref{Check_1} and \eqref{Check_2} account for gradient-model and function-value errors; \Cref{SECTION_3} chooses them to ensure small-step acceptance. The strongly convex analysis uses \eqref{Check_1} alone, while the general convex analysis allows either version.

When the gradient estimate is small, the relaxation may allow
the Armijo test to pass with little predicted decrease.
Condition~\eqref{Check_3} therefore regulates step size 
after acceptance. If \(\|\mathbf G_t\|<\epsilon_{\mathrm{rel}}\),
the next trial step size is reduced; otherwise, it may increase.
The accepted point and auxiliary update are retained in both cases.

\paragraph{Auxiliary sequence and momentum direction.}
At trial \(t\), the auxiliary displacement
\(\bar {\bm{x}}_t-\bm{x}_{t-1}\) determines the direction of the extrapolation
\(\hat {\bm{y}}_t=\bm{x}_t+\hat\rho_t(\bar {\bm{x}}_t-\bm{x}_{t-1})\).
The initialization \(\bm{x}_0=\bm{x}_1=\bar {\bm{x}}_1\) gives
\(\hat {\bm{y}}_1=\bm{x}_1\).
On an accepted trial,
\[
\bar {\bm{x}}_{t+1}-\bm{x}_t
=
(\bm{x}_{t+1}-\bm{x}_t)+(\gamma_t-\gamma_t')\mathbf{G}_t,
\qquad
\hat {\bm{y}}_{t+1}-\bm{x}_{t+1}
=
\hat\rho_{t+1}(\bar {\bm{x}}_{t+1}-\bm{x}_t).
\]
Thus \(\gamma_t'\) changes the displacement carried into the next
trial, while \(\hat\rho_{t+1}\) scales that displacement.
When \(\gamma_t'=\gamma_t\), the auxiliary displacement equals
the usual Nesterov displacement \(\bm{x}_{t+1}-\bm{x}_t\).
Otherwise, the correction
\((\gamma_t-\gamma_t')\mathbf{G}_t\) can substantially change its
direction. For example, on an accepted trial with \(\bm{y}_t=\bm{x}_t\),
\(\mathbf{G}_t\ne0\), and \(\gamma_t'<0\), the auxiliary displacement
\(-\gamma_t'\mathbf{G}_t\) is opposite to the usual Nesterov
displacement \(-\gamma_t\mathbf{G}_t\).

\subsection{Stochastic Process and Filtration}
\label{subsection_3.1}
Let \((\Omega,\mathcal F,\mathbb P)\) be a probability space supporting
the oracle outputs in \Cref{algo}. The initialization is deterministic, and $\mathcal F_0=\{\varnothing,\Omega\}$. At trial \(t\), write
$\mathbf{G}_t\coloneqq \mathbf{G}(\hat {\bm{y}}_t,\hat\gamma_t)$,
$\mathbf{f}_t\coloneqq \mathbf{f}(\mathcal X_t,\hat\gamma_t)$, and denote the gradient-error norm and pairwise function-value error from
\Cref{FO,BZO} by
$W_t\coloneqq \|\mathbf{G}_t-\nabla\phi(\hat {\bm{y}}_t)\|$ and 
$D_t\coloneqq D_{\mathcal X_t,\hat\gamma_t}$.
The zeroth-order call follows the first-order call and depends on its output because the triple $\mathcal X_t$ includes the candidate $\hat{\bm{x}}_{t+1}=\hat{\bm{y}}_t-\hat\gamma_t\mathbf{G}_t$. We record these two stages with the half-step filtration
\begin{equation}\label{eq:trial_filtration}
\mathcal F_{t-\frac12}
\coloneqq 
\sigma(\mathcal F_{t-1},\mathbf{G}_t),
\qquad
\mathcal F_t
\coloneqq 
\sigma(\mathcal F_{t-\frac12},\mathbf{f}_t),
\qquad t\ge1.
\end{equation}
Before the first-order call, the trial state, $\hat{\bm{y}}_t$, $\hat\gamma_t$, and $\tau_{\mathrm{des}}^{(t)}$ are $\mathcal F_{t-1}$-measurable. The first-order call determines $W_t$, $\mathcal X_t$, condition~\eqref{Check_3}, and, when used, $\tau_{\mathrm{mod}}^{(t)}$; these quantities are $\mathcal F_{t-1/2}$-measurable. The zeroth-order call then determines $D_t$, the active acceptance tests, and the resulting indicators and updates.

We require the conditional laws of the calls to agree with the kernels in 
\Cref{FO,BZO}.  For every Borel set $\mathcal{A}\subseteq\mathbb R^n$ and $\mathcal{B}\subseteq\mathbb R^3$, almost surely,
$\mathbb P(\mathbf{G}_t\in \mathcal{A}\mid\mathcal F_{t-1})
=\mathbb P_{\hat {\bm{y}}_t,\hat\gamma_t}(\mathcal{A})$ and  $\mathbb P(\mathbf{f}_t\in \mathcal{B}\mid\mathcal F_{t-\frac12})
=\mathbb P_{\mathcal X_t,\hat\gamma_t}(\mathcal{B})$. 
For any integrable trial quantity \(Z_t\), we use the simplified notation
\[
\mathbb P_t(\,\cdot\,)
\coloneqq 
\mathbb P(\,\cdot\mid\mathcal F_{t-1}),
\qquad
\mathbb E_t[Z_t]
\coloneqq 
\mathbb E[Z_t\mid\mathcal F_{t-1}],
\qquad
\widetilde Z_t
\coloneqq 
Z_t-\mathbb E_t[Z_t].
\]
The oracle bounds then hold conditionally along the adaptive
trajectory.\footnote{For \(\textbf{sFO(}\epsilon_{\mathrm g},\upsilon,q\textbf{)}\), (6) gives \(\mathbb E_t[W_t]\le\epsilon_{\mathrm g}\); (7) or (8) gives, respectively, \(\mathbb E_t[|\widetilde W_t|^q]\le\upsilon\) or \(\mathbb E_t[\exp\{|\widetilde W_t|^q\}]\le e^\upsilon\). For \(\textbf{cFO(}\epsilon_{\mathrm g},\kappa_{\mathrm g},\delta_{\mathrm g},\epsilon_{\mathrm c}\textbf{)}\), (9) gives \(\mathbb P_t\{W_t\le\epsilon_{\mathrm g}+\kappa_{\mathrm g}\hat\gamma_t\|\mathbf G_t\|\}\ge1-\delta_{\mathrm g}\) and \(W_t\le\epsilon_{\mathrm g}+\epsilon_{\mathrm c}\) almost surely. The oracle \(\textbf{bZO(}\epsilon_{\mathrm f}\textbf{)}\) gives \(D_t\le\hat\gamma_t\epsilon_{\mathrm f}\) almost surely.} 

We measure complexity by the first trial whose extrapolated point or candidate meets the prescribed objective-gap or gradient-norm target. The following stopping time records this event.
\begin{definition}[\(\boldsymbol{\varepsilon}\)-stopping time]
\label{Stopping_time}
For an  accuracy target
\(\boldsymbol{\varepsilon}
=(\varepsilon_\phi,\varepsilon_\nabla)\), where
\(\varepsilon_\phi>0\) and \(\varepsilon_\nabla\ge0\), define
\begin{equation}\label{eq_dfn_stopping_time}
T_{\boldsymbol{\varepsilon}}
\coloneqq 
\inf\left\{
t\ge1:
\min\{\phi(\hat {\bm{x}}_{t+1}),\phi(\hat {\bm{y}}_t)\}-\phi^\star
\le\varepsilon_\phi
\ \text{or}\
\|\nabla\phi(\hat {\bm{y}}_t)\|\le\varepsilon_\nabla
\right\}.
\end{equation}
We use the convention $\inf\varnothing=+\infty$. The definition includes both the extrapolated point and the candidate, whether or not the trial is accepted. The event that it holds at trial $t$ is $\mathcal F_{t-1/2}$-measurable, so $T_{\boldsymbol{\varepsilon}}$ is a stopping time with respect to $\{\mathcal F_t\}_{t\geq0}$. When $\varepsilon_\nabla=0$, we abbreviate $\varepsilon_\phi$ by $\varepsilon$ and write
\[
T_\varepsilon
\coloneqq 
\inf\left\{
t\ge1:
\min\{\phi(\hat {\bm{x}}_{t+1}),\phi(\hat {\bm{y}}_t)\}-\phi^\star
\le\varepsilon
\right\}.
\]
\end{definition}

\subsection{Counting reliable trials}
\label{subsec:acceptance}
For the complexity analysis, we count reliable trials whose
accepted step sizes are bounded below. The following indicators
relate this count to gradient accuracy and the step-size updates.

\begin{definition}[Trial indicators]
\label{Key_indicator}
For each trial \(t\), define the following indicators.

\begin{itemize}[
  leftmargin=0.4cm,
  itemsep=3pt plus 1pt minus 1pt,
  topsep=3pt plus 1pt minus 1pt
]

\item \emph{\textbf{Accepted} trial.}
\[
\Theta_t
\coloneqq 
\mathbf1\!\left\{
\textnormal{the active acceptance condition(s) hold at trial \(t\)}
\right\}
\qquad\textnormal{is }\mathcal F_t\textnormal{-measurable}.
\]

\item \emph{\textbf{Reliable} trial.}
\[
\Lambda_t
\coloneqq 
\Theta_t\mathbf1\!\left\{
\epsilon_{\mathrm{rel}}\le\|\mathbf{G}_t\|
\right\},\qquad\Lambda_t\textnormal{ is }\mathcal F_t\textnormal{-measurable}.\]

\item \emph{\textbf{True} trial.}
\[
I_t
\coloneqq 
\begin{cases}
\mathbf1\{W_t\le\widetilde\epsilon_{\mathrm{g}}\},
&\textnormal{under \textbf{\textsc{sFO}}},\\[1mm]
\mathbf1\!\left\{
W_t\le
\epsilon_{\mathrm{g}}+\kappa_{\mathrm{g}}\hat\gamma_t\|\mathbf{G}_t\|
\right\},
&\textnormal{under \textbf{\textsc{cFO}}},
\end{cases}
\qquad
I_t\textnormal{ is }\mathcal F_{t-\frac12}\textnormal{-measurable}.
\]

\item \emph{\textbf{Large} trial.}
\[
U_t
\coloneqq 
\mathbf1\!\left\{
\max\{\hat\gamma_t,\hat\gamma_{t+1}\}>\bar\gamma
\right\}
\qquad\textnormal{is }\mathcal F_t\textnormal{-measurable},
\]
where \(\bar\gamma>0\) is the small-step threshold specified in the
corresponding analysis.

\item \emph{\textbf{Good} trial.}
Trial \(t\) is called good when it is both reliable and large, that
is, when \(\Lambda_tU_t=1\), and
\[
N_{\mathrm{good}}(t)
\coloneqq 
\sum_{i=1}^t\Lambda_iU_i
\]
denotes the number of good steps through trial \(t\).

\end{itemize}
\end{definition}

The indicator \(I_t\) records first-order accuracy and is generally not observable. For \(\textbf{sFO(}\epsilon_{\mathrm g},\upsilon,q\textbf{)}\), choose the algorithmic threshold \(\widetilde\epsilon_{\mathrm g}>\epsilon_{\mathrm g}\). Conditional Markov bounds from \eqref{eq:SFO_moment}, \eqref{eq:SFO_subweibull}, and the probability condition~\eqref{eq:CFO} for \(\textbf{cFO(}\epsilon_{\mathrm g},\kappa_{\mathrm g},\delta_{\mathrm g},\epsilon_{\mathrm c}\textbf{)}\) give
\begin{equation}\label{eq:true_probability}
\mathbb P_t(I_t=1)\ge p
\coloneqq 
\begin{cases}
\displaystyle
1-
\frac{\upsilon}
{(\widetilde\epsilon_{\mathrm{g}}-\epsilon_{\mathrm{g}})^q},
&\textnormal{under finite-moment \textbf{\textsc{sFO}}},\\[3mm]
\displaystyle
1-
\exp\!\left\{
\upsilon-
(\widetilde\epsilon_{\mathrm{g}}-\epsilon_{\mathrm{g}})^q
\right\},
&\textnormal{under Sub-Weibull \textbf{\textsc{sFO}}},\\[3mm]
1-\delta_{\mathrm{g}},
&\textnormal{under \textbf{\textsc{cFO}}}.
\end{cases}
\end{equation}
 True realizations  must occur often enough to compensate for the step-size decreases on other trials. The relevant probability threshold depends on the increase and decrease factors \((\nu_{\mathrm{inc}},\nu_{\mathrm{dec}})\):
\begin{equation}\label{equ:p_0}
p_0
\coloneqq 
{\ln(\nu_{\mathrm{dec}}^{-1})}/
{\ln(\nu_{\mathrm{inc}}\nu_{\mathrm{dec}}^{-1})}
\in(0,1).
\end{equation}
In each setting below, we choose parameters so that the lower bound $p$ in \eqref{eq:true_probability} exceeds $p_0$.\footnote{Indeed,
$p>p_0~
\Longleftrightarrow~
p\ln(\nu_{\mathrm{inc}})
+(1-p)\ln(\nu_{\mathrm{dec}})>0$. In the symmetric case
\(\nu_{\mathrm{inc}}\nu_{\mathrm{dec}}=1\), one has \(p_0=1/2\), recovering the
threshold used in \cite{Cartis2018,Albert2021,Xie2024,stoFISTA}; the asymmetric form
\eqref{equ:p_0} follows the extension in \cite{scheinberg2025stochasticadaptiveoptimizationunreliable}.}

We cap the small-step threshold $\bar\gamma$ at the initial trial step size $\gamma_0/\nu_{\mathrm{inc}}$. On a good step, $\Lambda_t=1$ implies acceptance, so $\gamma_t=\hat\gamma_t$ and $\hat\gamma_{t+1}\leq\nu_{\mathrm{inc}}\hat\gamma_t$. The large-step condition therefore gives
\begin{equation}\label{eq:good_step_size_lower}
\{\Lambda_tU_t=1\}
\qquad\text{implies}\qquad
\{\gamma_t>{\bar\gamma}/{\nu_{\mathrm{inc}}}\}.
\end{equation}

For each oracle setting, the algorithmic parameters  and target-accuracy
conditions will ensure that

\begin{equation}\label{two_proposition}
\boxed{~
\begin{aligned}
\mathbb P_t(I_t=1)\ge p>p_0,\quad
t\ge1;\qquad
(1-U_t)I_t\le\Lambda_t,\quad
1\le t<T_{\boldsymbol\varepsilon},
\end{aligned}\quad\textnormal{a.s.}~}
\end{equation}
The second relation states that every true trial with \(U_t=0\)
before \(T_{\boldsymbol\varepsilon}\) is reliable.

For simplicity, assume $-\ln(\nu_{\mathrm{inc}})/\ln(\nu_{\mathrm{dec}})\in\mathbb N^+$.\footnote{Here $\mathbb N^+=\{1,2,\ldots\}$ and $[a]_+:=\max\{a,0\}$. The integrality assumption removes floor and ceiling terms from the counting argument.} For \(\hat p\in(p_0,p)\), the bound below compares
\(N_{\mathrm{good}}(t)\) with the deterministic function
\(\bar N_{\mathrm{good}}(t)\), defined together with the initial
step-size adjustment \(N_0\) by
\begin{equation}\label{eq:N_good_envelope}
\textstyle
\bar N_{\mathrm{good}}(t)
\coloneqq 
\left[
\frac{p_0}{1-p_0}
(\hat p-p_0)t-N_0
\right]_+,
\qquad
N_0
\coloneqq 
p_0
\left\lceil
{
\ln(\hat\gamma_1/\bar\gamma)
}/{
\ln(\nu_{\mathrm{dec}}^{-1})
}
\right\rceil .
\end{equation}
The following lemma applies the counting argument of
Lemmas~2--3 in
\cite{scheinberg2025stochasticadaptiveoptimizationunreliable}, with reliable acceptance playing the role of a successful trial. 
Under \eqref{eq:gamma-dyn}, both rejection and unreliable acceptance decrease the next trial step size.
\begin{lemma}[Good-step growth]
\label{lem:good_step_growth}
Suppose \eqref{two_proposition} holds. Then, for every \(t\ge1\) and
every \(\hat p\in(p_0,p)\),
\begingroup\postdisplaypenalty=10000
\begin{equation}\label{eq:good_step_count}
\textstyle
\mathbb P\left\{
t<T_{\boldsymbol\varepsilon},
N_{\mathrm{good}}(t)<\bar N_{\mathrm{good}}(t)
\right\}
\le
\exp\left(
-\frac{(p-\hat p)^2}{2p}t
\right).
\end{equation}
\par\nobreak\hfill\textnormal{\small\itshape(\hyperref[proof:good_step_growth]{Proof in the appendix})}\par
\endgroup
\end{lemma}

\subsection{Parameterized Momentum and Lyapunov Function}
\label{subsection_momentum strategy}
A standard analysis of Nesterov acceleration combines the objective gap with a weighted squared distance from an auxiliary point to a minimizer~
\cite{Nestrov2014book,localLip}. 
This sum serves as a Lyapunov function, and the momentum coefficients are chosen so that the decrease from the gradient step contracts it. We adapt this construction to  \texttt{SAAS}. Rejection leaves the Lyapunov value unchanged; on acceptance, the recursion has a contraction factor $1-\alpha_t$, with $\alpha_t\in(0,1)$, together with terms accounting for the oracle errors.

The Armijo test~\eqref{Check_1} contains the decrease term
\(-\theta\gamma_t\|\mathbf G_t\|^2\). We allocate \(\theta_{\mathrm m}\gamma_t\|\mathbf G_t\|^2\)
to the distance identity,
leaving
\(-(\theta-\theta_{\mathrm m})\gamma_t\|\mathbf G_t\|^2\)
in the Lyapunov recursion to control the error terms.  The extrapolation coefficient \(\hat\rho_t\) and auxiliary-update
coefficient \(\gamma_t'\) are chosen jointly to obtain this identity.

\begin{definition}
\label{para}Let \(\{\hat\gamma_t\}_{t\ge1}\) and
\(\{\gamma_t\}_{t\ge0}\) be the trial and retained step-size
sequences generated by \Cref{algo}. Both sequences take values
in \((0,\gamma_{\max}]\).
Fix
\((\theta_{\mathrm{m}},\vartheta_{\mathrm{m}})
\in(0,\theta\,]\times[0,1)\), 
\(
\alpha_0\in\left(
(1-\vartheta_{\mathrm{m}})
\sqrt{2\theta_{\mathrm{m}}\mu\gamma_0},
\,1
\right)
\).
For every trial \(t\ge1\), let \(\hat\alpha_t\) be the unique root in
\((0,1)\) of
\begin{align}\label{para_alpha}
\frac{\hat\alpha_t^2}{\hat\gamma_t}
=
(1-\hat\alpha_t)
\frac{\alpha_{t-1}^2}{\gamma_{t-1}}
+
2\theta_{\mathrm{m}}(1-\vartheta_{\mathrm{m}})^2
\mu\hat\alpha_t,
\qquad \text{and let}\qquad
\alpha_t\coloneqq 
\begin{cases}
\alpha_{t-1}, & \Theta_t=0,\\
\hat\alpha_t, & \Theta_t=1.
\end{cases}
\end{align}
When \(\mu>0\), additionally assume that
$2\theta_{\mathrm{m}}(1-\vartheta_{\mathrm{m}})^2
\mu\gamma_{\max}<1$.
\end{definition}
Equation~\eqref{para_alpha} is a scalar quadratic, so its positive root
is available in closed form.

\begin{definition}[Momentum parameterization]
\label{Momentum_para}
Under the setting of \Cref{para}, for each trial \(t\ge1\), set
$\hat\beta_t
\coloneqq 
2(1-\vartheta_{\mathrm{m}})^2\theta_{\mathrm{m}}\mu
\hat\gamma_t\hat\alpha_t^{-1}$, and define
\begin{equation}\label{eq:momentum_parameterization}
\hat\rho_t
\coloneqq \frac{1-\alpha_{t-1}}{\alpha_{t-1}}
\frac{
\hat\alpha_t(1-\hat\beta_t)(1-\vartheta_{\mathrm{m}})
}{
(1-\hat\alpha_t)(1-\vartheta_{\mathrm{m}})
+\hat\alpha_t(1-\hat\beta_t)
},\qquad
\gamma_t'
\coloneqq 
\frac{\gamma_t}{1-\alpha_t}
\left(
2\theta_{\mathrm{m}}
-\frac{\alpha_t}{1-\vartheta_{\mathrm{m}}}
\right).
\end{equation}
On an accepted trial, we also write
\(\beta_t\coloneqq \hat\beta_t\) and \(\rho_t\coloneqq \hat\rho_t\).
\end{definition}

The parameter restrictions give \(\hat\beta_t\in[0,1)\).
The trial coefficients \(\hat\alpha_t\), \(\hat\beta_t\),
and \(\hat\rho_t\) are \(\mathcal F_{t-1}\)-measurable.
The accepted quantities \(\alpha_t\) and \(\gamma_t'\)
are \(\mathcal F_t\)-measurable.

\begin{proposition}
\label{prop:gamma_prime_equal_gamma}
Suppose that the momentum quantities are generated by
\Cref{Momentum_para}. If \(\theta_{\mathrm{m}}<1/2\), then
\(\gamma_i'\ne\gamma_i\) on every accepted trial \(i\).
More generally, if two accepted trials \(s\ne t\) satisfy
$\gamma_s'=\gamma_s$,
$\gamma_t'=\gamma_t$,
$\alpha_s\ne\alpha_t$,
then
$(\theta_{\mathrm{m}},\vartheta_{\mathrm{m}})=(1/2,0)$.
Conversely, this parameter choice gives
\(\gamma_i'=\gamma_i\) on every accepted trial \(i\).
\end{proposition}
\begin{proof}
For any accepted trial \(i\), the definition of \(\gamma_i'\) in
\eqref{eq:momentum_parameterization} gives
\(\gamma_i'=\gamma_i\) if and only if
\(2\theta_{\mathrm{m}}-1
=\alpha_i\vartheta_{\mathrm{m}}/(1-\vartheta_{\mathrm{m}})\).
The right-hand side is nonnegative, so equality is impossible when
\(\theta_{\mathrm{m}}<1/2\). If it holds at accepted trials \(s\) and \(t\),
subtraction gives
\((\alpha_s-\alpha_t)\vartheta_{\mathrm{m}}/(1-\vartheta_{\mathrm{m}})=0\).
Thus \(\alpha_s\ne\alpha_t\) implies
\(\vartheta_{\mathrm{m}}=0\) and then \(\theta_{\mathrm{m}}=1/2\).
The converse follows directly from
\eqref{eq:momentum_parameterization}.
\end{proof}

\begin{remark}[Relation to adaptive Nesterov momentum]
\label{rem:classical_momentum_specialization}
Setting $(\theta_{\mathrm m},\vartheta_{\mathrm m})=(1/2,0)$ recovers the adaptive Nesterov momentum of
\cite{localLip}, and setting $\mu=0$ further gives the momentum used in \texttt{FISTA-SS}
\cite{stoFISTA}. The parameterization also retains a useful feature of adaptive Nesterov backtracking: reducing the trial step size reduces the extrapolation
\cite[Appendix~B.3.1, Algorithm~31]{d2021acceleration}. During consecutive rejections, the retained parameters stay fixed and $\hat\rho_t=O(\sqrt{\hat\gamma_t})$ as $\hat\gamma_t\downarrow0$. Hence $\hat{\bm{y}}_t-\bm{x}_t=\hat\rho_t(\bar{\bm{x}}_t-\bm{x}_{t-1})$ vanishes, bringing the query point back toward the current iterate.
\end{remark}

To analyze the momentum update, we rescale the displacement $\bar{\bm{x}}_t-\bm{x}_{t-1}$ into an auxiliary point $\bm{z}_t$. The rescaling is chosen so that rejection leaves $\bm{z}_t$ unchanged and the trial extrapolation is a convex combination of $\bm{x}_t$ and $\bm{z}_t$. We then use the distance from $\bm{z}_t$ to a minimizer in the Lyapunov function.

\begin{definition}[Lyapunov function]
\label{def:Lyapunov_struct}
For the Lyapunov analysis, define, for each \(t\ge1\),
\begin{equation}\label{z_t}
\bm{z}_t
\coloneqq 
\bm{x}_t+
(1-\vartheta_{\mathrm{m}})(1-\alpha_{t-1})\alpha_{t-1}^{-1}
(\bar {\bm{x}}_t-\bm{x}_{t-1}).
\end{equation}
The associated Lyapunov function at \((\bm{x}_t,\bm{z}_t)\) is
\begin{equation}\label{Energy}
\Phi_{t}
\coloneqq 
\phi(\bm{x}_t)-\phi^\star
+
\frac{\alpha_{t-1}^2}
{4\theta_{\mathrm{m}}(1-\vartheta_{\mathrm{m}})^2\gamma_{t-1}}
\left\|\bm{z}_t-\bm{x}^\star\right\|^2,
\end{equation}
where \(\bm{x}^\star\) is an arbitrary minimizer of \(\phi\). We use the convention \(\Phi_0\coloneqq \Phi_1\) when needed.
\end{definition}
Substituting \eqref{z_t}  into the extrapolation update \eqref{update_y_t} and using \eqref{eq:momentum_parameterization} gives, for every trial,
\begin{equation}\label{eq:yhat_zt}
\hat {\bm{y}}_t
=
\bm{x}_t+
\frac{\hat\alpha_t(1-\hat\beta_t)}
{(1-\vartheta_{\mathrm{m}})(1-\hat\alpha_t)
+\hat\alpha_t(1-\hat\beta_t)}
(\bm{z}_t-\bm{x}_t)
\in\operatorname{conv}\{\bm{x}_t,\bm{z}_t\}.
\end{equation}

\begin{lemma}[One-step Lyapunov recursion]
\label{energy_1}
Let the momentum quantities in \Cref{algo} satisfy
\Cref{Momentum_para}, and let \(\Phi_t\) be defined by
\Cref{def:Lyapunov_struct}. If condition~\eqref{Check_2} is enabled,
then, for every $t\ge1$ and
$(\theta_{\mathrm{m}},\vartheta_{\mathrm{m}})
\in(0,\theta]\times[0,1)$, 
\begin{equation}\label{energy_Ite}
\left\{
\begin{aligned}
\Phi_{t+1}
&=\Phi_t,
&&\text{if }\Theta_t=0,\\
\Phi_{t+1}
&\le
(1-\alpha_t)\Phi_t-C_t+E_t,
&&\text{if }\Theta_t=1.
\end{aligned}
\right.
\end{equation}
On an accepted trial, the nonnegative term \(C_t\) and the
error term \(E_t\) are defined by
\begin{align}
 {C}_t
&\coloneqq 
(\theta-\theta_{\mathrm{m}})\cdot
\gamma_t\|\mathbf{G}_t\|^2
+
\alpha_t\cdot\left(
\phi(\bm{y}_t)-\phi^\star
+\frac{\mu}{2}\|\bm{y}_t-\bm{x}^\star\|^2
\right) {\vartheta_{\mathrm{m}}}/
{(1-\vartheta_{\mathrm{m}})}
\ge0,
\label{T_c}\\[2mm]
 {E}_t
&\coloneqq 
D_t+\tau_{\mathrm{des}}^{(t)}
+(1-\alpha_t)\tau_{\mathrm{mod}}^{(t)}
+{\alpha_t}\cdot
\left\langle
\bm{y}_t-\bm{x}^\star,\nabla\phi(\bm{y}_t)-\mathbf{G}_t
\right\rangle/{(1-\vartheta_{\mathrm{m}})} .
\label{T_e}
\end{align}
\end{lemma}

\begin{proof}

Write  \(Q_t\coloneqq \Phi_t-[\phi(\bm{x}_t)-\phi^\star]\) for the distance term in \eqref{Energy}.  
If \(\Theta_t=0\), the rejected-trial update, \eqref{z_t}, and
\eqref{para_alpha} give
\(\bm{x}_{t+1}=\bm{x}_t\), \(\bm{z}_{t+1}=\bm{z}_t\),
\(\gamma_t=\gamma_{t-1}\), and \(\alpha_t=\alpha_{t-1}\).
It follows that $Q_{t+1}=Q_t$ and $\Phi_{t+1}=\Phi_t$.

Now consider an accepted trial and omit hats. Apply
condition~\eqref{Check_1}, split $f_t(\bm{y}_t)$ with weights $\alpha_t$ and $1-\alpha_t$, and apply condition \eqref{Check_2} to the second part. 
The function-value error is a convex combination of two pairwise differences of \(f_t-\phi\), each bounded above by \(D_t\) by its definition in~\eqref{eq:ZO_bound}. We obtain
\[
\begin{aligned}
\phi(\bm{x}_{t+1})-\phi^\star
&\le
(1-\alpha_t)\bigl(\phi(\bm{x}_t)-\phi^\star\bigr)
+\alpha_t\bigl(\phi(\bm{y}_t)-\phi^\star\bigr)\\
&\quad
+(1-\alpha_t)\langle\mathbf{G}_t,\bm{y}_t-\bm{x}_t\rangle
-\theta\gamma_t\|\mathbf{G}_t\|^2
+D_t+\tau_{\mathrm{des}}^{(t)}
+(1-\alpha_t)\tau_{\mathrm{mod}}^{(t)}.
\end{aligned}
\]
To bound the part involving $\phi(\bm{y}_t)-\phi^\star$, write $\mathbf{G}_t=\nabla\phi(\bm{y}_t)-[\nabla\phi(\bm{y}_t)-\mathbf{G}_t]$ and use strong convexity at $\bm{y}_t$:
\begin{equation}\label{eq:energy_comparison_minimizer}
\begin{aligned}
&\alpha_t\bigl(\phi(\bm{y}_t)-\phi^\star\bigr)
-\frac{\alpha_t}{1-\vartheta_{\mathrm{m}}}
\langle\mathbf{G}_t,\bm{y}_t-\bm{x}^\star\rangle\\
&\quad\le
-\frac{\mu\alpha_t}{2}\|\bm{y}_t-\bm{x}^\star\|^2
-\frac{\vartheta_{\mathrm{m}}\alpha_t}{1-\vartheta_{\mathrm{m}}}
\left(
\phi(\bm{y}_t)-\phi^\star
+\frac{\mu}{2}\|\bm{y}_t-\bm{x}^\star\|^2
\right)
+\frac{\alpha_t}{1-\vartheta_{\mathrm{m}}}
\langle \bm{y}_t-\bm{x}^\star,\nabla\phi(\bm{y}_t)-\mathbf{G}_t\rangle.
\end{aligned}
\end{equation}
Substituting this inequality into the preceding bound and using the definitions of $C_t$ and $E_t$ gives
\begin{align}
&\phi(\bm{x}_{t+1})-\phi^\star
-(1-\alpha_t)\langle\mathbf{G}_t,\bm{y}_t-\bm{x}_t\rangle
-\frac{\alpha_t}{1-\vartheta_{\mathrm{m}}}
\langle\mathbf{G}_t,\bm{y}_t-\bm{x}^\star\rangle
+\theta_{\mathrm{m}}\gamma_t\|\mathbf{G}_t\|^2
\notag\\
&\qquad\le
(1-\alpha_t)\bigl(\phi(\bm{x}_t)-\phi^\star\bigr)
-\frac{\mu\alpha_t}{2}\|\bm{y}_t-\bm{x}^\star\|^2
-C_t+E_t.
\label{eq:checked_energy_part}
\end{align}
Under the accepted updates, the momentum parameters \eqref{para_alpha}, \eqref{eq:momentum_parameterization}, and \eqref{z_t}
 imply the following identity:
\begin{align}
Q_{t+1}
&=
(1-\alpha_t)Q_t
+\frac{\mu\alpha_t}{2}\|\bm{y}_t-\bm{x}^\star\|^2
-(1-\alpha_t)\langle\mathbf{G}_t,\bm{y}_t-\bm{x}_t\rangle
\notag\\
&\quad
-\frac{\alpha_t}{1-\vartheta_{\mathrm{m}}}
\langle\mathbf{G}_t,\bm{y}_t-\bm{x}^\star\rangle
+\theta_{\mathrm{m}}\gamma_t\|\mathbf{G}_t\|^2
-\Delta_t^{\mathrm{quad}},
\label{eq:accepted_distance_energy_identity}
\end{align}
where
\[
\Delta_t^{\mathrm{quad}}\coloneqq 
\frac{\mu(1-\vartheta_{\mathrm{m}})^2(1-\alpha_t)^2}
{2\alpha_t(1-\beta_t)}
\|\bm{y}_t-\bm{x}_t\|^2\ge0.
\]
The verification of \eqref{eq:accepted_distance_energy_identity} is given in the  \hyperref[proof:energy_algebra]{appendix}. It follows from the auxiliary-point update and the squared-norm identity for a convex combination.

Combining the distance identity with the function-value estimate cancels the gradient terms. Since $\Delta_t^{\mathrm{quad}}\geq0$ and $\Phi_t=\phi(\bm{x}_t)-\phi^\star+Q_t$, this gives
\begin{equation}\label{eq:energy_quadratic_closure}
\Phi_{t+1}
\le
(1-\alpha_t)\Phi_t-C_t+E_t-\Delta_t^{\mathrm{quad}}
\le
(1-\alpha_t)\Phi_t-C_t+E_t.
\end{equation}
Together with the rejected-trial identity, this proves~\eqref{energy_Ite}.
\end{proof}

\begin{corollary}[One-step recursion without condition~\eqref{Check_2}]
\label{cor:energy_without_check_ii} 
Under the problem assumptions and momentum parameterization of
\Cref{energy_1}, suppose every accepted trial satisfies
condition~\eqref{Check_1}, whether or not condition~\eqref{Check_2}
is additionally imposed. A rejected trial
satisfies \(\Phi_{t+1}=\Phi_t\), whereas every accepted trial satisfies
\begin{align}
\Phi_{t+1}
&\le
(1-\alpha_t)\Phi_t-C_t
+\tau_{\mathrm{des}}^{(t)}+D_t-\frac{\mu(1-\alpha_t)}{2}\|\bm{y}_t-\bm{x}_t\|^2
\notag\\[-1mm]
&\quad
+\big\langle
\nabla\phi(\bm{y}_t)-\mathbf{G}_t,
(1-\alpha_t)(\bm{y}_t-\bm{x}_t)
+{\alpha_t}(\bm{y}_t-\bm{x}^\star)/({1-\vartheta_{\mathrm{m}}})
\big\rangle.
\label{eq:energy_without_check_ii}
\end{align}
\end{corollary}
\begin{proof}Rejection leaves $\Phi_t$ unchanged by the same argument as in the proof of~\Cref{energy_1}. On an accepted trial, omit hats. Without using condition~\eqref{Check_2}, condition~\eqref{Check_1} and the definition of \(D_t\) in~\eqref{eq:ZO_bound} give
\[
\phi(\bm{x}_{t+1})-\phi^\star
\le
\phi(\bm{y}_t)-\phi^\star
-\theta\gamma_t\|\mathbf{G}_t\|^2
+D_t+\tau_{\mathrm{des}}^{(t)}.
\]
We now use strong convexity directly, with comparison point $\bm{x}_t$:
\[
\begin{aligned}
\phi(\bm{y}_t)-\phi^\star
\le\phi(\bm{x}_t)-\phi^\star+\langle\mathbf{G}_t,\bm{y}_t-\bm{x}_t\rangle
-\frac\mu2\|\bm{y}_t-\bm{x}_t\|^2+\langle\nabla\phi(\bm{y}_t)-\mathbf{G}_t,\bm{y}_t-\bm{x}_t\rangle.
\end{aligned}
\]
As in \Cref{energy_1}, split $\phi(\bm{y}_t)-\phi^\star$ with weights $1-\alpha_t$ and $\alpha_t$. Apply the preceding estimate to the first part and \eqref{eq:energy_comparison_minimizer} to the second. Adding the distance identity \eqref{eq:accepted_distance_energy_identity} cancels the two terms that are linear in $\mathbf{G}_t$. The result is the right-hand side of \eqref{eq:energy_without_check_ii} minus $\Delta_t^{\mathrm{quad}}$. Dropping this nonpositive term proves the claim.
\end{proof}

Both recursions have the multiplicative factor $1-\alpha_t$ on an accepted trial, while rejection leaves the Lyapunov value unchanged. To measure the contraction accumulated over accepted trials, define
\begin{equation}\label{dfn_lambda_t}
 \lambda_t
\coloneqq 
\textstyle\prod_{i=1}^t(1-\Theta_i\alpha_i).
\end{equation}
Every good trial is accepted with $\gamma_t>\bar\gamma/\nu_{\mathrm{inc}}$. The momentum recurrence \eqref{para_alpha} and the trial count in \Cref{lem:good_step_growth} then yield the following bounds on $\lambda_t$.

\begin{lemma}[Good-step contraction for \(\lambda_t\)]
\label{lem:good_step_contraction}
Suppose the conditions of \Cref{lem:good_step_growth} hold, and let
\(\bar N_{\mathrm{good}}(t)\) and \(\lambda_t\) be defined by
\eqref{eq:N_good_envelope} and \eqref{dfn_lambda_t}, respectively. If
\(\mu=0\),
\begin{equation}\label{pro_lambda_t_rate_mu=0}
\mathbb P\left\{
t<T_{\boldsymbol\varepsilon},\
\lambda_t>
4\left(
2+
\alpha_0
\sqrt{{\bar\gamma}/(
{\nu_{\mathrm{inc}}\gamma_0)}}\,
\bar N_{\mathrm{good}}(t)
\right)^{-2}
\right\}
\,\le\,
\exp\left(
-\tfrac{(p-\hat p)^2}{2p}t
\right),
\end{equation}
whereas, if \(\mu>0\),
\begingroup\postdisplaypenalty=10000
\begin{equation}\label{pro_lambda_t_rate_mu>0}
\mathbb P\left\{
t<T_{\boldsymbol\varepsilon},\
\lambda_t>
\left(
1-(1-\vartheta_{\mathrm{m}})
\sqrt{
{2\theta_{\mathrm{m}}\mu\bar\gamma}/
{\nu_{\mathrm{inc}}}
}
\right)^{\bar N_{\mathrm{good}}(t)}
\right\}
\,\le\,
\exp\left(
-\tfrac{(p-\hat p)^2}{2p}t
\right).
\end{equation}
\par\nobreak\hfill\textnormal{\small\itshape(\hyperref[proof:good_step_contraction]{Proof in the appendix})}\par
\endgroup
\end{lemma}

At the scale $\bar\gamma=\Theta(1/L)$, with $\theta_{\mathrm m}$ and $1-\vartheta_{\mathrm m}$ bounded away from zero, these estimates recover the deterministic accelerated orders: quadratic decay for $\mu=0$ and a factor $1-\Theta(\sqrt{\mu/L})$ per good step for $\mu>0$. The complexity analysis combines this contraction with bounds on the accumulated oracle errors.
\section{Complexity analysis}\label{SECTION_3}
\subsection{A common stopping-time bound}
\label{subsec:frame}
Our analysis gives a recursion of the form
\(\Phi_{t+1}\leq V_t\Phi_t+R_t\), with \(V_t\geq0\)
and possibly signed \(R_t\).
Iterating it gives a product multiplying the initial value
and a weighted sum of residuals. The following theorem converts
probability bounds on these two quantities into a bound on
the stopping time.

\begin{theorem}[A high-probability stopping-time bound]
\label{Big_threorem}
Let
\(\boldsymbol\varepsilon=(\varepsilon_\phi,\varepsilon_\nabla)\), where
\(\varepsilon_\phi>0\) and \(\varepsilon_\nabla\ge0\), and let
\(T_{\boldsymbol\varepsilon}\) be the stopping time in
\Cref{Stopping_time}.  Suppose that \(\Phi_t\ge0\) is adapted and
compatible with the stopping rule in the sense that
$\Phi_t\le\varepsilon_\phi
\Longrightarrow
T_{\boldsymbol\varepsilon}\le t$.
Assume that \(\Phi_1\le\Phi_0\) and that there exist adapted sequences
\(V_t\ge0\) and \(R_t\in\mathbb R\) satisfying
\begin{equation}\label{eq:abstract_one_step}
\Phi_{t+1}\le V_t\Phi_t+R_t,
\qquad 1\le t<T_{\boldsymbol\varepsilon}.
\end{equation}
Suppose further that there exist
\(H\colon \mathbb N^+\to\mathbb R_+\),
\(P\colon \mathbb N^+\to[0,1]\),
\(\varepsilon_{\mathrm{res}}\ge0\), and
\(\bar T\in\mathbb N^+\) such that, for every integer
\(t\ge\bar T\),
\begin{equation}\label{eq:abstract_tail_event}
\begin{aligned}
\mathbb P\Bigl(&
\Bigl\{
 t<T_{\boldsymbol\varepsilon},\
 \prod_{i=1}^tV_i>H(t)
\Bigr\}\cup\,
\Bigl\{
 t<T_{\boldsymbol\varepsilon},\
 \sum_{i=1}^t
 \bigl(\prod_{j=i+1}^tV_j\bigr)R_i
 >\varepsilon_{\mathrm{res}}
\Bigr\}
\Bigr)
\le P(t).
\end{aligned}
\end{equation}
Assume that \(H\) is non-increasing on
\(\{t\in\mathbb N^+:t\ge\bar T\}\) and that
\(\varepsilon_\phi>\varepsilon_{\mathrm{res}}\). Define its generalized
inverse by\footnote{We use the conventions that an empty product equals
one and \(\inf\varnothing=+\infty\).}
\begin{equation}\label{eq:abstract_inverse_H}
H^{-1}(u)
\coloneqq 
\inf\left\{
s\in\mathbb N^+:
s\ge\bar T,\ H(s)\le u
\right\},
\qquad u>0.
\end{equation}
For \(\Phi_0>0\),\footnote{If
\(\Phi_0=0\), then \(\Phi_1=0\) and the conclusion follows immediately
from compatibility.} every integer \(t\) satisfying
\begin{equation}\label{eq:abstract_complexity_inverse}
t\ge\max\left\{
\bar T,\,
H^{-1}\!\left(
\frac{\varepsilon_\phi-\varepsilon_{\mathrm{res}}}{\Phi_0}
\right)
\right\}
\end{equation}
guarantees
\(\mathbb P\{T_{\boldsymbol\varepsilon}\leq t+1\}\geq1-P(t)\).

\end{theorem}
\begin{proof}
\phantomsection\label{proof_bigtheorem}
On $\{t<T_{\boldsymbol\varepsilon}\}$, iterate \eqref{eq:abstract_one_step}, using $V_i\geq0$ and $\Phi_1\leq\Phi_0$, to obtain
\begin{equation}\label{eq:abstract_lyapunov_iteration}
\textstyle\Phi_{t+1}
\le
\Phi_0\prod_{i=1}^tV_i
+
\sum_{i=1}^t
\left(\prod_{j=i+1}^tV_j\right)R_i.
\end{equation}
If \(t\) satisfies \eqref{eq:abstract_complexity_inverse}, the
definition of \(H^{-1}\) and the monotonicity of \(H\) give
\(H(t)\Phi_0+\varepsilon_{\mathrm{res}}\leq\varepsilon_\phi\).
On \(\{t<T_{\boldsymbol\varepsilon}\}\) outside the event in
\eqref{eq:abstract_tail_event},
\eqref{eq:abstract_lyapunov_iteration} therefore gives
\(\Phi_{t+1}\leq\varepsilon_\phi\), and compatibility yields
\(T_{\boldsymbol\varepsilon}\leq t+1\).
The same conclusion holds on \(\{T_{\boldsymbol\varepsilon}\leq t\}\).
Hence \(\mathbb P\{T_{\boldsymbol\varepsilon}\leq t+1\}\geq1-P(t)\).
\end{proof}

\phantomsection\label{rem:framework_instantiation}
The bounds count trials. Obtaining sample complexity also requires
the number of samples used by each oracle call, which may depend
on the trial step size through the accuracy requirement
in \textsc{bZO} and \textsc{cFO}~\cite{Jin-sample}.
In each application below, we use the Lyapunov function from 
\Cref{def:Lyapunov_struct}. It is compatible with the stopping rule because $\Phi_s\geq\phi(\bm{x}_s)-\phi^\star$ and $\bm{x}_s$ is either the initial point or a previously accepted candidate. The initial point itself is tested through $\hat{\bm{y}}_1=\bm{x}_1$. The remaining quantities \(V_t,R_t,H(t),P(t),\varepsilon_{\mathrm{res}}\) in \Cref{Big_threorem} are specified separately for each oracle setting.

\subsection{Strongly Convex Objectives}\label{sub_complexity_>0}
We first assume that $\phi$ satisfies \Cref{assump:SC-Lsmooth,assump:global-min}  with strong-convexity modulus at least $\mu>0$, and that this value of $\mu$ is available to \Cref{algo}. Acceptance in this subsection uses condition~\eqref{Check_1} alone; condition~\eqref{Check_2} and $\tau_{\mathrm{mod}}^{(t)}$ are omitted. In addition to the requirements in \Cref{para,Momentum_para}, choose \footnote{By
\Cref{para_lemma_d}, these conditions imply
\(\alpha_t/(1-\alpha_t)\le
(1+\nu_{\mathrm{inc}})/(1-\theta_{\mathrm{m}})\) for every \(t\ge1\).}
\[
\gamma_{\max}={1}/{(2\mu)},
\qquad
\alpha_0\le
{(1+\nu_{\mathrm{inc}})}/{(2+\nu_{\mathrm{inc}}-\theta_{\mathrm{m}})}.
\]
We use the objective-gap stopping time in \Cref{Stopping_time} by setting $\varepsilon_\nabla=0$ and writing $\varepsilon:=\varepsilon_\phi$. We begin with \textsc{sFO} and then treat \textsc{cFO}, in both cases with \textsc{bZO}. The following \textsc{sFO} choices set the Armijo descent relaxation $\tau_{\rm des}^{(t)}$ and the norm threshold $\epsilon_{\rm rel}$ in terms of the available error bounds.

\begin{assumption}[Parameters for the \textsc{sFO} analysis]
\label{ass:mu_pos_sfo_bzo}
Suppose \Cref{algo} uses the finite-moment-tail \(\textbf{sFO(}\epsilon_{\mathrm g},\upsilon,q\textbf{)}\) specified in \eqref{eq:SFO_bias} and \eqref{eq:SFO_moment}, together with \(\textbf{bZO(}\epsilon_{\mathrm f}\textbf{)}\), and let \(\mu>0\). Let
\(\widetilde\epsilon_{\mathrm{g}}>\epsilon_{\mathrm{g}}\) and
\(\widetilde\epsilon_{\mathrm{f}}\ge\epsilon_{\mathrm{f}}\) be available bounds. Choose
\(\vartheta_{\mathrm{m}}=0\), \(0<\theta_{\mathrm{m}}<\theta<1\), and
\begin{equation}\label{eq:tolerances_mu_sfo}
\epsilon_{\mathrm{rel}}^2>
\frac{\widetilde\epsilon_{\mathrm{g}}^2/[2(1-\theta)]+2\widetilde\epsilon_{\mathrm{f}}}
{\theta-\theta_{\mathrm{m}}},\qquad
\widetilde\epsilon_{\mathrm{g}}>
\epsilon_{\mathrm{g}}+\left(\frac{\upsilon}{1-p_0}\right)^{1/q},\qquad
\tau_{\mathrm{des}}^{(t)}\coloneqq \hat\gamma_t
\left(\frac{\widetilde\epsilon_{\mathrm{g}}^2}{2(1-\theta)}
+\widetilde\epsilon_{\mathrm{f}}\right).
\end{equation}
\end{assumption}

\begin{theorem}[Strongly convex \textsc{sFO}]
\label{thm:mu_pos_sfo_bzo}
Suppose \Cref{ass:mu_pos_sfo_bzo} holds with \(1<q\le2\). Set
\[
p\coloneqq 1-\frac{\upsilon}{(\widetilde\epsilon_{\mathrm{g}}-\epsilon_{\mathrm{g}})^q}>p_0,
\qquad
\bar\gamma\coloneqq\min\left\{\frac{1-\theta}{L},\frac{\gamma_0}{\nu_{\mathrm{inc}}}\right\},
\qquad
C_\alpha\coloneqq\sqrt{\frac{2\theta_{\mathrm{m}}\mu\bar\gamma}{\nu_{\mathrm{inc}}}}.
\]
Choose \(\hat p\in(p_0,p)\), and use \(N_0\) and
\(\bar N_{\mathrm{good}}(t)\) from \eqref{eq:N_good_envelope}.
Let the target satisfy
\begin{equation}\label{eq:precision_mu_sfo}
\varepsilon>\tfrac{1}{\mu}
\max\{2\epsilon_{\mathrm{rel}}^2,
C_{\mathrm{s}}^2\widetilde\epsilon_{\mathrm{g}}^2\}.
\end{equation}
Here \(C_{\mathrm{s}}\) depends only on
\((\theta,\theta_{\mathrm{m}},\nu_{\mathrm{inc}},\nu_{\mathrm{dec}})\)
and is specified in \Cref{tab:complexity_constants}.

\mbox{If \(\Phi_0\le\varepsilon\), then \(T_\varepsilon=1\).} Otherwise, every integer \(t\) satisfying
\begingroup\postdisplaypenalty=10000
\begin{equation}\label{eq:complexity_mu_sfo}
t\ge
\frac{1-p_0}{p_0(\hat p-p_0)}
\left[N_0+\frac{1+2\ln(\Phi_0/\varepsilon)}{C_\alpha}\right]
\end{equation}
guarantees
\begin{equation}\label{eq:probability_mu_sfo}
\mathbb P\{T_\varepsilon\le t+1\}
\ge1-\exp\!\left(-\frac{(p-\hat p)^2}{2p}t\right)
-\frac{C}{[C_\alpha\bar N_{\mathrm{good}}(t)]^{q-1}}.
\end{equation}
Here \(C>0\) depends only on \((q,p_0)\); its exact value is given in
\Cref{tab:sfo_tail_constants}.
\par\nobreak\hfill\textnormal{\small\itshape(\hyperref[proof:mu_pos_sfo_bzo]{Proof in the appendix})}\par
\endgroup
\end{theorem}

\begin{corollary}[\textsc{sFO} tail refinements]
\label{cor:mu_pos_sfo_tail_refinements}
Retain the mean-error condition \eqref{eq:SFO_bias}, the oracle \(\textbf{bZO(}\epsilon_{\mathrm f}\textbf{)}\), and the algorithmic and accuracy-target conditions of \Cref{thm:mu_pos_sfo_bzo}, except for the tail-dependent calibration specified below.
 For each case, use the
specified \(p\), choose \(\hat p\in(p_0,p)\), and define
\(\bar N_{\mathrm{good}}(t)\) accordingly. For \(\Phi_0>\varepsilon\),
the corresponding bound holds for every integer \(t\) satisfying
\eqref{eq:complexity_mu_sfo}; for \(\Phi_0\le\varepsilon\), \(T_\varepsilon=1\).
\begin{enumerate}[label=\textnormal{(\roman*)},leftmargin=*]
\item\label{cor:sfo_higher_moment_reduction}
\emph{Higher finite moments.} For the finite-moment-tail \(\textbf{sFO(}\epsilon_{\mathrm g},\upsilon,q\textbf{)}\) with \(q>2\), retain the calibration in~\eqref{eq:tolerances_mu_sfo}
and set \(p\coloneqq 1-\upsilon/(\widetilde\epsilon_{\mathrm{g}}-\epsilon_{\mathrm{g}})^q\). Then
\begin{equation}\label{eq:probability_mu_sfo_high_moment}
\mathbb P\{T_\varepsilon\le t+1\}
\ge1-\exp\!\left(-\frac{(p-\hat p)^2}{2p}t\right)
-2\exp\!\left(-cC_\alpha\bar N_{\mathrm{good}}(t)\right)
-\frac{C}{[C_\alpha\bar N_{\mathrm{good}}(t)]^{q-1}}.
\end{equation}
\item
\emph{Sub-Weibull tails.}
For the Sub-Weibull-tail alternative \eqref{eq:SFO_subweibull} of \(\textbf{sFO(}\epsilon_{\mathrm g},\upsilon,q\textbf{)}\), replace the calibration of \(\widetilde\epsilon_{\mathrm g}\) in 
\eqref{eq:tolerances_mu_sfo} by
\(\widetilde\epsilon_{\mathrm{g}}>\epsilon_{\mathrm{g}}+
[\ln(e^\upsilon/(1-p_0))]^{1/q}\), and set
\(p\coloneqq 1-\exp\{\upsilon-(\widetilde\epsilon_{\mathrm{g}}-\epsilon_{\mathrm{g}})^q\}\).
Then
\begin{equation}\label{eq:probability_mu_sfo_subweibull}
\mathbb P\{T_\varepsilon\le t+1\}
\ge1-\exp\!\left(-\frac{(p-\hat p)^2}{2p}t\right)
-\exp\!\left(-cC_\alpha\bar N_{\mathrm{good}}(t)\right).
\end{equation}
\end{enumerate}
In \textnormal{(i)}, \(c,C>0\) depend only on \((q,p_0)\);
in \textnormal{(ii)}, \(c>0\) depends only on \((q,\upsilon,p_0)\).
Their exact values are given in \Cref{tab:sfo_tail_constants}.
\par\nobreak\hfill\textnormal{\small\itshape(\hyperref[proof:mu_pos_sfo_tail_refinements]{Proof in the appendix})}\par
\end{corollary}

\begin{remark}[Accuracy and confidence with persistent errors]
\label{rem:admissible_p_mu_sfo}
With the search and momentum parameters fixed, taking
\(\epsilon_{\mathrm{rel}}^2\) to be a fixed multiple of its
lower bound in \eqref{eq:tolerances_mu_sfo} gives a sufficient
target-accuracy threshold of order
\((\widetilde\epsilon_{\mathrm g}^{\,2}
+\widetilde\epsilon_{\mathrm f})/\mu\).
Increasing \(\widetilde\epsilon_{\mathrm g}\) raises the
probability lower bound \(p\) and also enlarges this threshold.
Once these inputs are fixed, higher confidence is obtained
by running more trials.

To see the cost of confidence, fix the search, momentum, and tail parameters, keep $p-p_0$ bounded away from zero, and take $\hat p=(p+p_0)/2$. For an admissible target accuracy $\varepsilon<\Phi_0$ and $0<\delta<1/2$, \Cref{thm:mu_pos_sfo_bzo,cor:mu_pos_sfo_tail_refinements} give success probability at least $1-\delta$ within
\[
\underbrace{
\mathcal O\!\left(
\frac{\ln(\Phi_0/\varepsilon)+\delta^{-1/(q-1)}}{C_\alpha}
\right)
}_{\text{finite moments, }q>1},
\qquad
\underbrace{
\mathcal O\!\left(
\frac{\ln(\Phi_0/\varepsilon)+\ln(1/\delta)}{C_\alpha}
\right)
}_{\text{Sub-Weibull}}
\] trials.
When $\bar\gamma=\Theta(1/L)$, both bounds have the accelerated prefactor $C_\alpha^{-1}=\Theta(\sqrt{L/\mu})$. Their difference lies in the confidence term: polynomial under finite moments and logarithmic under Sub-Weibull tails. The constants absorb the initial step-size adjustment and the probability margin used to convert good steps into total trials \cite{Cartis2018,Albert2021}. The same bounds give almost-sure finite hitting times, with finite expectation for $q>2$ or Sub-Weibull errors.

With the other parameters fixed, $C_\alpha$ grows as $\sqrt{\theta_{\mathrm m}}$. The lower bound on $\epsilon_{\mathrm{rel}}^2$ in (43) diverges as $\theta_{\mathrm m}\uparrow\theta$, whereas $C_s=\Theta(\theta_{\mathrm m}^{-1})$ as $\theta_{\mathrm m}\downarrow0$. For $\widetilde\epsilon_{\mathrm g}>0$, the sufficient accuracy conditions \eqref{eq:tolerances_mu_sfo} and \eqref{eq:precision_mu_sfo} therefore require increasingly large targets at both endpoints.
\end{remark}

We next turn to bounded corruption. The \textsc{cFO} model provides a relative-accuracy estimate on true trials and an almost-sure error cap on false trials. We use the first to retain contraction and the second to bound the loss from corrupted accepted steps.

\begin{assumption}[Parameters for the \textsc{cFO} analysis]
\label{ass:mu_pos_cfo_bzo}
Suppose \Cref{algo} uses \(\textbf{cFO(}\epsilon_{\mathrm g},\kappa_{\mathrm g},\delta_{\mathrm g},\epsilon_{\mathrm c}\textbf{)}\) from \Cref{FO} together with \(\textbf{bZO(}\epsilon_{\mathrm f}\textbf{)}\), and let \(\mu>0\).
Let \(\widetilde\epsilon_{\mathrm{g}}\ge\epsilon_{\mathrm{g}}\) and
\(\widetilde\epsilon_{\mathrm{f}}\ge\epsilon_{\mathrm{f}}\) be available upper bounds on the
corresponding oracle error levels.
Choose \(0<\theta_{\mathrm{m}}<\theta<1\), set
\(\vartheta_{\mathrm{m}}\coloneqq \widetilde\epsilon_{\mathrm{g}}/
(\widetilde\epsilon_{\mathrm{g}}+2\epsilon_{\mathrm{rel}})\), $\tau_{\mathrm{des}}^{(t)}
\coloneqq 
\hat\gamma_t
\bigl(
\tfrac{\widetilde\epsilon_{\mathrm{g}}^2}{2(1-\theta)}
+\widetilde\epsilon_{\mathrm{f}}
\bigr)$ and require
\begin{equation}\label{eq:tolerances_mu_cfo}
\epsilon_{\mathrm{rel}}^2>
\frac{1}{\theta-\theta_{\mathrm{m}}}
\left(
\tfrac{\widetilde\epsilon_{\mathrm{g}}^2}{2(1-\theta)}
+2\widetilde\epsilon_{\mathrm{f}}
\right).
\end{equation}
Assume moreover $\delta_{\mathrm{g}}<1-p_0$ and
$\kappa_{\mathrm{g}}<\widetilde\kappa_{\mathrm{g}}$, where
\[
\widetilde\kappa_{\mathrm{g}}
\coloneqq 
\mu\sqrt{\Delta_{\mathrm{c}}}
\min\{\sqrt2,\,2(1-\vartheta_{\mathrm{m}})\},\qquad \Delta_{\mathrm{c}}
\coloneqq 
\theta-\theta_{\mathrm{m}}
-
\left(
\tfrac{\widetilde\epsilon_{\mathrm{g}}^2}{2(1-\theta)}
+2\widetilde\epsilon_{\mathrm{f}}
\right)/\epsilon_{\mathrm{rel}}^2>0.
\]
\end{assumption}

\begin{theorem}[Strongly convex \textsc{cFO}]
\label{thm:mu_pos_cfo_bzo}
Suppose \Cref{ass:mu_pos_cfo_bzo} holds. Set
\(p\coloneqq1-\delta_{\mathrm{g}}\) and
\(\bar\gamma\coloneqq\min\{(1-\theta)/(L+2\kappa_{\mathrm{g}}),
\gamma_0/\nu_{\mathrm{inc}}\}\).
With \(C_1,C_2,C_\kappa^{\mathrm{c}}\) specified in
\Cref{tab:complexity_constants}, define
\begin{equation}\label{eq:constants_mu_cfo}
r_\mu^{\mathrm{c}}\coloneqq
\exp\!\left(C_\kappa^{\mathrm{c}}\frac{p_0(p-p_0)}{1-p_0}\right)-1.
\end{equation}
Let the objective-gap target
satisfy\footnote{When \(\epsilon_{\mathrm{g}}=\vartheta_{\mathrm{m}}=0\), the first
condition in \eqref{eq:precision_mu_cfo} is void.}
\begin{equation}\label{eq:precision_mu_cfo}
\varepsilon>\max\left\{
\frac{\epsilon_{\mathrm{g}}^2}{2\mu\vartheta_{\mathrm{m}}^2},\;
\frac{\bigl(\epsilon_{\mathrm{g}}+(1+\kappa_{\mathrm{g}}\bar\gamma)
\epsilon_{\mathrm{rel}}\bigr)^2}{2\mu},\;
\left(
\frac{\delta_{\mathrm{g}}C_2+
\sqrt{\delta_{\mathrm{g}}^2C_2^2+4r_\mu^{\mathrm{c}}C_1}}
{2r_\mu^{\mathrm{c}}}
\right)^2
\right\}.
\end{equation}
Define
\begin{equation}\label{eq:pprime_mu_cfo}
R_\phi\coloneqq1+\frac{C_2\sqrt\varepsilon}{\varepsilon+C_1},
\qquad
p'(\varepsilon)\coloneqq p_0+
(1-p_0)\frac{\ln(1+C_1/\varepsilon)+(1-p_0)\ln R_\phi}
{p_0C_\kappa^{\mathrm{c}}+(1-p_0)\ln R_\phi}.
\end{equation}
\mbox{If \(\Phi_0\le\varepsilon\), then \(T_\varepsilon=1\).} Otherwise,
for any \(\hat p\in(p'(\varepsilon),p)\), every positive integer
\(t\) satisfying
\begingroup\postdisplaypenalty=10000
\begin{equation}\label{eq:complexity_mu_cfo}
t\ge\max\left\{
\frac{(1-p_0)N_0}{p_0(\hat p-p_0)},\;
\frac{1-p_0}{\hat p-p'(\varepsilon)}
\frac{\ln(\Phi_0/\varepsilon)+C_\kappa^{\mathrm{c}}N_0}
{p_0C_\kappa^{\mathrm{c}}+(1-p_0)\ln R_\phi}
\right\}
\end{equation}
guarantees
\begin{equation}\label{eq:probability_mu_cfo}
\mathbb P\{T_\varepsilon\le t+1\}
\ge
1-
\exp\left(
-\frac{(p-\hat p)^2}{2p}t
\right).
\end{equation}
\par\nobreak\hfill\textnormal{\small\itshape(\hyperref[proof:mu_pos_cfo_bzo]{Proof in the appendix})}\par
\endgroup
\end{theorem}

\begin{remark}[Progress under bounded corruption]
\label{rem:cfo_progress_margin}
The search analysis uses $p>p_0$ to establish step-size recovery.
Accounting for oracle errors in \Cref{thm:mu_pos_cfo_bzo} leads
to the stronger requirement $p>p'(\varepsilon)$, which, by
\eqref{eq:pprime_mu_cfo}, is equivalent to
\[
\ln(1+C_1/\varepsilon)+\delta_{\mathrm g}\ln R_\phi
<
C_\kappa^{\mathrm c}\frac{p_0(p-p_0)}{1-p_0}.
\]
The third restriction in \eqref{eq:precision_mu_cfo}
is sufficient for this inequality and can be written as
\[
\frac{C_1}{\varepsilon}
+\frac{\delta_{\mathrm g}C_2}{\sqrt{\varepsilon}}
<r_\mu^{\mathrm c}.
\]
For fixed search and momentum parameters,
$C_1=\mathcal{O}((\widetilde\epsilon_{\mathrm g}^{\,2}
+\widetilde\epsilon_{\mathrm f})/\mu)$ and
$C_2=\mathcal{O}(\epsilon_{\mathrm c}/\sqrt{\mu})$.
Thus the restriction on target accuracy depends on the persistent error levels
as well as the probability and magnitude of corruption.

For fixed search parameters, taking
$\hat p=(p+p'(\varepsilon))/2$ in
\Cref{thm:mu_pos_cfo_bzo} shows that \texttt{SAAS} reaches
an admissible target $\varepsilon<\Phi_0$ with probability
at least $1-\delta$, $0<\delta<1/2$, within
\[
\mathcal{O}\!\left(
\frac{N_0+(C_\kappa^{\mathrm c})^{-1}
\ln(\Phi_0/\varepsilon)}
{p-p'(\varepsilon)}
+
\frac{\ln(1/\delta)}
{[p-p'(\varepsilon)]^2}
\right)
\]
trials. At fixed confidence, this bound retains the accelerated
$\sqrt{L/\mu}$ dependence when
$C_\kappa^{\mathrm c}=\Theta(\sqrt{\mu/L})$
and $p-p'(\varepsilon)$ is bounded away from zero.
The stated scaling of $C_\kappa^{\mathrm c}$ holds at
$\bar\gamma=\Theta(1/L)$ when $\theta_{\mathrm m}$
and $1-\vartheta_{\mathrm m}$ are bounded away from zero
and $\kappa_{\mathrm g}/\widetilde\kappa_{\mathrm g}$
is bounded away from one.

Since $\widetilde\kappa_{\mathrm g}=\mathcal{O}(\mu)$,
the admissible relative-error coefficient
$\kappa_{\mathrm g}\gamma$ is $\mathcal{O}(\mu/L)$
for step sizes of order $1/L$.
The almost-sure error bound in cFO controls false accepted trials,
so the exponential probability estimate follows from concentration
of the number of true trials.
\end{remark}

\subsection{General Convex Objectives}\label{sec_mu=0_cons}
We now let $\mu=0$. The coefficient of $\|\bm z_t-\bm x^\star\|^2$ in $\Phi_t$ is then proportional to $\lambda_{t-1}$, by \Cref{para_lemma_a}. We impose a uniform bound on $\|\bm z_t-\bm x^\star\|$ to control the first-order error terms. Related boundedness assumptions appear in
\cite{Albert2021,Xie2024}.

\begin{assumption}[General convex setting]
\label{ass:mu_zero_sfo_bzo}
Suppose \Cref{algo} uses the finite-moment-tail \(\textbf{sFO(}\epsilon_{\mathrm g},\upsilon,q\textbf{)}\) together with \(\textbf{bZO(}\epsilon_{\mathrm f}\textbf{)}\), and let \(\mu=0\). Let
\(\widetilde\epsilon_{\mathrm{g}}>\epsilon_{\mathrm{g}}\) and
\(\widetilde\epsilon_{\mathrm{f}}\ge\epsilon_{\mathrm{f}}\) be available upper bounds.
Choose \(\vartheta_{\mathrm{m}}=0\) and
\(0<\theta_{\mathrm{m}}<\theta<1\), and set
\(\tau_{\mathrm{des}}^{(t)}\coloneqq \hat\gamma_t
[\widetilde\epsilon_{\mathrm{g}}^2/(2(1-\theta))+\widetilde\epsilon_{\mathrm{f}}]\), with
\begin{equation}\label{eq:tolerances_mu_zero_sfo}
\epsilon_{\mathrm{rel}}^2>
\frac{\nu_{\mathrm{inc}}-\nu_{\mathrm{dec}}}
{(1-\nu_{\mathrm{dec}})(\theta-\theta_{\mathrm{m}})}
\left(\frac{\widetilde\epsilon_{\mathrm{g}}^2}{2(1-\theta)}
+3\widetilde\epsilon_{\mathrm{f}}\right),
\qquad
\widetilde\epsilon_{\mathrm{g}}>\epsilon_{\mathrm{g}}+
\left(\frac{\upsilon}{1-p_0}\right)^{1/q}.
\end{equation}
Condition~\eqref{Check_2} may be disabled or enabled with
\begin{equation}\label{eq:model_tolerance_mu_zero_sfo}
\tau_{\mathrm{mod}}^{(t)}\coloneqq 
\widetilde\epsilon_{\mathrm{g}}\|\hat {\bm{y}}_t-\bm{x}_t\|
+\hat\gamma_t\widetilde\epsilon_{\mathrm{f}}.
\end{equation}
For some \(\bm{x}^\star\in\operatorname*{arg\,min}\phi\) and
\(\widehat\Omega_z>0\), assume
\begin{equation}\label{eq:bounded_auxiliary_mu_zero_sfo}
\sup_{t\ge1}\|\bm{z}_t-\bm{x}^\star\|\le\widehat\Omega_z
\qquad\textnormal{a.s.}
\end{equation}
\end{assumption}

\begin{theorem}[General convex complexity]
\label{thm:mu_zero_sfo_bzo}
Suppose \Cref{ass:mu_zero_sfo_bzo} holds with \(1<q\le2\).
Set \(p\coloneqq 1-\upsilon/(\widetilde\epsilon_{\mathrm{g}}-\epsilon_{\mathrm{g}})^q\)
and \(\bar\gamma\coloneqq \min\{(1-\theta)/L,
\gamma_0/\nu_{\mathrm{inc}}\}\), and choose \(\hat p\in(p_0,p)\).
Let \(B_{\hat p}\), \(C_{\hat p}\), \(\bar T_{\mathrm{gc}}\), and
\(\varepsilon_{\mathrm{res}}\) be as in
\Cref{tab:complexity_constants}, and let
\(\boldsymbol\varepsilon=(\varepsilon_\phi,\varepsilon_\nabla)\) satisfy
\begin{equation}\label{eq:precision_mu_zero_sfo}
\varepsilon_\phi>\varepsilon_{\mathrm{res}},
\qquad
\varepsilon_\nabla>\epsilon_{\mathrm{rel}}+\widetilde\epsilon_{\mathrm{g}}.
\end{equation}
Then every integer \(t\) satisfying
\begingroup\postdisplaypenalty=10000
\begin{equation}\label{eq:complexity_mu_zero_sfo}
t\ge\max\left\{\bar T_{\mathrm{gc}},\;
\sqrt{\frac{C_{\hat p}\Phi_0}{\varepsilon_\phi-\varepsilon_{\mathrm{res}}}}\right\}
\end{equation}
guarantees
\begin{align}
\mathbb P\{T_{\boldsymbol\varepsilon}\le t+1\}
\ge1-\exp\left(-\frac{(p-\hat p)^2}{2p}t\right)
-\frac{2(1-p_0)}{t^{q-1}}.
\label{eq:probability_mu_zero_sfo}
\end{align}
\par\nobreak\hfill\textnormal{\small\itshape(\hyperref[proof:mu_zero_sfo_bzo]{Proof in the appendix})}\par
\endgroup
\end{theorem}

The trajectory bound also yields a guarantee for the
objective-gap target alone.
\begin{corollary}[Objective-gap complexity]
\label{cor:mu_zero_objective_gap}
In the setting of \Cref{thm:mu_zero_sfo_bzo}, suppose additionally that
\begin{equation}\label{eq:precision_mu_zero_objective_gap}
\varepsilon_\phi>
\max\{1,\theta_{\mathrm{m}}^{-1}-1\}\,
\widehat\Omega_z(\epsilon_{\mathrm{rel}}+\widetilde\epsilon_{\mathrm{g}}).
\end{equation}
Then the same trial bound and probability guarantee hold for
\(T_{\varepsilon_\phi}\), the objective-gap stopping time in
\Cref{Stopping_time}. 
\par\nobreak\hfill\textnormal{\small\itshape(\hyperref[proof:mu_zero_objective_gap]{Proof in the appendix})}\par
\end{corollary}
\begin{remark}[Trajectory control and objective-gap accuracy]
\label{rem:general_convex_interpretation}
Fix the search and oracle parameters, the trajectory bound,
and the initialization, and take $\hat p=(p+p_0)/2$.
For $0<\delta<1/2$, \Cref{thm:mu_zero_sfo_bzo} guarantees
that the objective-gap or gradient target is reached with
probability at least $1-\delta$ within
\[
\mathcal{O}\!\left(
\sqrt{\frac{C_{\hat p}\Phi_0}
{\varepsilon_\phi-\varepsilon_{\mathrm{res}}}}
+\delta^{-1/(q-1)}
\right)
\]
trials. The first term in $\varepsilon_{\mathrm{res}}$ (\Cref{tab:complexity_constants}) is proportional to $\widehat\Omega_z\widetilde\epsilon_{\mathrm g}$, compared with $\widetilde\epsilon_{\mathrm g}^{2}/\mu$ in the strongly convex sFO bound. The second bounds the propagated relaxation terms after subtracting the decrease retained on reliable trials. The step-size update \eqref{eq:gamma-dyn} makes this bound independent of the number of trials.
The gradient target ensures that true trials pass
condition~\eqref{Check_3} before termination, so sufficiently
small true trials are reliable. Under
\eqref{eq:precision_mu_zero_objective_gap}, the trajectory bound
and convexity allow \(\varepsilon_\nabla\) to be chosen so that
meeting this target also gives an objective gap at most
\(\varepsilon_\phi\). 

For a fixed horizon $t$, the uniform trajectory assumption
can be replaced by the bound
$\max_{1\leq i\leq t+1}\|\bm{z}_i-\bm{x}^\star\|
\leq\widehat\Omega_z$
holding with high probability.
The probability that this bound fails is then added to the
failure probability in \eqref{eq:probability_mu_zero_sfo}.
\end{remark}
\begin{proposition}[A sufficient trajectory bound]
\label{rem:gc_trajectory}
The trajectory assumption \eqref{eq:bounded_auxiliary_mu_zero_sfo}
can also be ensured by a bound on accepted trials: for some
\(\widehat\Omega_y>0\),
\begin{equation}\label{eq:bounded_trajectory_mu_zero_sfo}
\sup_{\{t\ge1:\,\Theta_t=1\}}
\max\{\|\bm{y}_t-\bm{z}_t\|,\|\bm{y}_t-\bm{x}^\star\|\}
\le\widehat\Omega_y\qquad\textnormal{a.s.}
\end{equation}
Accepted trials occur infinitely often almost surely under the oracle and acceptance conditions above. Since rejection leaves $\bm{z}_t$ unchanged, a bound at the next accepted trial also bounds the current $\bm{z}_t$. Thus \eqref{eq:bounded_trajectory_mu_zero_sfo} implies
\eqref{eq:bounded_auxiliary_mu_zero_sfo} with $\widehat\Omega_z=2\widehat\Omega_y$, whether condition~\eqref{Check_2} is enabled or not.
\par\nobreak\hfill\textnormal{\small\itshape(\hyperref[proof:gc_trajectory]{Proof in the appendix})}\par
\end{proposition}

The constants in \Cref{thm:mu_pos_sfo_bzo,thm:mu_pos_cfo_bzo,thm:mu_zero_sfo_bzo}
are collected in \Cref{tab:complexity_constants}. In each case, \(\bar\gamma\)
and \(p\) have the values given in the corresponding theorem; \(N_0\) is
defined in \eqref{eq:N_good_envelope}.

\begin{table}[H]
\centering
\caption{Constants used in the regime-specific complexity bounds.}
\label{tab:complexity_constants}
\renewcommand{\arraystretch}{1.25}
\begin{tabularx}{\linewidth}
{@{}c c >{\raggedright\arraybackslash}X@{}}
\toprule
First-order oracle & Constant & Definition\\
\midrule

\multicolumn{3}{@{}l}{\textbf{Strongly convex regime} \((\mu>0)\)}
\\[1mm]
\textsc{sFO}
& \(C_{\mathrm{s}}\)
& \(\displaystyle
C_{\mathrm{s}}\coloneqq
\frac{(2+\nu_{\mathrm{inc}}-\theta_{\mathrm{m}})(\theta+3\theta_{\mathrm{m}})
\sqrt{2\nu_{\mathrm{inc}}}}
{\theta_{\mathrm{m}}(1-\theta_{\mathrm{m}})(1-\nu_{\mathrm{dec}})}
\)\\
\addlinespace[3pt]
\textsc{cFO}
&
\(C_1\)
&
\(\displaystyle
\begin{aligned}
C_1
\coloneqq &
\frac{2\widetilde\epsilon_{\mathrm{g}}^2+\widetilde\epsilon_{\mathrm{f}}}{2\mu}
+
\frac{1+\nu_{\mathrm{inc}}}{1-\theta_{\mathrm{m}}}
\frac{
\frac{\widetilde\epsilon_{\mathrm{g}}^2}{2(1-\theta)}
+2\widetilde\epsilon_{\mathrm{f}}
}
{2\mu\theta_{\mathrm{m}}(1-\vartheta_{\mathrm{m}})^2}
\end{aligned}
\)
\\
\addlinespace[3pt]
\textsc{cFO}
&
\(C_2\)
&
\(\displaystyle
C_2\coloneqq
\sqrt{\frac{2}{\mu}}\,
\frac{1+\nu_{\mathrm{inc}}}{1-\theta_{\mathrm{m}}}\,
\frac{\epsilon_{\mathrm{c}}}{1-\vartheta_{\mathrm{m}}}
\)
\\
\addlinespace[3pt]
\textsc{cFO}
&
\(C_\kappa^{\mathrm{c}}\)
&\(\displaystyle C_\kappa^{\mathrm{c}}\coloneqq
-\ln\left(
1-
\left[
1-\left(\tfrac{\kappa_{\mathrm{g}}}{\widetilde\kappa_{\mathrm{g}}}\right)^2
\right]
(1-\vartheta_{\mathrm{m}})
\sqrt{
\frac{2\theta_{\mathrm{m}}\mu\bar\gamma}{\nu_{\mathrm{inc}}}
}
\,\right)\)
\\
\addlinespace[4pt]

\multicolumn{3}{@{}l}{\textbf{General convex regime} \((\mu=0)\)}
\\[1mm]
\textsc{sFO}
&
\(\bar T_{\mathrm{gc}}\)
&
\(\displaystyle
\bar T_{\mathrm{gc}}\coloneqq \max\!\left\{
1,
\left\lceil
\frac{3N_0(1-p_0)}{p_0(\hat p-p_0)}
\right\rceil
\right\}\)
\\
\addlinespace[3pt]
\textsc{sFO}
&
\(B_{\hat p}\)
&
\(\displaystyle
B_{\hat p}
\coloneqq
C_{\hat p}\alpha_0\sqrt{\frac{\gamma_{\max}}{\gamma_0}}
\left[
\bar T_{\mathrm{gc}}^{-1}
+
\frac{1+\bar T_{\mathrm{gc}}^{-1}}{2}
\alpha_0\sqrt{\frac{\gamma_{\max}}{\gamma_0}}
\right]
\)
\\
\addlinespace[3pt]
\textsc{sFO}
&
\(C_{\hat p}\)
&
\(\displaystyle
C_{\hat p}
\coloneqq
\left[
\frac{
3(1-p_0)
}{
p_0(\hat p-p_0)\alpha_0
}
\right]^2
\frac{\nu_{\mathrm{inc}}\gamma_0}{\bar\gamma}
\)
\\
\addlinespace[3pt]
\textsc{sFO}
& \(\varepsilon_{\mathrm{res}}\)
& \(\displaystyle
\varepsilon_{\mathrm{res}}\coloneqq
\widehat\Omega_zB_{\hat p}\widetilde\epsilon_{\mathrm{g}}
+\frac{\gamma_{\max}}{1-\nu_{\mathrm{dec}}}
\left(\frac{\widetilde\epsilon_{\mathrm{g}}^2}{2(1-\theta)}
+3\widetilde\epsilon_{\mathrm{f}}\right)
\)\\
\bottomrule
\end{tabularx}
\end{table}

\clearpage
\begingroup
\raggedbottom
\captionsetup{skip=6pt}
\section{Numerical experiments}
\label{sec:experiments}

\subsection{Experimental setup}
\label{subsec:exp_protocol}
We examine how gradient-error direction, heavy tails, and bias affect
accelerated search, then vary the search and momentum scales
$(\theta,\theta_{\rm m})$ and the gradient-norm safeguard. Strongly
convex quadratics and least-squares problems provide controlled
comparisons with known optimal values and exact function evaluations.
We also train nonconvex neural networks to examine these design choices
beyond the convex setting of the theory. We use the search and momentum
updates of \Cref{algo} with empirically chosen inputs and omit
condition~\eqref{Check_2}
throughout. Within each setting, methods share the same underlying
random draws at each first-order-call index.

We evaluate performance at the main iterate \(\bm{x}_t\), reporting
the normalized objective gap
\(r_t=[\phi(\bm{x}_t)-\phi^\star]/
[\phi(\bm{x}_0)-\phi^\star]\)
for convex problems and cross-entropy on the full training set,
normalized by its initial value, for neural networks. For each run,
the \emph{full score} and \emph{late score} average the base-10 logarithm
of this quantity over the full evaluation horizon and its final quarter,
respectively; lower scores are better. A \emph{paired advantage}
subtracts the score of the method or allocation being evaluated from
that of the stated reference on a matched run. Unless stated otherwise,
curves show pointwise medians with 10th--90th percentile bands across
runs, and scores and paired advantages are summarized by their medians.
Oracle constructions, parameter selection, and cost accounting are
detailed in \Cref{app:exp_selected_details}.

\subsection{Sensitivity to gradient-error geometry}
\label{subsec:exp_noise_geometry}

We vary the error direction while keeping its norm distribution fixed
at a common query point. The objective is the strongly convex quadratic
\(\phi(\bm{x})=\tfrac12x_1^2+\tfrac L2x_2^2\), with
\(L=10^4\), \(\mu=1\), and \(\bm{x}_0=(1,1)\). The error norm is
\(0.1\|\nabla\phi(\bm{y})\|\) with probability 0.9 and
\(\sqrt{1000}\|\nabla\phi(\bm{y})\|\) otherwise, capped in both cases at
\(\epsilon_{\mathrm{c}}=\sqrt{1000}\|\nabla\phi(\bm{x}_0)\|\).
The four directions are collinear with the gradient, isotropic,
orthogonal to the gradient, and fixed. Each profile is conditionally
centered, and its law is independent of the trial step size.
\Cref{app:exp_noise_geometry} gives the construction and its
\textsc{sFO} interpretation.

We compare \texttt{SAAS} with fixed-step \texttt{GD}, \texttt{NAG},
and \texttt{AGNES}~\cite{gupta2024nesterovaccelerationdespitenoisy},
and with \texttt{clipped-SSTM}~\cite{gorbunov2020accelerated}.\footnote{We use
\texttt{clipped-SSTM} with restarts, denoted by \texttt{R-clipped-SSTM}.
The restart schedule for each study is given in the appendix.}
Each condition uses 200 paired oracle paths and \(10^5\) first-order
calls. Configurations are selected in method-specific preliminary
experiments and then held fixed across all four directions;
\Cref{app:exp_noise_geometry} records the parameters and tuning protocol.

\begin{figure}[!htbp]
\centering
\includegraphics[width=\linewidth]
{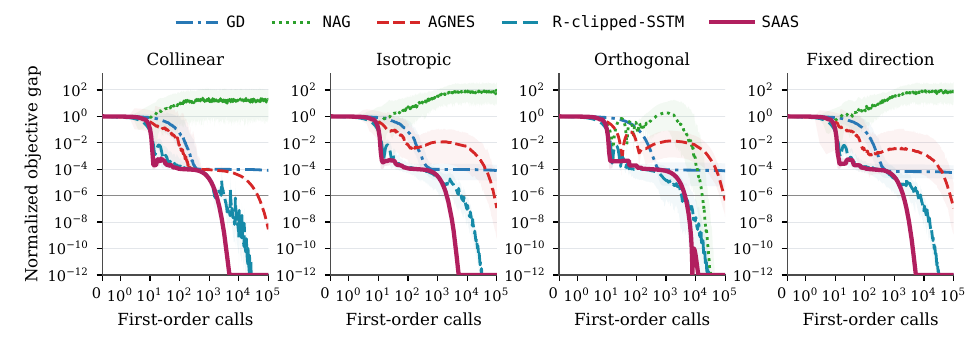}
\caption{\textbf{Error direction changes the method comparison.}
Normalized objective gaps for four centered error profiles with the
same conditional error-norm distribution at a common query point.
Curves and bands summarize 200 paired paths; the horizontal line marks
\(10^{-6}\).}
\label{fig:exp_noise_geometry}
\end{figure}

Changing the error direction reverses the ordering of \texttt{SAAS} and \texttt{R-clipped-SSTM} by calls to the target $10^{-6}$. Both median trajectories reach this target in all four profiles: \texttt{R-clipped-SSTM} requires fewer calls in the collinear and orthogonal cases, and \texttt{SAAS} in the isotropic and fixed-direction cases. \texttt{AGNES} also reaches the target in all four cases, whereas \texttt{GD} does not do so within $10^5$ calls. In three profiles, the median NAG gap exceeds its initial value. These differences occur under the same conditional error-norm law at a common query point. \Cref{tab:exp_weighted_target_costs} reports the crossing counts and weighted function-value costs.

\subsection{Power-law errors in least squares}
\label{subsec:exp_powerlaw_ls}

We next vary the tail index and add persistent bias, using
128-dimensional strongly convex least-squares objectives
\(\phi(\bm{x})=\|A\bm{x}\|^2/(2m)=\bm{x}^\top H\bm{x}/2\).
Here \(H\) is diagonal, with entries geometrically spaced between
\(1/\kappa\) and \(1\). The three final-evaluation problems use
\(\kappa\in\{300,1000,3000\}\), have minimizer
\(\bm{x}^\star=0\), and are initialized at unit objective gap.

The gradient error has independent, centered power-law coordinates
with tail index \(\widetilde q\in\{1.8,1.5,1.2\}\).
Its scale gives a median error norm of
\(0.005\|\nabla\phi(\bm{x}_0)\|\). A fourth setting adds a fixed bias
at \(\widetilde q=1.5\). These laws have infinite variance and satisfy
\eqref{eq:SFO_moment} for every \(1<q<\widetilde q\);
\Cref{app:exp_powerlaw_ls} gives their construction.
We compare \texttt{SAAS}, stochastic adaptive step search
(\texttt{SASS})~\cite{Xie2024},
\texttt{R-clipped-SSTM}, \texttt{AGNES}, and classical strongly convex
\texttt{NAG}, whose parameters are fixed at their exact-gradient values.
Here \texttt{SASS} uses the same search and gradient-norm regulation
without momentum.

The four tuned methods receive equal first-order-call budgets for
candidate evaluation on common development problems and oracle streams.
Final evaluation uses separate problems and streams, with 24 paired
paths per problem and 4,500 calls per path.
\Cref{fig:exp_powerlaw_ls} shows the first 2,000 calls in the four
noisy settings. Full trajectories, the exact-gradient reference, and
the selected parameters are given in \Cref{app:exp_powerlaw_ls}.

\begin{figure}[!htbp]
\centering
\includegraphics[width=\linewidth]
{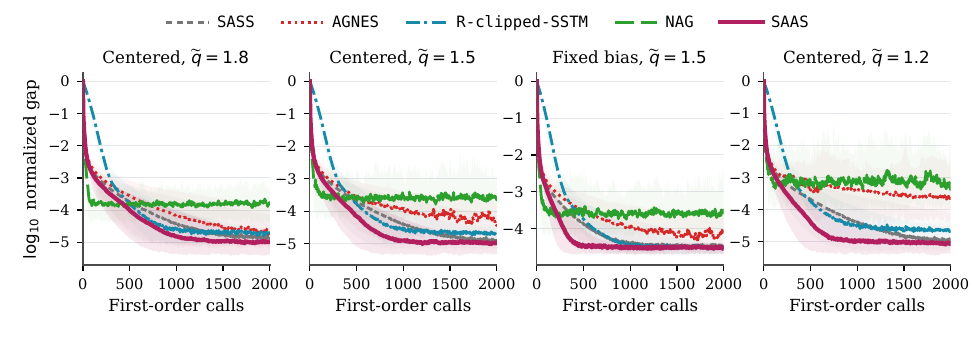}
\caption{\textbf{Heavy tails and bias affect early and late progress differently.}
Normalized gaps over the first 2,000 calls. Each panel pools 24 paired
paths on each of three fixed problems, so its bands include both
problem and oracle variation. Scores use the full 4,500-call horizon
shown in \Cref{fig:app_powerlaw_ls_problemwise}.}
\label{fig:exp_powerlaw_ls}
\end{figure}

\texttt{SAAS} has the lowest median full score in all four noisy settings, while the lowest median late score depends on the error law. \texttt{SAAS} remains lowest for centered errors with $\widetilde q=1.8$ and $1.5$. At $\widetilde q=1.2$, \texttt{SASS} has a lower late score ($-5.035$ versus $-5.027$); with fixed bias, both \texttt{SASS} and \texttt{R-clipped-SSTM} have lower late scores than \texttt{SAAS}. \Cref{app:exp_powerlaw_ls} gives the full trajectories, scores, and paired comparisons for the three problems.

Additional comparisons on real-data ridge-logistic objectives are
reported in \Cref{app:exp_logistic_details}.

\FloatBarrier

\subsection{Neural-network comparisons}
\label{subsec:exp_nn_comparison}

We train a two-hidden-layer ReLU network of widths \(32\)--\(16\)
on handwritten digits~\cite{pedregosa2011scikit}, using minibatches of
128 examples. At each call, all methods receive a minibatch gradient
whose sign is reversed with probability
\(\delta_{\mathrm{g}}\in\{0.1,0.2,0.3,0.4\}\). Every output is rescaled
by \(1/(1-2\delta_{\mathrm{g}})\) to preserve its mean.
The sign draws are independent across calls and of the trial step size.
For \texttt{SAAS} and \texttt{SASS}, we also add independent
uniform errors in \([-a,a]\) to the query and candidate losses on the
same minibatch, with \(a/\phi(\bm{x}_0)\in\{0,10^{-4},10^{-3}\}\).
The amplitude is fixed within each run, and the two searches share the
same standardized error draws at each call index.

We compare \texttt{SAAS} with \texttt{SASS},
\texttt{Adam}~\cite{kingma2015adam}, globally norm-clipped
\texttt{RMSprop}, and \texttt{AGNES}. For each sign-flip probability,
the base parameters are selected at \(a=0\) and then held fixed as
\(a\) varies. At positive noise levels, each search separately selects
its constant Armijo relaxation \(2\epsilon_{\mathrm{f}}'\) on pilot
paths, using the same additional selection budget for the two searches.
Evaluation uses eight paired paths excluded from parameter selection.
\Cref{app:exp_nn_comparison} gives the tuning stages and selected inputs.

Each method is run for at most 2,400 minibatch forward and backward
passes, with each pass counting as one work unit. A search trial uses
two forward passes and one backward pass; an optimizer update uses one
of each. Full-training loss evaluations used for monitoring are excluded
from these counts.

\begin{figure}[!htbp]
\centering
\includegraphics[width=\linewidth]{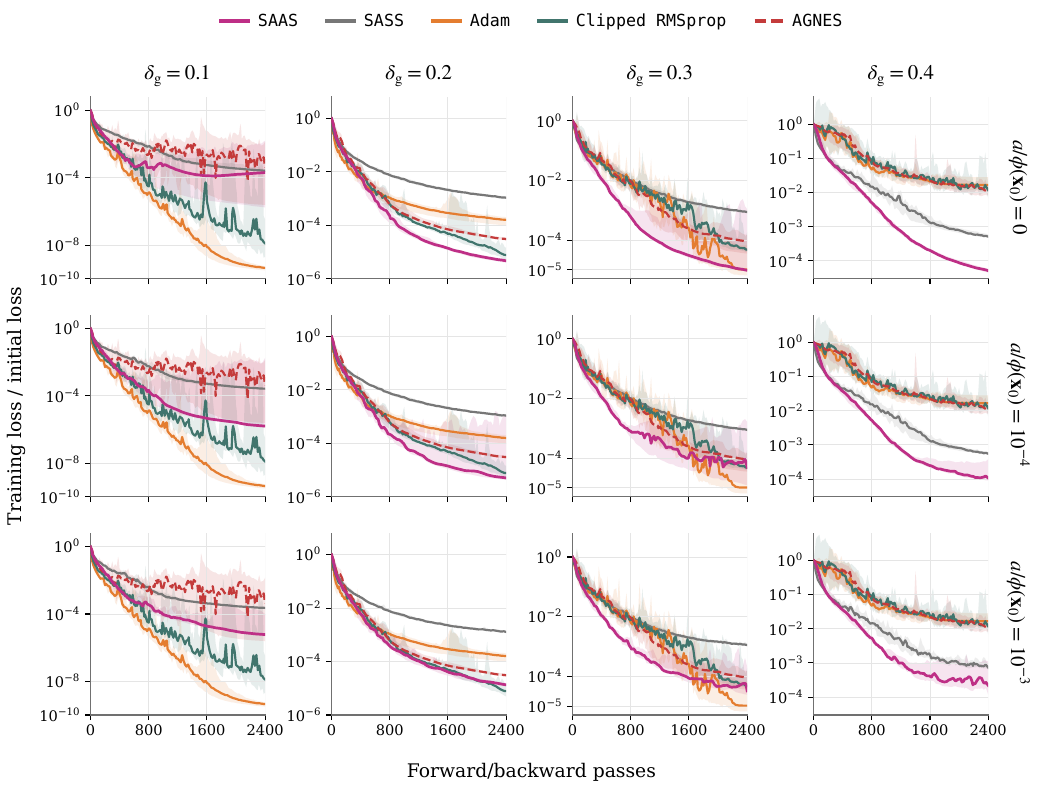}
\caption{\textbf{Training with gradient and function-value errors.}
Columns vary the gradient sign-flip probability \(\delta_{\mathrm{g}}\);
rows use function-value error amplitudes
\(a/\phi(\bm{x}_0)=0,10^{-4},10^{-3}\).
Curves and bands show pointwise medians and 10th--90th percentiles
of normalized full-training loss over eight paired paths.
Vertical ranges are shared within each column.
The horizontal axis counts minibatch forward and backward
passes. The two searches receive the function-value errors;
\texttt{Adam}, clipped \texttt{RMSprop}, and \texttt{AGNES} use the same
paths across rows because their updates do not use function values.}
\label{fig:exp_nn_signflip}
\end{figure}

\Cref{fig:exp_nn_signflip} reports the comparison.
 The full-score ranking changes with the reversal probability. \texttt{Adam} and clipped \texttt{RMSprop} have lower median full scores than \texttt{SAAS} at $\delta_{\mathrm g}=0.1$ across all three function-value noise levels; \texttt{SAAS} has the lowest median full score in the nine settings with $\delta_{\mathrm g}=0.2,0.3,0.4$. The final-loss comparison differs: for $\delta_{\mathrm g}=0.3$, \texttt{Adam} has a lower median final loss than \texttt{SAAS} at both positive noise levels; for $\delta_{\mathrm g}=0.2$, clipped \texttt{RMSprop} does so at the larger level. For $\delta_{\mathrm g}=0.2,0.3,0.4$, \texttt{SAAS} has higher full and late scores at both positive noise levels than at $a=0$, with the Armijo relaxation selected separately for each setting. \Cref{tab:exp_nn_comparison_summary}~ reports the scores, final losses, and paired full-score comparisons at $a/\varphi(\bm x_0)=10^{-3}$.

\FloatBarrier

\subsection{Ablation studies}
\label{subsec:exp_ablations}

We now vary two components of \texttt{SAAS}: the gradient-norm
safeguard and the allocation of predicted model reduction between
search and momentum.

\Needspace{8\baselineskip}
\subsubsection{Gradient-norm safeguard}
\label{subsubsec:exp_nn_safeguard}

We vary \(\epsilon_{\rm rel}\) in condition~\eqref{Check_3} while
holding the other search inputs fixed. After an accepted trial with
\(\|\mathbf G_t\|<\epsilon_{\rm rel}\), the algorithm retains the
accepted point and reduces the next trial step size.
\Cref{fig:exp_nn_search_diagnostics} compares
\(\epsilon_{\rm rel}=0\) and \(0.012\) on twelve paired digits paths,
without a step-size cap. At the displayed endpoint, the positive
threshold lowers \texttt{SAAS}'s median normalized training loss from
\(9.18\times10^{-4}\) to \(3.58\times10^{-4}\); its late score over
the displayed window is lower on eleven of the twelve paired paths.
Its smaller trial
steps late in training accompany more frequent acceptance of reversed
gradients. The threshold does not activate in \texttt{SASS},
whose trajectories coincide across rows. The inputs and selection
procedure appear in \Cref{app:exp_nn_search_diagnostics}.

\begin{figure}[!htbp]
\centering
\includegraphics[width=\linewidth]{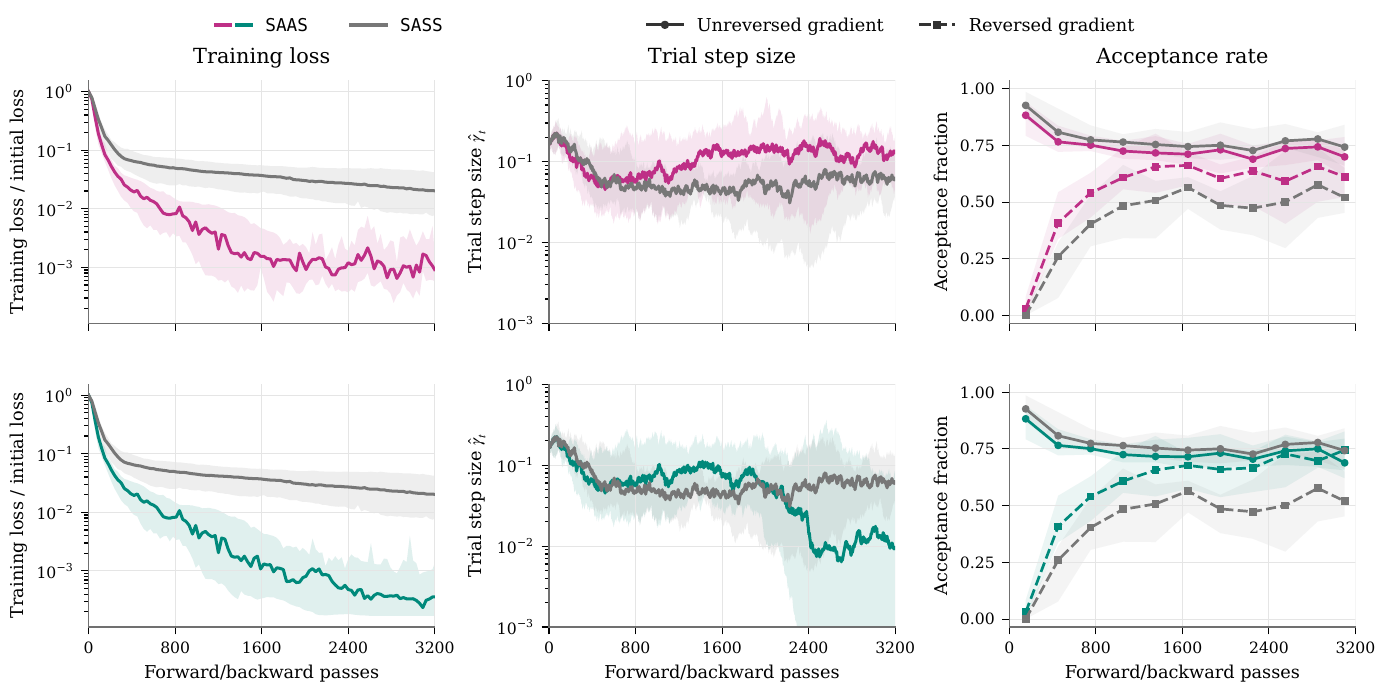}
\caption{\textbf{Effect of the gradient-norm safeguard.}
Training loss, trial step size, and acceptance rates over the first
1,066 trials (3,198 charged forward/backward work units). Both rows use \(\delta_{\mathrm{g}}=0.3\)
and pointwise function-value errors with
\(a/\phi(\bm{x}_0)=0.007\); the upper and lower rows use
\(\epsilon_{\rm rel}=0\) and \(0.012\), respectively. Other search
inputs are common, with no step-size cap. Curves and bands show
pointwise medians and 10th--90th percentiles over twelve paired paths.
Acceptance rates separate reversed and unreversed gradients in
consecutive 100-trial intervals. The lower row's step-size band
continues below the displayed \(10^{-3}\) axis limit.}
\label{fig:exp_nn_search_diagnostics}
\end{figure}

\FloatBarrier

\subsubsection{Search and momentum scales}

\label{subsubsec:exp_nn_ablation}

We vary the search scale \(\theta\) and momentum allocation
\(\theta_{\mathrm m}\), using the classical pair \((0.5,0.5)\) as a
reference. The neural-network comparison is shown here; a complementary
ablation on strongly convex quadratics is in
\Cref{app:exp_selected_ablation_details}.

For neural-network training, we return to the digits task of
\Cref{subsec:exp_nn_comparison},
varying \((\theta,\theta_{\mathrm m})\) across three network widths
and four batch sizes without a step-size cap. For each width, we fix
the expansion and contraction factors selected on separate
\(\theta=0.5\) runs, then vary the two scales without retuning them.
Each trial uses the same fresh minibatch for its gradient and Armijo
comparison, with no added sign flips or function-value errors and zero
Armijo relaxation. The gradient-norm threshold is zero throughout this
scale comparison.
\Cref{fig:exp_nn_curves_05} compares all four batch sizes at
\(\theta=0.5\); \Cref{fig:exp_nn_batch_allocation} summarizes the
three smaller batches.
The full grid and protocol appear in \Cref{app:exp_nn_ablation}.

\begin{figure}[H]
\centering
\captionsetup{font=small,skip=6pt}
\includegraphics[width=\linewidth]{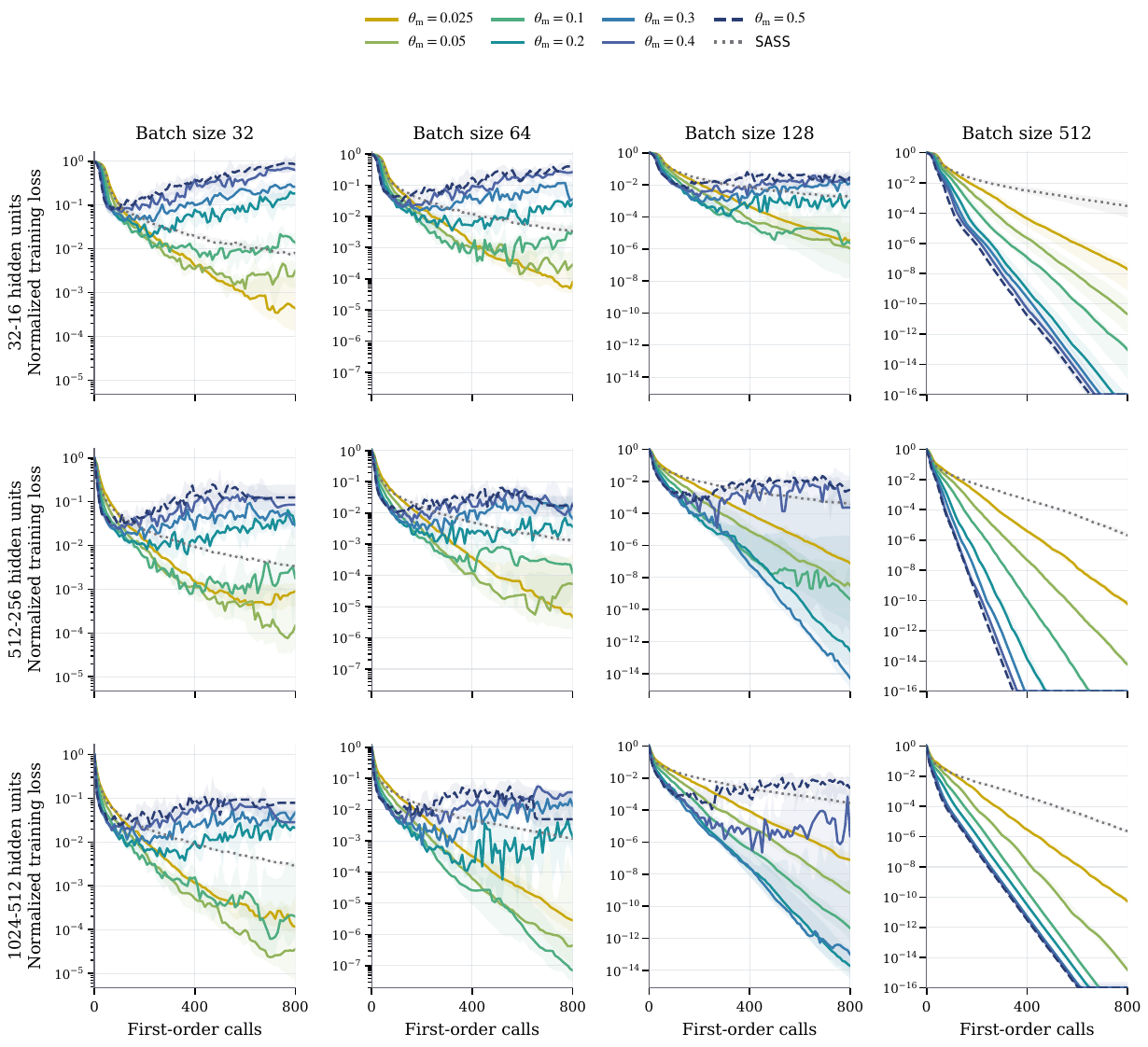}
\caption{\textbf{Momentum allocation at different batch sizes.}
The search scale is \(\theta=0.5\).
Rows show hidden-layer widths \(32\)--\(16\), \(512\)--\(256\), and
\(1024\)--\(512\); columns show batch sizes 32, 64, 128, and 512.
Curves and bands are pointwise medians and 10th--90th percentiles of
normalized training loss over five paired paths and 800 first-order
calls. The dashed curve uses the classical pair \((0.5,0.5)\);
the gray curve is \texttt{SASS} at the same \(\theta\).}
\label{fig:exp_nn_curves_05}
\end{figure}

\begin{figure}[!htbp]
\centering
\captionsetup{font=small,skip=6pt}
\includegraphics[width=0.82\linewidth]{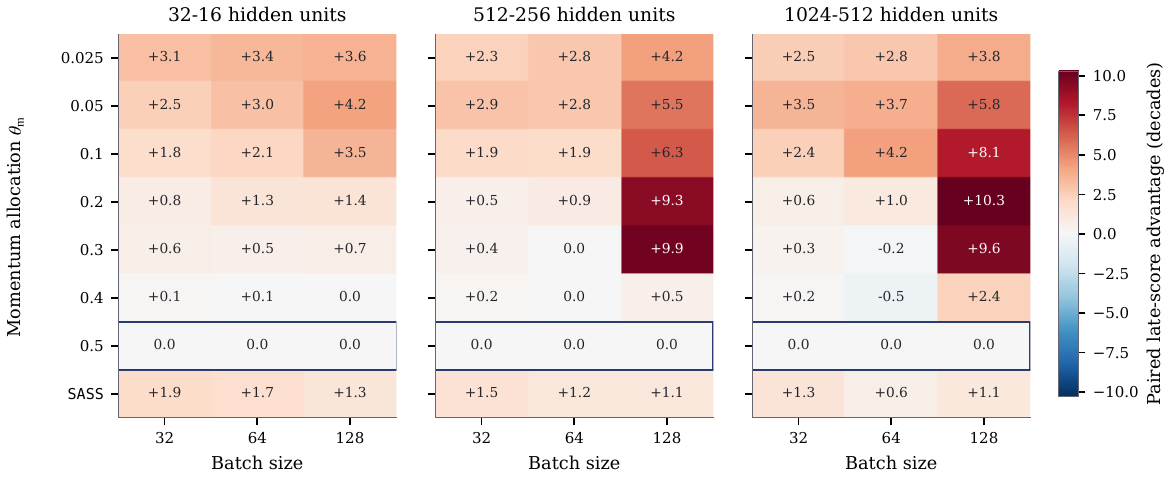}
\caption{\textbf{Batch size changes the preferred momentum allocation.}
Columns show batch sizes 32, 64, and 128. At \(\theta=0.5\), each cell gives the median paired late-score advantage
over \(\theta_{\mathrm{m}}=0.5\) across five paths. Positive values indicate
a lower late log-loss score. The panels correspond to the network widths
in \Cref{fig:exp_nn_curves_05}; the bottom row shows \texttt{SASS}.
All panels use the same color scale.}
\label{fig:exp_nn_batch_allocation}
\end{figure}

At batch sizes 32, 64, and 128, at least one smaller momentum
allocation lowers the late score relative to the classical pair in
each network. The paired gains at batch size 128 are largest in the two
wider networks. At batch size 512, no tested scale pair has a positive
median paired late-score advantage over the classical pair. Thus the
allocation favored by the late score changes with
batch size even at a fixed search scale; the paired differences for the
three smaller batches appear in \Cref{fig:exp_nn_batch_allocation}.

\FloatBarrier
\endgroup
\section{Related work}
\texttt{SAAS} brings together probabilistic adaptive search and accelerated methods with inexact information. Their interaction raises the question of how to control gradient errors when the estimate used in an accepted step also enters subsequent momentum updates.

\paragraph{Probabilistic models and adaptive search.}
Probabilistic models replace accuracy at every trial by an accuracy
condition that holds with sufficiently high conditional probability.
In particular, the line-search analyses in \cite{Cartis2018,Albert2021,Xie2024} impose
no bound on the gradient error outside this event and require neither
unbiasedness nor bounded variance. Early trust-region analyses established
almost-sure stationarity with exact function values \cite{Bandeira2014probabilistic} and with
noisy function estimates \cite{Chen2018trust}. Expected complexity bounds cover stochastic trust-region methods
\cite{blanchet2019convergence}, as well as line-search and cubic-regularization methods
with exact function values \cite{Cartis2018}.

For inexact function values, stochastic line search has been analyzed
under dynamic accuracy requirements and conditional second-moment
bounds on function errors \cite{courtny}. Relaxed acceptance tests also give
expected complexity bounds above noise-dependent accuracy thresholds
under bounded function-value errors \cite{Albert2021}. High-probability bounds
for adaptive step search allow function errors satisfying a one-sided
exponential-moment condition \cite{Xie2024}. The unified framework in \cite{scheinberg2025stochasticadaptiveoptimizationunreliable}
treats finite-moment function-value errors and allows separate
multiplicative expansion and contraction factors, together with
norm-based step-size regulation after acceptance. Related developments
treat stochastic cubic regularization \cite{scheinberg2023cubic} and equality-constrained
optimization through stochastic sequential quadratic programming \cite{SQP-XIE}.

In these line-search methods, function-value comparisons control
objective increases even when the gradient accuracy condition fails.
With momentum, an accepted gradient also changes an auxiliary state
whose distance to a minimizer enters the Lyapunov function. Descent
at the candidate therefore does not by itself supply the error control
needed in the accelerated recursion. Our cFO model bounds gradient errors almost surely, including
outside the accuracy event. The accompanying relative-error
restriction allows the analysis to absorb errors on accurate trials
while retaining accelerated contraction.

\paragraph{Error control in accelerated methods.}
Deterministic inexact-oracle analyses relate accelerated convergence
to both the size of the errors and the way they accumulate.
On bounded feasible sets, controlling the linear-model error caused
by gradient approximation yields accelerated decay up to a fixed
residual \cite{d2008smooth}.
The coupled local-model oracle in \cite{devolder2014first} makes the accumulation of
oracle errors explicit in the accelerated convergence bound. With
separate errors in gradients and proximal computations, weighted
summability conditions recover the exact-method rates \cite{Inexactprox2011}.
Accelerated methods with adaptive gradient estimation likewise retain
deterministic-order iteration bounds under suitable relative and
absolute accuracy conditions, including for biased estimates \cite{bollapragada2025convergencecomplexityproximalgradient}.

For unbiased stochastic gradients, accelerated bounds balance
deterministic progress with a separate noise contribution. \texttt{AC-SA}
attains optimal bounds on expected suboptimality under bounded
variance \cite{lan2012optimal}, while state-dependent models allow the variance
to grow with suboptimality \cite{ilandarideva2025state}. When the stochastic-gradient
second moment is bounded by a multiple of the squared true-gradient
norm, accelerated \texttt{SGD} recovers deterministic-type rates with a
noise-dependent factor \cite{Vaswani2019a}. \texttt{AGNES} also treats unbiased
multiplicative noise, using additional freedom in the momentum update
to obtain accelerated expected convergence with fixed gradient step
sizes at any finite noise level \cite{gupta2024nesterovaccelerationdespitenoisy}. These multiplicative models
make the noise vanish as the true gradient tends to zero.

Clipping provides high-probability accelerated guarantees for
heavy-tailed unbiased estimates under finite variance \cite{gorbunov2020accelerated} and,
more generally, finite moments of order in $(1,2]$ \cite{nguyen2023improved}. Unmodified
stochastic subgradients can also attain optimal expected oracle
complexity on bounded domains under finite-moment assumptions;
high-probability bounds in this setting are established under an
additional Sub-Weibull condition \cite{he2025accelerated}. \texttt{SAAS} studies accelerated
search with persistent, possibly biased gradient errors. Its sFO
analysis controls weighted sums of error magnitudes without relying
on cancellation of the gradient errors. This allows infinite variance
and leads to high-probability trial bounds for reaching
error-dependent convergence neighborhoods.

\paragraph{Accelerated adaptive search.}
Backtracking is already part of deterministic accelerated composite
methods \cite{FISTA,Nesterov2013}. Full-backtracking \texttt{FISTA} allows the step size
to increase as well as decrease \cite{BKTR}, and related strategies
exploit strong convexity \cite{mu>0adaptive}. \texttt{FISTA-SS} combines full backtracking
with probabilistically accurate gradients and permits biased
estimates \cite{stoFISTA}. With exact function and proximal evaluations,
it obtains an accelerated bound on the expected number of trials
needed to reach a target objective gap. Its analysis supplements
probabilistic accuracy with a conditional mean-square error bound
scaled by the trial step size, momentum weight, and iteration index.
This additional bound controls all oracle outcomes and limits error
accumulation in the accelerated recursion.

\texttt{SAAS} develops accelerated search for smooth objectives
with persistent gradient errors and inexact function-value
comparisons. The separately chosen search and momentum fractions
enter a common Lyapunov recursion. We combine this recursion
with the step-size counting argument
in~\cite{scheinberg2025stochasticadaptiveoptimizationunreliable}
to obtain high-probability trial bounds under the stated
first-order error models.

\section{Conclusion}\label{sec:conclusion}
The analysis shows that momentum allocation affects admissible target accuracy as well as contraction. A trial can satisfy the decrease test while introducing gradient errors into later extrapolation; the margin \(\theta-\theta_{\mathrm m}\) leaves certified decrease available to absorb those errors. At the same time, a more demanding test can incur additional gradient queries through rejection. The balance between search and momentum must therefore be assessed over the full sequence of trials, accounting for both the cost of obtaining accepted progress and the errors carried forward after acceptance.

The experiments show that the balance favoring stable and rapid progress varies with gradient reliability and minibatch size. Smaller momentum allocations improve late accuracy in the noisy ablations, while the classical choice can be faster with exact gradients. The sensitivity to error direction at a fixed error-norm distribution also shows that noise magnitude alone does not determine performance. Together with the differences between early and late performance, these observations suggest choosing the two scales jointly for the oracle setting and the accuracy sought. Developing rules that adjust this choice using observed acceptance and progress is a natural next step.

\section*{Acknowledgments}
We thank Dr.~Miaolan Xie for support during the preparation of an early
version of this paper.

\FloatBarrier
\phantomsection
\addcontentsline{toc}{section}{References}
\bibliographystyle{alpha}
\bibliography{references}

\clearpage
\appendix
\appendixnumbering
\appendixproofspacing
\section{Proofs}\label{app:proofs}
\subsection{Auxiliary results}

\paragraph{Momentum-parameter properties.}
\begin{lemma}\label{Para_lemma}
Let \(\{\gamma_t\}_{t\ge0}\) be generated by
\eqref{eq:gamma-dyn}, and let \(\{\alpha_t\}_{t\ge0}\) be defined by
\Cref{para}. Let \(\lambda_t\) be defined by \eqref{dfn_lambda_t},
with \(\lambda_0\coloneqq 1\).
For every \(t\ge1\), 
\begin{equation}\label{eq:momentum_ratio_invariant}
2\theta_{\mathrm{m}}(1-\vartheta_{\mathrm{m}})^2\mu
<\frac{\alpha_t^2}{\gamma_t}
\le\frac{\alpha_{t-1}^2}{\gamma_{t-1}}.
\end{equation}
The same bounds hold with \((\alpha_t,\gamma_t)\) replaced by
\((\hat\alpha_t,\hat\gamma_t)\).
The following properties also hold.
\begin{enumerate}[
  label=\textbf{\Alph*.},
  ref=\thelemma(\Alph*),
  leftmargin=*,
  labelsep=0.6em,
  itemsep=3pt plus 1pt minus 1pt,
  topsep=3pt plus 1pt minus 1pt,
  parsep=0pt,
  partopsep=0pt
]

\item\label[lemma]{para_lemma_a}
\textbf{Monotonicity.}
For every \(t\ge1\),
\[
{\gamma_t\lambda_t}/{\alpha_t^2}
\,\le\,
{\gamma_{t-1}\lambda_{t-1}}/{\alpha_{t-1}^2}
\,\le\,\cdots\,\le\,{\gamma_0}/{\alpha_0^2}.
\]
If \(\mu=0\), all the inequalities are equalities.

\item\label[lemma]{para_lemma_b}
\textbf{Two-sided quadratic bound.}
If \(\mu=0\), then, for every \(t\ge1\),
\[
\left(\textstyle
1+{\alpha_0}/{\sqrt{\gamma_0}}
\sum_{i=1}^t\Theta_i\sqrt{\gamma_i}
\right)^{-2}
\le\;
\lambda_t
\;\le\;
\left(\textstyle
1+\tfrac{\alpha_0}{2}/{\sqrt{\gamma_0}}
\sum_{i=1}^t\Theta_i\sqrt{\gamma_i}
\right)^{-2}.
\]

\item\label[lemma]{para_lemma_c}
\textbf{Exponential upper bound.}
If \(\mu>0\), then, for every \(t\ge1\),
\[
\lambda_t
\le
\exp\left(\textstyle
\sum_{i=1}^t\Theta_i
\ln\left(
1-(1-\vartheta_{\mathrm{m}})
\sqrt{2\theta_{\mathrm{m}}\mu\gamma_i}
\right)
\right).
\]

\item\label[lemma]{para_lemma_d}
\textbf{Uniform upper bound.}
Suppose that \(\gamma_{\max}=1/(2\mu)\) when \(\mu>0\), and set
\(\bar\alpha\coloneqq (1+\nu_{\mathrm{inc}})/(2+\nu_{\mathrm{inc}}-\theta_{\mathrm{m}})\).
If \(\alpha_0\le\bar\alpha\), then \(\alpha_t\le\bar\alpha\) for
every \(t\ge0\), \(\hat\alpha_t\le\bar\alpha\) for every \(t\ge1\), and
\begin{equation}\label{eq:alpha_ratio_simple}
{\bar\alpha}/{(1-\bar\alpha)}
={(1+\nu_{\mathrm{inc}})}/{(1-\theta_{\mathrm{m}})}.
\end{equation}

\item\label[lemma]{para_lemma_e}
\textbf{Trial weights.}
Under the conditions of \Cref{para_lemma_d}, \(\hat\alpha_t\) is
\(\mathcal F_{t-1}\)-measurable and, for every \(t\ge1\),
\begin{equation}\label{eq:sfo_trial_weight_balance}
\sum_{i=1}^t\hat\alpha_i
\le
\frac{2\sqrt{\nu_{\mathrm{inc}}}}
{(1-\bar\alpha)(1-\nu_{\mathrm{dec}})}
\left(1+\sum_{i=1}^t\Theta_i\alpha_i\right)\qquad \textit{on every path.}
\end{equation}
\end{enumerate}
\end{lemma}
\begin{proof}
\phantomsection\label{proof:momentum_bounds}
Write \(\chi_\mu\coloneqq 2\theta_{\mathrm{m}}(1-\vartheta_{\mathrm{m}})^2\mu\).
We first establish \eqref{eq:momentum_ratio_invariant} and the root
bounds used below. At a fixed retained pair, \(\hat\alpha_t\) is the
positive root of
\[
\frac{\alpha^2}{\hat\gamma_t}
-(1-\alpha)\frac{\alpha_{t-1}^2}{\gamma_{t-1}}
-\chi_\mu\alpha=0.
\]
The positive leading coefficient and negative constant term give a
unique positive root. At \(\alpha=1\), the left-hand side equals
\(1/\hat\gamma_t-\chi_\mu\ge1/\gamma_{\max}-\chi_\mu>0\)
by \Cref{para}; for \(\mu=0\), this last inequality is automatic.
Thus \(0<\hat\alpha_t<1\). The initialization gives
\(\alpha_0^2/\gamma_0>\chi_\mu\), and \eqref{para_alpha} writes
the trial ratio as
\[
\frac{\hat\alpha_t^2}{\hat\gamma_t}
=(1-\hat\alpha_t)\frac{\alpha_{t-1}^2}{\gamma_{t-1}}
+\hat\alpha_t\chi_\mu.
\]
This ratio lies strictly between \(\chi_\mu\) and the retained ratio.
Since acceptance retains the trial pair and rejection preserves the
previous pair, induction proves \eqref{eq:momentum_ratio_invariant}
and \(0<\alpha_t<1\).

\medskip
\noindent\textbf{Part A.}
There is nothing to prove when \(\Theta_t=0\). Otherwise,
\eqref{para_alpha} and \(\chi_\mu\ge0\) give
\(\alpha_t^2/\gamma_t\ge
(1-\alpha_t)\alpha_{t-1}^2/\gamma_{t-1}\). Since
\(\lambda_t=(1-\alpha_t)\lambda_{t-1}\),
\[
\frac{\gamma_t\lambda_t}{\alpha_t^2}
=
\frac{\gamma_t(1-\alpha_t)\lambda_{t-1}}{\alpha_t^2}
\le
\frac{\gamma_{t-1}\lambda_{t-1}}{\alpha_{t-1}^2}.
\]
Iteration proves \Cref{para_lemma_a}. If \(\mu=0\), then
\(\chi_\mu=0\), and equality holds throughout.

\medskip
\noindent\textbf{Part B.}
Let \(\mu=0\) and write \(s_t\coloneqq \sqrt{\gamma_t}/\alpha_t\). On an
accepted trial, \(s_t=s_{t-1}/\sqrt{1-\alpha_t}\), so
\[
s_t-s_{t-1}
=\frac{\sqrt{\gamma_t}}{1+\sqrt{1-\alpha_t}}
\in\left[\frac12\sqrt{\gamma_t},\sqrt{\gamma_t}\right].
\]
The increment is zero on rejected trials, and hence
\[
s_0+\frac12\sum_{i=1}^t\Theta_i\sqrt{\gamma_i}
\le
s_t
\le
s_0+\sum_{i=1}^t\Theta_i\sqrt{\gamma_i}.
\]
\Cref{para_lemma_a} gives
\(\lambda_t=(s_0/s_t)^2\); substituting the preceding two-sided bound
proves \Cref{para_lemma_b}.

\medskip
\noindent\textbf{Part C.}
For \(\mu>0\), \eqref{eq:momentum_ratio_invariant} gives
$\alpha_t
>
\sqrt{\chi_\mu\gamma_t}
=
(1-\vartheta_{\mathrm{m}})\sqrt{2\theta_{\mathrm{m}}\mu\gamma_t}$
whenever $\Theta_t=1$. 
Consequently,
$\lambda_t
\le
\prod_{i=1}^t
\left(
1-(1-\vartheta_{\mathrm{m}})
\sqrt{2\theta_{\mathrm{m}}\mu\gamma_i}
\right)^{\Theta_i}$, 
which proves \Cref{para_lemma_c}.

\medskip
\noindent\textbf{Part D.}
The initialization and step-size dynamics give
\(\hat\gamma_t/\gamma_{t-1}\le\nu_{\mathrm{inc}}\) on every trial.
Moreover, \(0\le\chi_\mu\hat\gamma_t
\le\theta_{\mathrm{m}}(1-\vartheta_{\mathrm{m}})^2<1\).
Rearranging \eqref{para_alpha} therefore gives
\[
\frac{\hat\alpha_t}{1-\hat\alpha_t}
=
\frac{
\hat\alpha_t+\frac{\hat\gamma_t}{\gamma_{t-1}}\alpha_{t-1}^2
}{
1-\chi_\mu\hat\gamma_t
}
\le
\frac{1+\nu_{\mathrm{inc}}}
{1-\theta_{\mathrm{m}}(1-\vartheta_{\mathrm{m}})^2}
\le
\frac{1+\nu_{\mathrm{inc}}}{1-\theta_{\mathrm{m}}}.
\]
Thus \(\hat\alpha_t\le\bar\alpha\). Acceptance sets
\(\alpha_t=\hat\alpha_t\), while rejection retains \(\alpha_{t-1}\),
so \(\alpha_0\le\bar\alpha\) gives the bound for \(\alpha_t\).
The definition of \(\bar\alpha\) yields \eqref{eq:alpha_ratio_simple}.

\medskip
\noindent\textbf{Part E.}
The retained pair and \(\hat\gamma_t\) are
\(\mathcal F_{t-1}\)-measurable, so the same holds for \(\hat\alpha_t\).
At a fixed retained pair, implicit differentiation of
\eqref{para_alpha} gives
\[
\frac{1-\bar\alpha}{2}
\le
\frac{\partial\log\hat\alpha_t}{\partial\log\hat\gamma_t}
=\frac{1-\hat\alpha_t}
{2-\hat\alpha_t-\chi_\mu\hat\gamma_t/\hat\alpha_t}
\le
\frac12
\qquad\text{when }\hat\alpha_t\le\bar\alpha.
\]
Here \(0\le\chi_\mu\hat\gamma_t/\hat\alpha_t<\hat\alpha_t\),
and the upper bound holds throughout \(0<\hat\alpha_t<1\).
The root also increases with the retained ratio
\(\alpha_{t-1}^2/\gamma_{t-1}\). After acceptance this ratio cannot
increase by \eqref{eq:momentum_ratio_invariant}, including on an
accepted but unreliable trial with \(\Lambda_t=0\); rejection leaves
it unchanged. Integrating the derivative bounds, using this
monotonicity, and applying the concavity bound
\(\nu_{\mathrm{dec}}^{(1-\bar\alpha)/2}
\le1-(1-\bar\alpha)(1-\nu_{\mathrm{dec}})/2\) yields
\begin{equation}\label{eq:sfo_trial_weight_dynamics}
\hat\alpha_{t+1}\le
\begin{cases}
\left[1-\dfrac{(1-\bar\alpha)(1-\nu_{\mathrm{dec}})}{2}\right]
\hat\alpha_t,&\Lambda_t=0,\\
\sqrt{\nu_{\mathrm{inc}}}\,\hat\alpha_t,&\Lambda_t=1.
\end{cases}
\end{equation}
Summing the resulting inequalities and using
\(\Lambda_i\le\Theta_i\), \(\Theta_i\hat\alpha_i=\Theta_i\alpha_i\),
and \(\hat\alpha_1<1\) gives
\[
\begin{aligned}
\frac{(1-\bar\alpha)(1-\nu_{\mathrm{dec}})}{2}
\sum_{i=1}^t\hat\alpha_i
&\le\hat\alpha_1-\hat\alpha_{t+1}+\left[\sqrt{\nu_{\mathrm{inc}}}-1+
\frac{(1-\bar\alpha)(1-\nu_{\mathrm{dec}})}{2}\right]
\sum_{i=1}^t\Lambda_i\hat\alpha_i\\
&\le\sqrt{\nu_{\mathrm{inc}}}
\left(1+\sum_{i=1}^t\Theta_i\alpha_i\right).
\end{aligned}
\]
This proves \eqref{eq:sfo_trial_weight_balance}.
\end{proof}
\newpage

\begin{proof}[\textbf{Algebraic verification for \Cref{energy_1}}]
Consider an accepted trial and suppress hats.
\phantomsection\label{proof:energy_algebra}%
Define locally
\[
\eta_t\coloneqq
\frac{\alpha_{t-1}^2}
{2\theta_{\mathrm{m}}(1-\vartheta_{\mathrm{m}})^2\gamma_{t-1}},
\qquad
Q_t\coloneqq\frac{\eta_t}{2}\|\bm{z}_t-\bm{x}^\star\|^2,
\qquad
h_t\coloneqq\frac{2\theta_{\mathrm{m}}(1-\vartheta_{\mathrm{m}})\gamma_t}{\alpha_t}.
\]
By \eqref{eq:momentum_parameterization},
\(\gamma_t+\frac{(1-\vartheta_{\mathrm{m}})(1-\alpha_t)}{\alpha_t}\gamma_t'=h_t\),
and
\[
\begin{aligned}
\eta_{t+1}(1-\beta_t)=(1-\alpha_t)\eta_t,
\qquad
\eta_{t+1}\beta_t=\mu\alpha_t,\qquad
\eta_{t+1}h_t(1-\vartheta_{\mathrm{m}})={\alpha_t},
\qquad
{\eta_{t+1}h_t^2}=2\theta_{\mathrm{m}}\gamma_t.
\end{aligned}
\]
Substituting the accepted updates in \eqref{z_t} yields
\[
\begin{aligned}
\bm{z}_{t+1}&=(1-\beta_t)\bm{z}_t+\beta_t\bm{y}_t-h_t\mathbf{G}_t,
\\
\alpha_t(1-\beta_t)(\bm{z}_t-\bm{y}_t)
&={(1-\vartheta_{\mathrm{m}})(1-\alpha_t)}
(\bm{y}_t-\bm{x}_t).
\end{aligned}
\]
Expanding \(\|\bm{z}_{t+1}-\bm{x}^\star\|^2\) with these identities gives
\begin{align*}
Q_{t+1}
={}&(1-\alpha_t)Q_t
+\frac{\mu\alpha_t}{2}\|\bm{y}_t-\bm{x}^\star\|^2
+\theta_{\mathrm{m}}\gamma_t\|\mathbf{G}_t\|^2\\
&-(1-\alpha_t)\langle\mathbf{G}_t,\bm{y}_t-\bm{x}_t\rangle
-\frac{\alpha_t}{1-\vartheta_{\mathrm{m}}}
\langle\mathbf{G}_t,\bm{y}_t-\bm{x}^\star\rangle-\frac{\mu(1-\vartheta_{\mathrm{m}})^2(1-\alpha_t)^2}
{2\alpha_t(1-\beta_t)}\|\bm{y}_t-\bm{x}_t\|^2.
\end{align*}
This is \eqref{eq:accepted_distance_energy_identity}.
\end{proof}

\begin{proof}[\textbf{Proof of \Cref{lem:good_step_growth}}]
\phantomsection\label{proof:good_step_growth}
Apply Lemmas~2--3 of
\cite{scheinberg2025stochasticadaptiveoptimizationunreliable}
with \(I_i,\Lambda_i,U_i\) as the true, successful, and large-step
indicators, and \(\hat\gamma_i\) as the step-size parameter.
Their counting argument applies to \eqref{eq:gamma-dyn} and
\((1-U_i)I_i\le\Lambda_i\) before stopping: both rejected and
accepted but unreliable trials have \(\Lambda_i=0\) and reduce the
next trial step size. The cap preserves
\(\hat\gamma_{i+1}\le\nu_{\mathrm{inc}}\hat\gamma_i\) and is inactive
on small trials, since \(\bar\gamma\le\gamma_0/\nu_{\mathrm{inc}}\).
Using \(-\ln\nu_{\mathrm{inc}}/\ln\nu_{\mathrm{dec}}=(1-p_0)/p_0\), the
standing integrality assumption, and the definition of \(N_0\), the
two cited counting bounds give
\begin{equation}\label{eq:good_step_pathwise_count}
\sum_{i=1}^t I_i
\le \frac{1-p_0}{p_0}\bigl(N_{\mathrm{good}}(t)+N_0\bigr)+p_0t
\qquad\text{on }\{t<T_{\boldsymbol\varepsilon}\}.
\end{equation}
By \eqref{eq:N_good_envelope}, this implies
\begin{equation}\label{eq:good_step_event_inclusion}
\textstyle\left\{t<T_{\boldsymbol\varepsilon},\
N_{\mathrm{good}}(t)<\bar N_{\mathrm{good}}(t)\right\}
\subseteq
\left\{\sum_{i=1}^t I_i<\hat p t\right\};
\end{equation}
the left-hand event is empty when \(\bar N_{\mathrm{good}}(t)=0\).
Since \(\mathbb E[I_i\mid\mathcal F_{i-1}]\ge p\), for \(s<0\),
\(\mathbb E[e^{sI_i}\mid\mathcal F_{i-1}]
\le1-p+pe^s\le\exp\{p(s+s^2/2)\}\).
Iterating conditional expectations and applying Markov's inequality
with \(s=-(p-\hat p)/p\) therefore yields
\begin{equation}\label{eq:true_trial_count_tail}
\mathbb P\left\{\sum_{i=1}^t I_i<\hat p t\right\}
\le\exp\!\left\{t\left[(p-\hat p)s+\frac p2s^2\right]\right\}
=\exp\!\left(-\frac{(p-\hat p)^2}{2p}t\right).
\end{equation}
Together with \eqref{eq:good_step_event_inclusion}, this proves
\eqref{eq:good_step_count}.
\end{proof}

\begin{proof}[\textbf{Proof of \Cref{lem:good_step_contraction}}]
\phantomsection\label{proof:good_step_contraction}
Work on \(\{t<T_{\boldsymbol\varepsilon},\,
N_{\mathrm{good}}(t)\ge\bar N_{\mathrm{good}}(t)\}\). Each good step is
accepted and satisfies \eqref{eq:good_step_size_lower}.

\noindent\textbf{Case \(\mu=0\).}
By \Cref{para_lemma_b},
\begin{align*}
\lambda_t
\le
\left(
1+\frac{\alpha_0}{2\sqrt{\gamma_0}}
\sum_{i=1}^t\Theta_i\sqrt{\gamma_i}
\right)^{-2}
\le
\left(
1+\frac{\alpha_0}{2}
\sqrt{\frac{\bar\gamma}
{\nu_{\mathrm{inc}}\gamma_0}}\,
N_{\mathrm{good}}(t)
\right)^{-2}
\le
4\left(
2+\alpha_0
\sqrt{\frac{\bar\gamma}
{\nu_{\mathrm{inc}}\gamma_0}}\,
\bar N_{\mathrm{good}}(t)
\right)^{-2}.
\end{align*}

\noindent\textbf{Case \(\mu>0\).}
By \Cref{para_lemma_c},
\begin{align*}
\ln\lambda_t
&\le\;
\sum_{i=1}^t
\Theta_i
\ln\left(
1-(1-\vartheta_{\mathrm{m}})
\sqrt{2\theta_{\mathrm{m}}\mu\gamma_i}
\right)
\le\;
\bar N_{\mathrm{good}}(t)
\ln\left(
1-(1-\vartheta_{\mathrm{m}})
\sqrt{
\frac{2\theta_{\mathrm{m}}\mu\bar\gamma}
{\nu_{\mathrm{inc}}}
}
\right).
\end{align*}
Exponentiating gives the bound for \(\mu>0\). In both cases,
\eqref{eq:good_step_count} bounds the probability that the required
good-step count fails, proving
\eqref{pro_lambda_t_rate_mu=0}--\eqref{pro_lambda_t_rate_mu>0}.
\end{proof}
\paragraph{Martingale concentration inequalities.}
The first three lemmas treat deterministic bounds on the weight sums.
The final corollary applies these estimates when the total weight is random.
Throughout the three lemmas, fix \(t\ge1\) and \(S>0\), let
\(\{X_i\}_{i=1}^t\) be an integrable sequence adapted to
\(\{\mathcal F_i\}_{i=0}^t\), and let \(\{b_i\}_{i=1}^t\) be
nonnegative and predictable. We bound the  probability
\begin{equation}\label{eq:weighted_centered_maximal_probability}
\mathbb P\left\{
\max_{1\le n\le t}\sum_{i=1}^n b_i\widetilde X_i>S
\right\}.
\end{equation}
For uncentered applications, the same bounds also apply to
\begin{equation}\label{eq:weighted_uncentered_conversion}
\mathbb P\left\{\sum_{i=1}^t b_iX_i>\epsilon B_1+S\right\}
\end{equation}
whenever \(X_i\ge0\), \(\mathbb E_i[X_i]\le\epsilon\), and
\(\sum_{i=1}^t b_i\le B_1\) almost surely, for deterministic
\(\epsilon,B_1\ge0\): the inequality
\(\sum_i b_iX_i\le\epsilon B_1+\sum_i b_i\widetilde X_i\)
places the event in \eqref{eq:weighted_uncentered_conversion} inside
the event in \eqref{eq:weighted_centered_maximal_probability}.

\begin{lemma}[Martingale von Bahr--Esseen concentration]
\label{lem:weighted_concentration_unified}
In the common centered setup above, let
\(q\in(1,2]\) and suppose
\(\mathbb E_i[|\widetilde X_i|^q]\le M_q\) for all \(i\le t\),
where \(M_q\ge0\) is deterministic. If a deterministic
\(B_q\ge0\) satisfies \(\sum_{i=1}^t b_i^q\le B_q\) almost surely,
then
\begin{equation}\label{eq:weighted_concentration_moment}
\mathbb P\left\{
\max_{1\le n\le t}\sum_{i=1}^n b_i\widetilde X_i>S
\right\}
\le
\frac{K_qB_qM_q}{S^q},
\qquad K_q\coloneqq 2^{2-q}.
\end{equation}
\end{lemma}
\begin{proof}
The derivative of \(|x|^q\) is \((q-1)\)-H\"older with constant
\(qK_q\). Integration along a line segment gives
\(|x+y|^q\le|x|^q+q\operatorname{sign}(x)|x|^{q-1}y+K_q|y|^q\).
The sequence \(\{b_i\widetilde X_i\}\) consists of martingale
differences, and predictability gives
\(\mathbb E_i[|b_i\widetilde X_i|^q]\le M_qb_i^q\).
Iterating the preceding inequality and taking conditional expectations
eliminates the linear terms, yielding
\[
\mathbb E\left|\sum_{i=1}^t b_i\widetilde X_i\right|^q
\le K_qM_qB_q;
\]
the linear terms are integrable by H\"older's inequality. Doob's weak
maximal inequality applied to
\(|\sum_{i=1}^n b_i\widetilde X_i|^q\) proves
\eqref{eq:weighted_concentration_moment}.
\end{proof}

\begin{lemma}[Martingale Fuk--Nagaev concentration]
\label{lem:weighted_concentration_high_moment}
In the common centered setup above, let \(q>2\)
and suppose \(\mathbb E_i[|\widetilde X_i|^q]\le M_q\) for all
\(i\le t\), where \(M_q\ge0\) is deterministic. If deterministic
\(B_2,B_q\ge0\) satisfy \(\sum_{i=1}^t b_i^2\le B_2\) and
\(\sum_{i=1}^t b_i^q\le B_q\) almost surely, then
\begin{equation}\label{eq:weighted_concentration_high_moment}
\begin{aligned}
\mathbb P\left\{\max_{1\le n\le t}
\sum_{i=1}^n b_i\widetilde X_i>S\right\}
\le\exp\!\left(-\frac{2S^2}{(q+2)^2e^qM_q^{2/q}B_2}\right)+2\left(1+\frac2q\right)^q\frac{M_qB_q}{S^q}.
\end{aligned}
\end{equation}
The exponential term is understood as zero when \(M_qB_2=0\).
\end{lemma}
\begin{proof}
Conditional Jensen's inequality and predictability give
\[
\sum_{i=1}^t\mathbb E_i[(b_i\widetilde X_i)^2]
\le M_q^{2/q}B_2
\qquad\text{and}\qquad
\sum_{i=1}^t\mathbb E_i[|b_i\widetilde X_i|^q]\le M_qB_q
\]
almost surely. The martingale Fuk--Nagaev inequality
\cite[Theorem~2.4]{FAN2017538}, with truncation level \(qS/(q+2)\),
gives \eqref{eq:weighted_concentration_high_moment}.
\end{proof}

\begin{lemma}[Martingale Bernstein concentration under Sub-Weibull tails]
\label{lem:weighted_concentration_subweibull}
In the common centered setup above, suppose
\(\mathbb E_i\exp(|\widetilde X_i|^q)\le e^\upsilon\) for all
\(i\le t\), with \(\upsilon>0\) and \(1<q\le2\).
If deterministic \(B_2,b_{\max}\ge0\) satisfy
\(\sum_{i=1}^t b_i^2\le B_2\) and
\(\max_{i\le t}b_i\le b_{\max}\) almost surely, then
\begin{equation}\label{eq:weighted_concentration_subweibull}
\mathbb P\left\{\max_{1\le n\le t}
\sum_{i=1}^n b_i\widetilde X_i>S\right\}
\le\exp\!\left(-\frac{S^2}{4(e^\upsilon-1)B_2+2b_{\max}S}\right).
\end{equation}
The right-hand side is understood as zero when \(B_2=b_{\max}=0\).
\end{lemma}
\begin{proof}
When \(B_2b_{\max}=0\), all weights vanish and the claim is immediate.
Otherwise, \(x^k\le k!(e^{x^q}-1)\) for \(x\ge0\) and integers
\(k\ge2\): use \(x^k\le x^q\) for \(x\le1\), and
\(x^k/k!\le e^x-1\le e^{x^q}-1\) for \(x\ge1\).
Thus \(\mathbb E_i|\widetilde X_i|^k\le k!(e^\upsilon-1)\).
Conditional centering and the exponential series give, for
\(0<s<b_{\max}^{-1}\),
\begin{equation}\label{eq:subweibull_conditional_mgf}
\log\mathbb E_i e^{sb_i\widetilde X_i}
\le\frac{(e^\upsilon-1)s^2b_i^2}{1-sb_{\max}}.
\end{equation}
Consequently,
\(\exp\{s\sum_{i=1}^n b_i\widetilde X_i
-(e^\upsilon-1)s^2(1-sb_{\max})^{-1}\sum_{i=1}^n b_i^2\}\)
is a nonnegative supermartingale starting at one. At a crossing of
the centered sum above \(S\), its exponent is at least
\(sS-(e^\upsilon-1)s^2B_2/(1-sb_{\max})\).
Choose \(s=S/[2(e^\upsilon-1)B_2+b_{\max}S]\) and apply Ville's
inequality to obtain \eqref{eq:weighted_concentration_subweibull}.
\end{proof}

\begin{corollary}[Concentration with a random total weight]
\label{lem:sfo_weighted_boundary}
Let \(\{X_i\}_{i\ge1}\) be an integrable adapted sequence and let
\(0\le b_i\le1\) be predictable weights. For a cumulative-weight
threshold \(B>0\) and deviation level \(r>0\), consider
\begin{equation}\label{eq:sfo_weighted_boundary_probability}
\textstyle\mathbb P\left\{\exists n\ge1:\
\sum_{i=1}^n b_i\ge B,\
\sum_{i=1}^n b_i\widetilde X_i>r\sum_{i=1}^n b_i\right\}.
\end{equation}
The following are upper bounds for this probability; parts
\textnormal{(i)}--\textnormal{(ii)} require \(B\ge1\).
\begin{enumerate}[label=\textnormal{(\roman*)},leftmargin=*]
\item If \(\mathbb E_i|\widetilde X_i|^q\le M_q\) for \(1<q\le2\)
and a deterministic \(M_q\ge0\), an upper bound is
\begin{equation}\label{eq:sfo_weighted_boundary_moment}
\frac{6M_q}{(1-2^{1-q})r^qB^{q-1}}.
\end{equation}
\item If \(\mathbb E_i|\widetilde X_i|^q\le M_q\) for \(q>2\)
and a deterministic \(M_q\ge0\), an upper bound is
\begin{equation}\label{eq:sfo_weighted_boundary_high_moment}
2\exp\!\left(-\frac{2r^2B}{3(q+2)^2e^qM_q^{2/q}}\right)
+\frac{6(1+2/q)^qM_q}{(1-2^{1-q})r^qB^{q-1}}.
\end{equation}
In both finite-moment cases, the probability is zero when \(M_q=0\).
\item If \(\mathbb E_i\exp(|\widetilde X_i|^q)\le e^\upsilon\) for
\(1<q\le2\) and \(\upsilon>0\), an upper bound is
\begin{equation}\label{eq:sfo_weighted_boundary_subweibull}
\exp\!\left(-\frac{r^2B}{4(e^\upsilon-1)+2r}\right).
\end{equation}
This bound also holds with \(>\) replaced by \(\ge\) in
\eqref{eq:sfo_weighted_boundary_probability}.
\end{enumerate}
\end{corollary}
\begin{proof}
For the finite-moment cases, fix deterministic \(N\) and \(H>0\).
The stopped increments
\(Z_i\coloneqq b_i\widetilde X_i\mathbf1\{\sum_{j<i}b_j\le H\}\),
\(i\le N\), are martingale differences. Their total active weight
is at most \(H+1\), including the increment that first crosses \(H\).
For \(1<q\le2\),
\eqref{eq:weighted_concentration_moment}, with \(K_q\le2\), gives
\(\mathbb P\{\max_{n\le N}\sum_{i=1}^n Z_i>x\}
\le2M_q(H+1)/x^q\).
For \(q>2\), the conditional variance and \(q\)-moment sums are
bounded by \(M_q^{2/q}(H+1)\) and \(M_q(H+1)\), respectively;
apply
\eqref{eq:weighted_concentration_high_moment}.

Partition the possible total weights into
\([2^jB,2^{j+1}B)\), \(j\ge0\). On each interval, take
\(H=2^{j+1}B\) and \(x=r2^jB\), so that \(H+1\le3\cdot2^jB\).
Summing the polynomial terms over \(j\) gives the factor
\((1-2^{1-q})^{-1}\).
If the first high-moment exponential term is at most \(1/2\), a geometric-series
comparison bounds the sum by twice that term. Otherwise, the claimed
probability bound is at least one.
Letting \(N\to\infty\) proves the first two assertions.

For the Sub-Weibull case, \eqref{eq:subweibull_conditional_mgf} with
\(b_{\max}=1\), together with \(b_i^2\le b_i\), shows that
\(\exp\{s\sum_{i=1}^n b_i\widetilde X_i
-(e^\upsilon-1)s^2(1-s)^{-1}\sum_{i=1}^n b_i\}\)
is a nonnegative supermartingale for \(0<s<1\).
Choose \(s=r/[2(e^\upsilon-1)+r]\).
On the boundary event of part \textnormal{(iii)}, including equality, its
value is at least \(\exp\{r^2B/[4(e^\upsilon-1)+2r]\}\).
Ville's inequality gives the third assertion.
\end{proof}

\subsection{Complexity results}

\begin{proof}[\textbf{Proof of \Cref{thm:mu_pos_sfo_bzo}}]
\phantomsection\label{proof:mu_pos_sfo_bzo}
\noindent\emph{Step 1: One-step bound.}
The bound \eqref{eq:momentum_ratio_invariant} and \eqref{Energy} give
\begin{equation}\label{equ_coercive}
\Phi_i\ge\frac{\mu}{2}
\bigl(\|\bm{x}_i-\bm{x}^\star\|^2+\|\bm{z}_i-\bm{x}^\star\|^2\bigr).
\end{equation}
Strong convexity bounds the objective term below by
\(\mu\|\bm{x}_i-\bm{x}^\star\|^2/2\), while
\(\alpha_{i-1}^2/\gamma_{i-1}>
2\theta_{\mathrm{m}}(1-\vartheta_{\mathrm{m}})^2\mu\) gives the same bound
for the distance term. This also proves \eqref{equ_coercive} for
the \textsc{cFO} configuration below.

To cover both oracle models with one direction estimate, retain general
\(\vartheta_{\mathrm{m}}\in[0,1)\). On an accepted trial, suppress hats. Substituting
\eqref{eq:yhat_zt} gives
\[
\begin{aligned}
&\frac{1-\vartheta_{\mathrm{m}}}{\alpha_i}
\left[(1-\alpha_i)(\bm{y}_i-\bm{x}_i)
+\frac{\alpha_i}{1-\vartheta_{\mathrm{m}}}(\bm{y}_i-\bm{x}^\star)\right]\\
&\qquad=\bm{z}_i-\bm{x}^\star+
\frac{(1-\vartheta_{\mathrm{m}})(1-\alpha_i)\beta_i}
{(1-\vartheta_{\mathrm{m}})(1-\alpha_i)+\alpha_i(1-\beta_i)}
(\bm{x}_i-\bm{z}_i).
\end{aligned}
\]
The displayed fraction belongs to \([0,1]\), since its denominator
minus its numerator is
\((1-\beta_i)[(1-\vartheta_{\mathrm{m}})(1-\alpha_i)+\alpha_i]>0\).
Thus the right-hand side is a convex combination of
\(\bm{x}_i-\bm{x}^\star\) and \(\bm{z}_i-\bm{x}^\star\). By \eqref{equ_coercive},
\begin{equation}\label{eq:sc_error_direction_bound}
\left\|(1-\alpha_i)(\bm{y}_i-\bm{x}_i)
+\frac{\alpha_i}{1-\vartheta_{\mathrm{m}}}(\bm{y}_i-\bm{x}^\star)\right\|
\le\frac{\alpha_i}{1-\vartheta_{\mathrm{m}}}\sqrt{\frac{2\Phi_i}{\mu}}.
\end{equation}
Specializing to \(\vartheta_{\mathrm{m}}=0\), Cauchy--Schwarz bounds the
oracle inner product by \(\alpha_i\sqrt{2\Phi_i/\mu}\,W_i\).
Applying \Cref{cor:energy_without_check_ii}, the bound \(D_i\le\gamma_i\epsilon_{\mathrm f}\le\gamma_i\widetilde\epsilon_{\mathrm f}\) from \(\textbf{bZO(}\epsilon_{\mathrm f}\textbf{)}\), and \eqref{eq:tolerances_mu_sfo} gives
\begin{equation}\label{eq:regulated_one_step_mu_sfo}
\begin{aligned}
\Phi_{i+1}
&\le(1-\hat\alpha_i\Theta_i)\Phi_i
+\hat\alpha_i\Theta_i\sqrt{2\Phi_i/\mu}\,W_i+\bigl(
\tau_{\mathrm{des}}^{(i)}+\hat\gamma_i\widetilde\epsilon_{\mathrm{f}}
-(\theta-\theta_{\mathrm{m}})\hat\gamma_i\|\mathbf{G}_i\|^2\bigr)\Theta_i\\
&\le(1-\hat\alpha_i\Theta_i)\Phi_i
+\hat\alpha_i\Theta_i\sqrt{2\Phi_i/\mu}\,W_i
+(\tau_{\mathrm{des}}^{(i)}+\hat\gamma_i\widetilde\epsilon_{\mathrm{f}})
(\Theta_i-\Lambda_i).
\end{aligned}
\end{equation}
Both lines hold on rejected trials as well. The second uses
\(\tau_{\mathrm{des}}^{(i)}+\hat\gamma_i\widetilde\epsilon_{\mathrm{f}}
<(\theta-\theta_{\mathrm{m}})\hat\gamma_i\epsilon_{\mathrm{rel}}^2\)
on reliable trials.

For \(i<T_\varepsilon\), compatibility gives \(\Phi_i>\varepsilon\).
Moreover, the stopping rule and the strong-convexity inequality
\(\phi(\bm{y})-\phi^\star\le\|\nabla\phi(\bm{y})\|^2/(2\mu)\) give
\begin{equation}\label{eq:sc_pre_stop_gradient}
i<T_\varepsilon\quad\Longrightarrow\quad
\phi(\hat {\bm{y}}_i)-\phi^\star>\varepsilon,
\qquad\|\nabla\phi(\hat {\bm{y}}_i)\|>\sqrt{2\mu\varepsilon}.
\end{equation}
If the trial is accepted but unreliable, this implies
\(W_i>\sqrt{2\mu\varepsilon}-\epsilon_{\mathrm{rel}}
\ge\sqrt{\mu\varepsilon/2}\), since
\(\varepsilon>2\epsilon_{\mathrm{rel}}^2/\mu\).
On accepted trials, \eqref{eq:momentum_ratio_invariant} gives
\(\hat\gamma_i<\hat\alpha_i^2/(2\theta_{\mathrm{m}}\mu)
\le\hat\alpha_i/(2\theta_{\mathrm{m}}\mu)\). Hence
\[
\frac{(\tau_{\mathrm{des}}^{(i)}+\hat\gamma_i\widetilde\epsilon_{\mathrm{f}})
(\Theta_i-\Lambda_i)}{\Phi_i}
\le\frac{\sqrt2(\theta-\theta_{\mathrm{m}})}{4\theta_{\mathrm{m}}}
\frac{\hat\alpha_iW_i}{\sqrt{\mu\varepsilon}}.
\]
Divide \eqref{eq:regulated_one_step_mu_sfo} by $\Phi_i$ and replace $\Theta_i$ by one in the nonnegative error term, leaving the predictable weight $\hat\alpha_i$ in front of $W_i$. Applying $1+x\le e^x$ gives
\begin{equation}\label{eq:affine_module_mu_sfo}
\Phi_{i+1}\le V_i\Phi_i,\qquad
V_i\coloneqq \exp\!\left(-\hat\alpha_i\Theta_i+\hat\alpha_i
\frac{(\theta+3\theta_{\mathrm{m}})}
{2\sqrt2\theta_{\mathrm{m}}}\frac{W_i}{\sqrt{\mu\varepsilon}}\right),\qquad R_i\coloneqq 0,
\quad i<T_\varepsilon.
\end{equation}

\par\smallskip
\noindent\emph{Step 2: Small-step acceptance and accepted weights.}
The first inequality in \eqref{eq:tolerances_mu_sfo} implies
\(\epsilon_{\mathrm{rel}}>\widetilde\epsilon_{\mathrm{g}}\), since
\(2(1-\theta)(\theta-\theta_{\mathrm{m}})<1\).
On a true trial before stopping, \eqref{eq:sc_pre_stop_gradient} gives
\(\|\mathbf{G}_i\|>\sqrt{2\mu\varepsilon}
-\widetilde\epsilon_{\mathrm{g}}>\epsilon_{\mathrm{rel}}\), so condition~\eqref{Check_3}
holds. Independently of strong convexity or stopping, every trial with \(W_i\le\widetilde\epsilon_{\mathrm g}\) and \(\hat\gamma_i\le(1-\theta)/L\) satisfies, by smoothness, \eqref{eq:ZO_bound}, and Young's inequality,
\begin{equation}\label{eq:sfo_small_step_descent}
\begin{aligned}
f_i(\hat {\bm{x}}_{i+1})-f_i(\hat {\bm{y}}_i)
&\le-\hat\gamma_i\left((1-L\hat\gamma_i/2)\|\mathbf{G}_i\|^2
-\widetilde\epsilon_{\mathrm{g}}\|\mathbf{G}_i\|
-\widetilde\epsilon_{\mathrm{f}}\right)\\
&\le-\theta\hat\gamma_i\|\mathbf{G}_i\|^2+\tau_{\mathrm{des}}^{(i)}.
\end{aligned}
\end{equation}
Here the last step uses
\(\widetilde\epsilon_{\mathrm{g}}\|\mathbf{G}_i\|
\le(1-\theta)\|\mathbf{G}_i\|^2/2
+\widetilde\epsilon_{\mathrm{g}}^2/[2(1-\theta)]\) and the descent relaxation in \eqref{eq:tolerances_mu_sfo}.
Since \(\bar\gamma\le(1-\theta)/L\), this proves the small-step
implication in \eqref{two_proposition}.
The initialization also gives
\(\bar\gamma\le\gamma_0/\nu_{\mathrm{inc}}<\gamma_{\max}=1/(2\mu)\).
By \eqref{eq:true_probability} and the calibration in
\eqref{eq:tolerances_mu_sfo}, \(\mathbb P_i(I_i=1)\ge p>p_0\)
at every trial. By \eqref{eq:momentum_ratio_invariant} and
\eqref{eq:good_step_size_lower}, on a good step, \(\alpha_i>
\sqrt{2\theta_{\mathrm{m}}\mu\gamma_i}\ge C_\alpha\).
Thus \Cref{lem:good_step_growth} yields
\begin{equation}\label{eq:sfo_accepted_weight_tail}
\mathbb P\left\{t<T_\varepsilon,\
\sum_{i=1}^t\Theta_i\alpha_i<
C_\alpha\bar N_{\mathrm{good}}(t)\right\}
\le\exp\!\left(-\frac{(p-\hat p)^2}{2p}t\right).
\end{equation}

\par\smallskip
\noindent\emph{Step 3: Accumulated oracle errors.}
On \(\{t<T_\varepsilon\}\), if
\(\sum_{i=1}^t\hat\alpha_i\widetilde W_i
\le(\widetilde\epsilon_{\mathrm{g}}-\epsilon_{\mathrm{g}})\sum_{i=1}^t\hat\alpha_i\), the bound \(\mathbb E_i[W_i]\le\epsilon_{\mathrm g}\) from \(\textbf{sFO(}\epsilon_{\mathrm g},\upsilon,q\textbf{)}\) gives
\(\sum_{i=1}^t\hat\alpha_iW_i
\le\widetilde\epsilon_{\mathrm{g}}\sum_{i=1}^t\hat\alpha_i\).
By the target condition and the definition of \(C_{\mathrm{s}}\),
\[
\frac{(\theta+3\theta_{\mathrm{m}})}
{2\sqrt2\theta_{\mathrm{m}}}\frac{\widetilde\epsilon_g}{\sqrt{\mu\varepsilon}}
\le\frac{(1-\bar\alpha)(1-\nu_{\mathrm{dec}})}{4\sqrt{\nu_{\mathrm{inc}}}},
\]
where \(\bar\alpha\) is the bound in \Cref{para_lemma_d}.
Using \Cref{para_lemma_e} in \eqref{eq:affine_module_mu_sfo} now gives
\begin{equation}\label{eq:sfo_product_control}
\log\prod_{i=1}^t V_i
\le-\sum_{i=1}^t\Theta_i\alpha_i
+\frac{(1-\bar\alpha)(1-\nu_{\mathrm{dec}})}{4\sqrt{\nu_{\mathrm{inc}}}}
\sum_{i=1}^t\hat\alpha_i
\le\frac12\left(1-\sum_{i=1}^t\Theta_i\alpha_i\right).
\end{equation}

For the abstract theorem, take
\(H(t)\coloneqq \exp\{[1-C_\alpha\bar N_{\mathrm{good}}(t)]/2\}\) and
\(\bar T\coloneqq \lceil(1-p_0)[N_0+1/C_\alpha]/
[p_0(\hat p-p_0)]\rceil\).
For \(t\ge\bar T\), \(C_\alpha\bar N_{\mathrm{good}}(t)\ge1\). The finite-moment condition~\eqref{eq:SFO_moment} gives \(\mathbb E_i[|\widetilde W_i|^q]\le\upsilon\). Since the weights \(\hat\alpha_i\) are predictable and bounded by one, \Cref{lem:sfo_weighted_boundary}(i) applies with \(b_i=\hat\alpha_i\), \(X_i=W_i\), \(M_q=\upsilon\), \(B=C_\alpha\overline N_{\mathrm{good}}(t)\), and \(r=\widetilde\epsilon_{\mathrm g}-\epsilon_{\mathrm g}\).
By \eqref{eq:sfo_product_control},
\(t<T_\varepsilon\) and \(\prod_{i=1}^tV_i>H(t)\) require failure of either
the accepted-weight lower bound in \eqref{eq:sfo_accepted_weight_tail}
or the centered-error bound at the start of this step with
\(\sum_{i=1}^t\hat\alpha_i\ge
\sum_{i=1}^t\Theta_i\alpha_i\ge B\).
The latter fixed-horizon event lies in the boundary event of
\Cref{lem:sfo_weighted_boundary}. A union bound gives
\begin{equation}\label{eq:abstract_tail_mu_sfo}
\mathbb P\left\{t<T_\varepsilon,\ \prod_{i=1}^t V_i>H(t)\right\}
\le\exp\!\left(-\frac{(p-\hat p)^2}{2p}t\right)
+\frac{6(1-p_0)}
{(1-2^{1-q})[C_\alpha\bar N_{\mathrm{good}}(t)]^{q-1}},
\end{equation}
using \(\upsilon/(\widetilde\epsilon_{\mathrm{g}}-\epsilon_{\mathrm{g}})^q<1-p_0\).

Set \(P(t)\) to the minimum of one and the right-hand side of
\eqref{eq:abstract_tail_mu_sfo}, and let
\(\varepsilon_{\mathrm{res}}\coloneqq 0\).
Since \(R_i=0\), the residual event in
\eqref{eq:abstract_tail_event} is empty.
The multipliers \(V_i\) are nonnegative, \(H(t)\) is nonincreasing,
and compatibility follows from
the argument in \Cref{subsec:frame}.
The bound on \(t\) in \eqref{eq:complexity_mu_sfo} ensures both
\(t\ge\bar T\) and \(H(t)\Phi_0\le\varepsilon\).
Applying \Cref{Big_threorem} proves \eqref{eq:probability_mu_sfo}.
\end{proof}

\begin{proof}[\textbf{Proof of \Cref{cor:mu_pos_sfo_tail_refinements}}]
\phantomsection\label{proof:mu_pos_sfo_tail_refinements}
The retained assumptions give the same multiplier
\eqref{eq:affine_module_mu_sfo} and product bound
\eqref{eq:sfo_product_control}. In each branch, the calibration and
\eqref{eq:true_probability} give \(\mathbb P_i(I_i=1)\ge p>p_0\), so
\eqref{eq:sfo_accepted_weight_tail} also holds with that branch's
\(p,\hat p\), and \(\bar N_{\mathrm{good}}(t)\).
Use \(H(t)\) and the choices of \(b_i,X_i,B,r\) from Step 3 of the
proof of \Cref{thm:mu_pos_sfo_bzo} with these values. The weights remain predictable
and bounded by one, and the trial bound again gives
\(B=C_\alpha\bar N_{\mathrm{good}}(t)\ge1\).

For the finite-moment branch of \(\textbf{sFO(}\epsilon_{\mathrm g},\upsilon,q\textbf{)}\) with \(q>2\), condition~\eqref{eq:SFO_moment} allows us to apply~\eqref{eq:sfo_weighted_boundary_high_moment} with \(M_q=\upsilon\).
The calibration in \eqref{eq:tolerances_mu_sfo} gives
\(\upsilon/(\widetilde\epsilon_{\mathrm{g}}-\epsilon_{\mathrm{g}})^q<1-p_0\).
Substitution gives the constants \(c,C\) in
\Cref{tab:sfo_tail_constants} and bounds the oracle-error probability
by the last two terms in \eqref{eq:probability_mu_sfo_high_moment}.

For the Sub-Weibull branch, condition~\eqref{eq:SFO_subweibull} allows us to use \eqref{eq:sfo_weighted_boundary_subweibull} to obtain the oracle-error term
\(\exp\{-(\widetilde\epsilon_{\mathrm{g}}-\epsilon_{\mathrm{g}})^2
C_\alpha\bar N_{\mathrm{good}}(t)/
[4(e^\upsilon-1)+2(\widetilde\epsilon_{\mathrm{g}}-\epsilon_{\mathrm{g}})]\}\).
The coefficient \(r^2/[4(e^\upsilon-1)+2r]\) is increasing for
\(r>0\). Substituting the calibrated lower bound
\([\ln(e^\upsilon/(1-p_0))]^{1/q}\) for \(r\) gives the last term
in \eqref{eq:probability_mu_sfo_subweibull}.
In each case, combine the oracle-error bound with
\eqref{eq:sfo_accepted_weight_tail} and take \(P(t)\) to be the minimum
of one and their sum. The same exceptional-event inclusion and
application of \Cref{Big_threorem}, with
\(R_i=\varepsilon_{\mathrm{res}}=0\), give
\eqref{eq:probability_mu_sfo_high_moment} and
\eqref{eq:probability_mu_sfo_subweibull}, respectively.
\end{proof}

\begin{proof}[\textbf{Proof of \Cref{thm:mu_pos_cfo_bzo}}]~\,
\phantomsection\label{proof:mu_pos_cfo_bzo}
If \(\Phi_0\le\varepsilon\), compatibility gives \(T_\varepsilon=1\).
Assume henceforth that \(\Phi_0>\varepsilon\).

\par\medskip
\noindent\emph{Step 1: One-step Lyapunov multiplier.}
For \(t<T_\varepsilon\), a rejected trial has \(\Phi_{t+1}=\Phi_t\).
Consider an accepted trial, suppress hats, and set
\[
A\coloneqq \frac{\widetilde\epsilon_{\mathrm{g}}^2}{2(1-\theta)}
+2\widetilde\epsilon_{\mathrm{f}},
\qquad
\chi_\kappa\coloneqq 1-
\left(\frac{\kappa_{\mathrm{g}}}{\widetilde\kappa_{\mathrm{g}}}\right)^2,
\qquad
 \bm{d}_t\coloneqq (1-\alpha_t)(\bm{y}_t-\bm{x}_t)
 +\frac{\alpha_t}{1-\vartheta_{\mathrm{m}}}(\bm{y}_t-\bm{x}^\star).
\]
Here \(0<\chi_\kappa\le1\), since
\(0\le\kappa_{\mathrm{g}}<\widetilde\kappa_{\mathrm{g}}\).
The bounds in \(\textbf{cFO(}\epsilon_{\mathrm g},\kappa_{\mathrm g},\delta_{\mathrm g},\epsilon_{\mathrm c}\textbf{)}\) and the definition of \(I_t\) give
\[
W_t\le\epsilon_{\mathrm g}+I_t\kappa_{\mathrm g}\gamma_t\|\mathbf G_t\|+(1-I_t)\epsilon_{\mathrm c}\qquad a.s.
\]
The oracle \(\textbf{bZO(}\epsilon_{\mathrm f}\textbf{)}\) gives \(D_t\le\gamma_t\epsilon_{\mathrm f}\le\gamma_t\widetilde\epsilon_{\mathrm f}\). Substitute \(C_t\) into \Cref{cor:energy_without_check_ii} and apply these bounds, \eqref{eq:tolerances_mu_cfo}, and the definition of \(\Delta_{\mathrm c}\).
Splitting the contribution proportional to \(\epsilon_{\mathrm g}\)
along the two components of \(\bm{d}_t\), applying Cauchy--Schwarz,
and discarding
\(-(1-I_t)\Delta_{\mathrm{c}}\gamma_t\|\mathbf{G}_t\|^2\le0\) give
\[
\begin{aligned}
\Phi_{t+1}
&\le
(1-\alpha_t)\Phi_t
+(1-\alpha_t)
\left(
\epsilon_{\mathrm{g}}\|\bm{y}_t-\bm{x}_t\|
-\frac\mu2\|\bm{y}_t-\bm{x}_t\|^2
\right)\\
&\quad+
\frac{\alpha_t}{1-\vartheta_{\mathrm{m}}}
\left[
\epsilon_{\mathrm{g}}\|\bm{y}_t-\bm{x}^\star\|
-\vartheta_{\mathrm{m}}
\left(
\phi(\bm{y}_t)-\phi^\star
+\frac\mu2\|\bm{y}_t-\bm{x}^\star\|^2
\right)
\right]\\
&\quad+
I_t\left[
-\Delta_{\mathrm{c}}\gamma_t\|\mathbf{G}_t\|^2
+\kappa_{\mathrm{g}}\gamma_t\|\mathbf{G}_t\|\,\|\bm{d}_t\|
\right]
+(1-I_t)\epsilon_{\mathrm{c}}\|\bm{d}_t\|
+A\gamma_t(1-\Lambda_t).
\end{aligned}
\]
The terms involving \(\epsilon_{\mathrm g}\) satisfy
\[
\begin{aligned}
&(1-\alpha_t)
\left(
\epsilon_{\mathrm{g}}\|\bm{y}_t-\bm{x}_t\|
-\frac\mu2\|\bm{y}_t-\bm{x}_t\|^2
\right)
\le
(1-\alpha_t)\frac{\epsilon_{\mathrm{g}}^2}{2\mu},\\
&\frac{\alpha_t}{1-\vartheta_{\mathrm{m}}}
\left[
\epsilon_{\mathrm{g}}\|\bm{y}_t-\bm{x}^\star\|
-\vartheta_{\mathrm{m}}
\left(
\phi(\bm{y}_t)-\phi^\star
+\frac\mu2\|\bm{y}_t-\bm{x}^\star\|^2
\right)
\right]\le
\frac{\alpha_t}{1-\vartheta_{\mathrm{m}}}
\left[
\frac{\epsilon_{\mathrm{g}}^2}{2\mu\vartheta_{\mathrm{m}}}
-\vartheta_{\mathrm{m}}\varepsilon
\right]<0.
\end{aligned}
\]
The second line is absent when \(\epsilon_{\mathrm{g}}=\vartheta_{\mathrm{m}}=0\);
otherwise it follows from Young's inequality,
\eqref{eq:sc_pre_stop_gradient}, and
the first requirement in \eqref{eq:precision_mu_cfo}.

The shared estimate \eqref{eq:sc_error_direction_bound} gives
\(\|\bm{d}_t\|\le\alpha_t(1-\vartheta_{\mathrm{m}})^{-1}
\sqrt{2\Phi_t/\mu}\).
On a true accepted trial, Young's inequality therefore gives
\[
\begin{aligned}
&-\Delta_{\mathrm{c}}\gamma_t\|\mathbf{G}_t\|^2
+\kappa_{\mathrm{g}}\gamma_t\|\mathbf{G}_t\|\,\|\bm{d}_t\|\le
\frac{\kappa_{\mathrm{g}}^2\gamma_t}{4\Delta_{\mathrm{c}}}\|\bm{d}_t\|^2
\le
\alpha_t
\left(\frac{\kappa_{\mathrm{g}}}{\widetilde\kappa_{\mathrm{g}}}\right)^2
\Phi_t.
\end{aligned}
\]
The last inequality uses
\(\gamma_t\alpha_t\le\gamma_t\le\gamma_{\max}=1/(2\mu)\) and
\(\widetilde\kappa_{\mathrm{g}}
\le2\mu\sqrt{\Delta_{\mathrm{c}}}(1-\vartheta_{\mathrm{m}})\).
Since \(1-\alpha_t\le1-\chi_\kappa\alpha_t\), the base term and the
term multiplied by \(I_t\) are together bounded by
\((1-\chi_\kappa\alpha_t)\Phi_t\) for either value of \(I_t\).

To normalize the remaining terms, \Cref{para_lemma_d} and
\eqref{eq:momentum_ratio_invariant} give, on every accepted trial,
\[
\begin{aligned}
\frac{\alpha_t}{1-\chi_\kappa\alpha_t}
&\le\frac{\alpha_t}{1-\alpha_t}
\le\frac{1+\nu_{\mathrm{inc}}}{1-\theta_{\mathrm{m}}},\\
\frac{\gamma_t}{1-\chi_\kappa\alpha_t}
&<\frac{\alpha_t^2}
{2\theta_{\mathrm{m}}(1-\vartheta_{\mathrm{m}})^2\mu
(1-\chi_\kappa\alpha_t)}
\le\frac{1+\nu_{\mathrm{inc}}}
{2\theta_{\mathrm{m}}(1-\vartheta_{\mathrm{m}})^2\mu(1-\theta_{\mathrm{m}})}.
\end{aligned}
\]
Since \(\Phi_t>\varepsilon\), the direction bound and the definitions
of \(C_1,C_2\) now give
\[
\begin{aligned}
\epsilon_{\mathrm{c}}\|\bm{d}_t\|
&\le(1-\chi_\kappa\alpha_t)
\frac{C_2}{\sqrt\varepsilon}\Phi_t,\\
(1-\alpha_t)\frac{\epsilon_{\mathrm{g}}^2}{2\mu}
+A\gamma_t(1-\Lambda_t)
&\le(1-\chi_\kappa\alpha_t)C_1
\le(1-\chi_\kappa\alpha_t)\frac{C_1}{\varepsilon}\Phi_t.
\end{aligned}
\]
Here \((1-\vartheta_{\mathrm{m}})^{-1}
=(2\epsilon_{\mathrm{rel}}+\widetilde\epsilon_{\mathrm{g}})/(2\epsilon_{\mathrm{rel}})\)
gives \(C_2\), while \(1-\alpha_t\le1-\chi_\kappa\alpha_t\)
and \(\epsilon_{\mathrm{g}}^2\le2\widetilde\epsilon_{\mathrm{g}}^2+\widetilde\epsilon_{\mathrm{f}}\)
give the bound with \(C_1\).
Thus every accepted trial satisfies
\[
\Phi_{t+1}
\le
(1-\chi_\kappa\alpha_t)
\left(1+\frac{C_1}{\varepsilon}\right)
R_\phi^{1-I_t}\Phi_t,
\]
because
\((1+C_1/\varepsilon)(R_\phi-1)=C_2/\sqrt\varepsilon\).
Together with the rejected-trial identity, this gives, for all trial
outcomes,
\begin{equation}\label{eq:affine_module_mu_cfo}
\begin{aligned}
\Phi_{t+1}\le V_t\Phi_t,\qquad
V_t\coloneqq 
(1-\chi_\kappa\alpha_t\Theta_t)
\left(1+\frac{C_1}{\varepsilon}\right)^{\Theta_t}
R_\phi^{\Theta_t(1-I_t)}
,
\qquad R_t\coloneqq 0.
\end{aligned}
\end{equation}

\par\medskip
\noindent\emph{Step 2: Small-step acceptance and reliability.}
\par\nobreak
On \(\{I_t=1,U_t=0,t<T_\varepsilon\}\), smoothness,
\textsc{bZO}, and the true-trial bound give
\[
\begin{aligned}
f_t(\hat {\bm{x}}_{t+1})-f_t(\hat {\bm{y}}_t)
&\le-\hat\gamma_t[1-(L+2\kappa_{\mathrm{g}})\hat\gamma_t/2]
\|\mathbf{G}_t\|^2+\hat\gamma_t\widetilde\epsilon_{\mathrm{g}}\|\mathbf{G}_t\|
+\hat\gamma_t\widetilde\epsilon_{\mathrm{f}}.
\end{aligned}
\]
Apply the Young estimate in \eqref{eq:sfo_small_step_descent} with
\(L\) replaced by \(L+2\kappa_{\mathrm{g}}\). The same relaxation rule and
\(\hat\gamma_t\le\bar\gamma\le(1-\theta)/(L+2\kappa_{\mathrm{g}})\)
give condition~\eqref{Check_1}.
Moreover, \(\|\mathbf{G}_t\|<\epsilon_{\mathrm{rel}}\) would imply
\[
\|\nabla\phi(\hat {\bm{y}}_t)\|
<\epsilon_{\mathrm{g}}+(1+\kappa_{\mathrm{g}}\bar\gamma)\epsilon_{\mathrm{rel}}
<\sqrt{2\mu\varepsilon},
\]
contrary to \eqref{eq:sc_pre_stop_gradient}. Thus condition~\eqref{Check_3}
also holds and
\((1-U_t)I_t\le\Lambda_t\). The \textsc{cFO} branch of
\eqref{eq:true_probability} gives
\(\mathbb P_t(I_t=1)\ge p=1-\delta_{\mathrm{g}}>p_0\).

\par\medskip
\noindent\emph{Step 3: Cumulative multiplier and tail bound.}

First verify that the prescribed range for \(\hat p\) is nonempty.
Since \(1-p=\delta_{\mathrm{g}}\in[0,1]\), concavity gives
\[
\left(1+\frac{C_1}{\varepsilon}\right)R_\phi^{1-p}
\le1+\frac{C_1}{\varepsilon}
+\frac{\delta_{\mathrm{g}}C_2}{\sqrt\varepsilon}.
\]
The third condition in \eqref{eq:precision_mu_cfo} makes the right-hand
side smaller than
\(\exp\{C_\kappa^{\mathrm{c}}p_0(p-p_0)/(1-p_0)\}\).
Consequently, \eqref{eq:pprime_mu_cfo} gives
\(p'(\varepsilon)<p\). Also \(p'(\varepsilon)\ge p_0\), since
\(C_1\ge0\) and \(R_\phi\ge1\). 

Fix
\(\hat p\in(p'(\varepsilon),p)\). The true-count estimate \eqref{eq:true_trial_count_tail} gives
\(\mathbb P\{\sum_{i=1}^tI_i<\hat p t\}
\le\exp[-(p-\hat p)^2t/(2p)]\).
On
\(\{\sum_{i=1}^tI_i\ge\hat p t,\ t<T_\varepsilon\}\), the
contrapositive of \eqref{eq:good_step_event_inclusion} gives
\(N_{\mathrm{good}}(t)\ge\bar N_{\mathrm{good}}(t)\), while
\(\sum_{i=1}^t(1-I_i)\le(1-\hat p)t\). By
\eqref{eq:good_step_size_lower} and \eqref{eq:momentum_ratio_invariant},
every good step satisfies
\(\alpha_i>(1-\vartheta_{\mathrm{m}})
\sqrt{2\theta_{\mathrm{m}}\mu\bar\gamma/\nu_{\mathrm{inc}}}\).
Hence the definition of
\(C_\kappa^{\mathrm{c}}\) gives $\prod_{i=1}^t
(1-\chi_\kappa\alpha_i\Theta_i)
\le
\exp\{-C_\kappa^{\mathrm{c}}N_{\mathrm{good}}(t)\}$.

Set
$\bar T
\coloneqq \max\left\{1,
\left\lceil
\frac{(1-p_0)N_0}{p_0(\hat p-p_0)}
\right\rceil\right\}$. For \(t\ge\bar T\), the positive-part
operator in \(\bar N_{\mathrm{good}}(t)\) is inactive. Using
\(\sum_i\Theta_i\le t\),
\(\sum_i\Theta_i(1-I_i)\le\sum_i(1-I_i)\), and
\eqref{eq:affine_module_mu_cfo}, we obtain
\begin{align}
\prod_{i=1}^tV_i
&\le
\exp\{-C_\kappa^{\mathrm{c}}\bar N_{\mathrm{good}}(t)\}
\left(1+\frac{C_1}{\varepsilon}\right)^t
R_\phi^{(1-\hat p)t}\notag\\
&=
\exp\Bigg\{
C_\kappa^{\mathrm{c}}N_0
-
\left[
\frac{p_0}{1-p_0}C_\kappa^{\mathrm{c}}
+\ln R_\phi
\right]
\bigl(\hat p-p'(\varepsilon)\bigr)t
\Bigg\}
\eqqcolon H(t).
\label{eq:product_mu_cfo}
\end{align}
The equality follows directly from \eqref{eq:pprime_mu_cfo}.
Consequently,
\begin{equation}\label{eq:abstract_tail_mu_cfo}
\mathbb P\left\{
 t<T_\varepsilon,\
 \prod_{i=1}^tV_i>H(t)
\right\}
\le
\exp\left(-\frac{(p-\hat p)^2}{2p}t\right).
\end{equation}
Take \(P(t)\coloneqq\exp(-(p-\hat p)^2t/(2p))\) and
\(\varepsilon_{\mathrm{res}}\coloneqq0\).
The multipliers are nonnegative, and \(H(t)\) decreases since
\(\hat p>p'(\varepsilon)\). The zero-residual and compatibility
checks are as in the proof of \Cref{thm:mu_pos_sfo_bzo}.
For positive integers \(t\), \eqref{eq:complexity_mu_cfo} gives
\(t\ge\bar T\) and \(H(t)\Phi_0\le\varepsilon\), so
\Cref{Big_threorem} and \eqref{eq:abstract_tail_mu_cfo} prove
\eqref{eq:probability_mu_cfo}.
\end{proof}

\begin{proof}[\textbf{Proof of \Cref{rem:gc_trajectory}}]
\phantomsection\label{proof:gc_trajectory}
We first use the oracle and acceptance conditions to show that accepted
trials occur infinitely often almost surely. On
\(\{W_t\le\widetilde\epsilon_{\mathrm{g}}\}\),
\eqref{eq:sfo_small_step_descent} ensures condition~\eqref{Check_1} whenever
\(\hat\gamma_t\le(1-\theta)/L\). If condition~\eqref{Check_2} is
enabled, convexity and \textsc{bZO} give
\begin{equation}\label{eq:sfo_model_comparison}
f_t(\hat {\bm{y}}_t)-f_t(\bm{x}_t)
\le\langle\mathbf{G}_t,\hat {\bm{y}}_t-\bm{x}_t\rangle
+\widetilde\epsilon_{\mathrm{g}}\|\hat {\bm{y}}_t-\bm{x}_t\|
+\hat\gamma_t\widetilde\epsilon_{\mathrm{f}},
\end{equation}
which verifies it by \eqref{eq:model_tolerance_mu_zero_sfo}.
These acceptance implications hold at every trial,
since condition~\eqref{Check_3} regulates the next step size after
acceptance. By \eqref{eq:true_probability} and
\eqref{eq:tolerances_mu_zero_sfo},
\(\mathbb P_t\{W_t\le\widetilde\epsilon_{\mathrm{g}}\}\ge p>p_0>0\).

Fix an index \(n\). If all trials from \(n\) onward are rejected,
their step sizes satisfy
\(\hat\gamma_{n+k}=\nu_{\mathrm{dec}}^k\hat\gamma_n\).
Since \(\hat\gamma_n\le\gamma_{\max}\), a deterministic number of
these rejections suffices to enter the small-step range. Thereafter,
each rejection requires \(W_t>\widetilde\epsilon_{\mathrm{g}}\). Iterated
conditional expectation bounds the probability of \(m\) further
consecutive rejections by \((1-p)^m\). Letting \(m\to\infty\) shows
that the probability of rejection at every trial from \(n\) onward
is zero. A countable union over \(n\) proves the claim.

On the probability-one event where accepted trials occur infinitely
often and \eqref{eq:bounded_trajectory_mu_zero_sfo} holds, fix \(t\) and let
\(j\ge t\) be the next accepted trial. All intervening trials are
rejected, so \(\bm{z}_j=\bm{z}_t\). Therefore
\(\|\bm{z}_t-\bm{x}^\star\|\le\|\bm{z}_j-\bm{y}_j\|+\|\bm{y}_j-\bm{x}^\star\|
\le2\widehat\Omega_y\).
This holds for every \(t\), proving
\eqref{eq:bounded_auxiliary_mu_zero_sfo} with the stated constant.
\end{proof}

\begin{proof}[\textbf{Proof of \Cref{thm:mu_zero_sfo_bzo}}]
\phantomsection\label{proof:mu_zero_sfo_bzo}
\noindent\emph{Step 1: One-step affine recursion.}
Write \(A_0\coloneqq \widetilde\epsilon_{\mathrm{g}}^2/(2(1-\theta))
+3\widetilde\epsilon_{\mathrm{f}}\) and
\(\Delta_0\coloneqq (\theta-\theta_{\mathrm{m}})\epsilon_{\mathrm{rel}}^2-A_0\).
The calibration \eqref{eq:tolerances_mu_zero_sfo} gives
\begin{equation}\label{eq:regulated_general_convex_margin}
\Delta_0>\frac{\nu_{\mathrm{inc}}-1}{1-\nu_{\mathrm{dec}}}A_0.
\end{equation}
A rejected trial preserves \(\Phi_t\). On an accepted trial,
\(\bm{y}_t=(1-\alpha_t)\bm{x}_t+\alpha_t\bm{z}_t\), and condition~\eqref{Check_1}
holds under either convention for condition~\eqref{Check_2}.
In \eqref{eq:energy_without_check_ii} with
\(\mu=\vartheta_{\mathrm{m}}=0\), the error direction satisfies
\[
(1-\alpha_t)(\bm{y}_t-\bm{x}_t)+\alpha_t(\bm{y}_t-\bm{x}^\star)
=\alpha_t(\bm{z}_t-\bm{x}^\star).
\]
Thus its inner product with \(\nabla\phi(\bm{y}_t)-\mathbf{G}_t\)
is at most \(\alpha_t\widehat\Omega_zW_t\).
Moreover, \textsc{bZO} and the descent relaxation imply
\(D_t+\tau_{\mathrm{des}}^{(t)}
\le\gamma_t[\widetilde\epsilon_{\mathrm{g}}^2/(2(1-\theta))
+2\widetilde\epsilon_{\mathrm{f}}]\le A_0\gamma_t\).
If the trial is reliable,
\(\|\mathbf{G}_t\|\ge\epsilon_{\mathrm{rel}}\), and its remaining
step-size contribution is at most \(-\Delta_0\gamma_t\).
For an accepted but unreliable trial, it is at most \(A_0\gamma_t\).
Thus all three trial outcomes satisfy
\begin{equation}\label{eq:regulated_affine_modules_mu_zero}
\begin{aligned}
\Phi_{t+1}\le V_t\Phi_t+R_t,\qquad V_t\coloneqq 1-\alpha_t\Theta_t,\qquad
R_t\coloneqq \alpha_t\Theta_t\widehat\Omega_zW_t+
\hat\gamma_t[A_0(\Theta_t-\Lambda_t)-\Delta_0\Lambda_t].
\end{aligned}
\end{equation}

\par\smallskip
\noindent\emph{Step 2: Search conditions and the bound on \(\lambda_t\).}
On \(\{I_t=1,U_t=0,t<T_{\boldsymbol\varepsilon}\}\), one has
\(W_t\le\widetilde\epsilon_{\mathrm{g}}\) and
\(\hat\gamma_t\le\bar\gamma\le(1-\theta)/L\).
Equations~\eqref{eq:sfo_small_step_descent} and
\eqref{eq:sfo_model_comparison} therefore verify condition~\eqref{Check_1}
and, when enabled, condition~\eqref{Check_2}. Before termination,
\(\|\mathbf{G}_t\|\ge\|\nabla\phi(\hat {\bm{y}}_t)\|-W_t
>\varepsilon_\nabla-\widetilde\epsilon_{\mathrm{g}}>\epsilon_{\mathrm{rel}}\).
Thus condition~\eqref{Check_3} holds, proving
\((1-U_t)I_t\le\Lambda_t\). Equation~\eqref{eq:true_probability} gives
\(\mathbb P_t(I_t=1)\ge p>p_0\). Consequently,
\Cref{lem:good_step_contraction} and the definitions of
\(C_{\hat p},\bar T_{\mathrm{gc}}\) yield, for \(t\ge\bar T_{\mathrm{gc}}\),
\begin{equation}\label{eq:regulated_lambda_tail_mu_zero_sfo}
\mathbb P\{t<T_{\boldsymbol\varepsilon},\;
\lambda_t>C_{\hat p}t^{-2}\}
\le\exp\left(-\frac{(p-\hat p)^2}{2p}t\right).
\end{equation}

\par\smallskip
\noindent\emph{Step 3: Stochastic residual and step regulation.}
Fix $t\ge\bar T_{\mathrm{gc}}$ and write $r_i^{(t)}:=\prod_{j=i+1}^tV_j=\lambda_t/\lambda_i$. These weights depend on later acceptance decisions. We first bound $r_i^{(t)}\alpha_i\Theta_i$ by deterministic coefficients before applying concentration. Since $\mu=0$, \Cref{para_lemma_a} gives $\gamma_i\lambda_i/\alpha_i^2=\gamma_0/\alpha_0^2$. \Cref{para_lemma_b} also implies
\(\lambda_i^{-1/2}\le1+\zeta i\), where
\(\zeta\coloneqq \alpha_0\sqrt{\gamma_{\max}/\gamma_0}\).
Hence, on \(\{\lambda_t\le C_{\hat p}t^{-2}\}\), $r_i^{(t)}\alpha_i\Theta_i=
\lambda_t\alpha_0\sqrt{\gamma_i/(\gamma_0\lambda_i)}\,\Theta_i
\le\zeta C_{\hat p}(1+\zeta i)t^{-2}
\eqqcolon \bar a_i^{(t)}$.
These deterministic weights satisfy
\begin{equation}\label{eq:weights_mu_zero_sfo}
\sum_{i=1}^t\bar a_i^{(t)}
=\zeta C_{\hat p}\left(\frac1t+\frac{\zeta(t+1)}{2t}\right)
\le B_{\hat p},\qquad
\sum_{i=1}^t(\bar a_i^{(t)})^q
\le2^{q-1}B_{\hat p}^qt^{1-q}.
\end{equation}
For the second inequality, the largest weight is at most twice the
average weight. Apply \Cref{lem:weighted_concentration_unified} through
the uncentered form \eqref{eq:weighted_uncentered_conversion}, with
\(b_i=\bar a_i^{(t)}\), \(B_1=B_{\hat p}\),
\(B_q=2^{q-1}B_{\hat p}^qt^{1-q}\), and
\(S=(\widetilde\epsilon_{\mathrm{g}}-\epsilon_{\mathrm{g}})B_{\hat p}\).
Equations~\eqref{eq:SFO_bias}--\eqref{eq:SFO_moment} give
\begin{equation}\label{eq:regulated_weighted_error_mu_zero_sfo}
\mathbb P\left\{\sum_{i=1}^t\bar a_i^{(t)}W_i
>B_{\hat p}\widetilde\epsilon_{\mathrm{g}}\right\}
\le\frac{2\upsilon}
{(\widetilde\epsilon_{\mathrm{g}}-\epsilon_{\mathrm{g}})^q}\,t^{1-q}
\le\frac{2(1-p_0)}{t^{q-1}}.
\end{equation}
Here \(2^{q-1}K_q=2\), and the last inequality uses the calibration
of \(\widetilde\epsilon_{\mathrm{g}}\) in \Cref{ass:mu_zero_sfo_bzo}.
On the complementary event and
\(\{\lambda_t\le C_{\hat p}t^{-2}\}\), nonnegativity of \(W_i\)
and \(r_i^{(t)}\alpha_i\Theta_i\le\bar a_i^{(t)}\) imply
\begin{equation}\label{eq:regulated_stochastic_residual_bound}
\textstyle \widehat\Omega_z\sum_{i=1}^tr_i^{(t)}\alpha_i\Theta_i W_i
\le\widehat\Omega_zB_{\hat p}\widetilde\epsilon_{\mathrm{g}}.
\end{equation}

The step-size residual is bounded pathwise.
Since \(0<V_j\le1\),
\(0<r_1^{(t)}\le\cdots\le r_t^{(t)}=1\).
The update \eqref{eq:gamma-dyn} gives
\( (1-\nu_{\mathrm{dec}})\hat\gamma_i(1-\Lambda_i)
\le\hat\gamma_i-\hat\gamma_{i+1}
+(\nu_{\mathrm{inc}}-1)\hat\gamma_i\Lambda_i\).
Summation by parts gives
\[
\begin{aligned}
\textstyle\sum_{i=1}^tr_i^{(t)}(\hat\gamma_i-\hat\gamma_{i+1})
=r_1^{(t)}\hat\gamma_1+
\sum_{i=2}^t(r_i^{(t)}-r_{i-1}^{(t)})\hat\gamma_i
-\hat\gamma_{t+1}\le\gamma_{\max}.
\end{aligned}
\]
Consequently,
\begin{equation}\label{eq:regulated_weighted_step_budget}
\textstyle(1-\nu_{\mathrm{dec}})\sum_{i=1}^tr_i^{(t)}\hat\gamma_i(1-\Lambda_i)
\le\gamma_{\max}
+(\nu_{\mathrm{inc}}-1)\sum_{i=1}^tr_i^{(t)}\hat\gamma_i\Lambda_i.
\end{equation}
Using \(\Theta_i-\Lambda_i\le1-\Lambda_i\) and the definition
\eqref{eq:regulated_step_residual_mu_zero}, we obtain
\[
\begin{aligned}
\textstyle\sum_{i=1}^tr_i^{(t)}\hat\gamma_i
[A_0(\Theta_i-\Lambda_i)-\Delta_0\Lambda_i]\le\frac{A_0\gamma_{\max}}{1-\nu_{\mathrm{dec}}}+
\left[\frac{A_0(\nu_{\mathrm{inc}}-1)}{1-\nu_{\mathrm{dec}}}-\Delta_0\right]
\sum_{i=1}^tr_i^{(t)}\hat\gamma_i\Lambda_i
\le\varepsilon_{\mathrm{step}},
\end{aligned}
\]
where \eqref{eq:regulated_general_convex_margin} makes the coefficient
of the final sum negative, giving the last inequality.
Together with \eqref{eq:regulated_stochastic_residual_bound}, this shows
that the total propagated residual in
\eqref{eq:regulated_affine_modules_mu_zero} is at most
\(\widehat\Omega_zB_{\hat p}\widetilde\epsilon_{\mathrm{g}}
+\varepsilon_{\mathrm{step}}=\varepsilon_{\mathrm{res}}\).

Set \(H(t)\coloneqq C_{\hat p}t^{-2}\) and
\(P(t)\coloneqq \min\{1,\exp(-(p-\hat p)^2t/(2p))
+2(1-p_0)t^{1-q}\}\).
A union bound using \eqref{eq:regulated_lambda_tail_mu_zero_sfo} and
\eqref{eq:regulated_weighted_error_mu_zero_sfo} verifies
\eqref{eq:abstract_tail_event} with these \(H,P,\varepsilon_{\mathrm{res}}\).

The stated bound on \(t\) gives \(t\ge\bar T_{\mathrm{gc}}\) and
\(H(t)\Phi_0+\varepsilon_{\mathrm{res}}
=C_{\hat p}t^{-2}\Phi_0+\varepsilon_{\mathrm{res}}
\le\varepsilon_\phi\). Compatibility follows
from the argument in \Cref{subsec:frame}, and \Cref{Big_threorem}
proves \eqref{eq:probability_mu_zero_sfo}.

For the finite-horizon variant in
\Cref{rem:general_convex_interpretation}, intersect the two events used
above with the auxiliary-bound event
\(\max_{1\le i\le t+1}\|\bm{z}_i-\bm{x}^\star\|\le\widehat\Omega_z\).
All distance estimates then remain valid on this intersection.
A union bound adds the probability that this bound fails.
\end{proof}

\begin{proof}[\textbf{Proof of \Cref{cor:mu_zero_objective_gap}}]
\phantomsection\label{proof:mu_zero_objective_gap}
On an accepted trial, direct substitution of the updates into
\eqref{z_t} gives
\[
\bm{z}_{t+1}=\bm{z}_t-\frac{2\theta_{\mathrm{m}}\gamma_t}{\alpha_t}\mathbf{G}_t,
\qquad
\bm{x}_{t+1}=(1-\alpha_t)\bm{x}_t+
\alpha_t\left[\left(1-\frac1{2\theta_{\mathrm{m}}}\right)\bm{z}_t
+\frac{\bm{z}_{t+1}}{2\theta_{\mathrm{m}}}\right].
\]
The sum of the absolute values of the two coefficients in brackets
is \(\max\{1,\theta_{\mathrm{m}}^{-1}-1\}\). Since \(\bm{x}_1=\bm{z}_1\),
\eqref{eq:bounded_auxiliary_mu_zero_sfo} and induction give
\(\|\bm{x}_t-\bm{x}^\star\|\le\max\{1,\theta_{\mathrm{m}}^{-1}-1\}\widehat\Omega_z\) for every \(t\);
rejection preserves this bound. For every trial, including a rejected
one, \eqref{eq:yhat_zt} also gives
\(\|\hat {\bm{y}}_t-\bm{x}^\star\|\le\max\{1,\theta_{\mathrm{m}}^{-1}-1\}\widehat\Omega_z\).

Choose an analytical target
\(\varepsilon_\nabla\in
(\epsilon_{\mathrm{rel}}+\widetilde\epsilon_{\mathrm{g}},
\varepsilon_\phi/(\max\{1,\theta_{\mathrm{m}}^{-1}-1\}\widehat\Omega_z)]\), which is possible by
\eqref{eq:precision_mu_zero_objective_gap}.
If \(\|\nabla\phi(\hat {\bm{y}}_i)\|\le\varepsilon_\nabla\),
convexity yields
\(\phi(\hat {\bm{y}}_i)-\phi^\star
\le\langle\nabla\phi(\hat {\bm{y}}_i),\hat {\bm{y}}_i-\bm{x}^\star\rangle
\le\max\{1,\theta_{\mathrm{m}}^{-1}-1\}\widehat\Omega_z\varepsilon_\nabla
\le\varepsilon_\phi\).
Thus either branch of the stopping criterion in
\Cref{thm:mu_zero_sfo_bzo} gives a trial point satisfying the
objective-gap target, so
\(T_{\varepsilon_\phi}\le T_{\boldsymbol\varepsilon}\).
The theorem's trial bound and probability guarantee are independent
of this choice of \(\varepsilon_\nabla\), so they hold for
\(T_{\varepsilon_\phi}\) as claimed.
\end{proof}

\subsection{Constants in the complexity bounds}
\label{app:complexity_constants}
The constants \(c,C\) in the \textsc{sFO} probability bounds have the
values in \Cref{tab:sfo_tail_constants}. In particular, they are independent
of \(C_\alpha\), \(t\), and the target \(\varepsilon\).
\begin{table}[!htbp]
\centering
\caption{Exact coefficients in the \textsc{sFO} probability bounds.}
\label{tab:sfo_tail_constants}
\renewcommand{\arraystretch}{1.7}
\begin{tabular}{@{}lcc@{}}
\toprule
Tail condition & \(c\) & \(C\)\\
\midrule
Finite moments, \(1<q\le2\)
& --- & \(\displaystyle\frac{6(1-p_0)}{1-2^{1-q}}\)\\
\addlinespace[5pt]
Finite moments, \(q>2\)
& \(\displaystyle\frac{2}{3(q+2)^2e^q(1-p_0)^{2/q}}\)
& \(\displaystyle\frac{6(1+2/q)^q(1-p_0)}{1-2^{1-q}}\)\\
\addlinespace[5pt]
Sub-Weibull
& \(\displaystyle\frac{[\ln(e^\upsilon/(1-p_0))]^{2/q}}
{4(e^\upsilon-1)+2[\ln(e^\upsilon/(1-p_0))]^{1/q}}\)
& ---\\
\bottomrule
\end{tabular}
\end{table}
\FloatBarrier

For the general convex proof, the second contribution to
\(\varepsilon_{\mathrm{res}}\) in \Cref{tab:complexity_constants}
is denoted by
\begin{equation}\label{eq:regulated_step_residual_mu_zero}
\varepsilon_{\mathrm{step}}\coloneqq 
\frac{\gamma_{\max}}{1-\nu_{\mathrm{dec}}}
\left(\frac{\widetilde\epsilon_{\mathrm{g}}^2}{2(1-\theta)}
+3\widetilde\epsilon_{\mathrm{f}}\right).
\end{equation}
Thus \(\varepsilon_{\mathrm{res}}=
\widehat\Omega_zB_{\hat p}\widetilde\epsilon_{\mathrm{g}}+
\varepsilon_{\mathrm{step}}\), as used in
\eqref{eq:regulated_affine_modules_mu_zero}.

\clearpage
\begingroup

\captionsetup{font=small,labelfont=bf,skip=6pt,
  justification=raggedright,singlelinecheck=false}
\renewcommand{\arraystretch}{1.12}
\setlength{\textfloatsep}{14pt plus 2pt minus 2pt}
\setlength{\intextsep}{14pt plus 2pt minus 2pt}
\setlength{\floatsep}{12pt plus 2pt minus 2pt}

\section{Additional experimental details}
\label{app:exp_selected_details}

The evaluation metrics and pairing conventions are defined in
\Cref{subsec:exp_protocol}. This appendix gives the oracle constructions,
selected inputs, and supplementary comparisons. Candidate configurations
and path-level records are retained in the experiment archive.
For convex objectives, normalized gaps are clipped to
\([10^{-16},10^{12}]\) before taking logarithms, using a floor of
\(10^{-12}\) for the ridge-logistic scores.

\subsection{Error geometry and oracle costs}

\label{app:exp_noise_geometry}

\paragraph{Error construction.}
For the quadratic in \Cref{subsec:exp_noise_geometry}, let
\(\bm g=\nabla\phi(\bm y)\), \(S\) be Rademacher, \(\bm Z\sim N(0,I_2)\),
\(R\) a \(90^\circ\) rotation, and \(\bm u=(1,1)/\sqrt2\).
The collinear, isotropic, orthogonal, and fixed-direction profiles use,
respectively, \(\bm U=S\bm g/\|\bm g\|\),
\(\bm Z/\|\bm Z\|\), \(SR\bm g/\|\bm g\|\), and \(S\bm u\).
Independently draw \(B\sim\operatorname{Bernoulli}(0.1)\) and return
\[
\mathbf G=\bm g+\min\{r_B\|\bm g\|,\epsilon_{\mathrm c}\}\bm U,
\qquad r_0=0.1,\quad r_1=\sqrt{1000},\quad
\epsilon_{\mathrm c}=r_1\|\nabla\phi(\bm x_0)\|.
\]
The oracle returns zero at \(\bm g=0\). The profiles share their
conditional error-norm law at a common query point and use paired
Bernoulli and direction draws. The cap implies \textsc{sFO} with
\(\epsilon_{\mathrm g}=\epsilon_{\mathrm c}\) and suitable tail constants.
On \(B=0\), the relative bound also gives
\(\|\mathbf G-\bm g\|\le\|\mathbf G\|/9\), as in probabilistic
line search~\cite{Cartis2018,Xie2024}; this coefficient is independent
of the trial step, unlike the \(\gamma\|\mathbf G\|\) term in
\textsc{cFO}.

\paragraph{Selected inputs.}
\texttt{GD} and \texttt{NAG} use step \(1/(101L)\), with
\texttt{NAG} momentum 0.998011904. Starting from \(\bm v=0\),
\texttt{AGNES} uses \(\bm y=\bm x+a\bm v\),
\(\bm x^+=\bm y-\eta\mathbf G(\bm y)\), and
\(\bm v^+=\rho(\bm v-\mathbf G(\bm y))\), with
\(\eta=1/(26L)\), \(a=1.465060526\times10^{-7}\), and
\(\rho=0.996924260\). For \texttt{R-clipped-SSTM}, the local weight is
\(a_k=(k+1)/(32L)\) and the clipping threshold is \(B_j/a_k\), where
\(B_j=2^{-9}\sqrt{(1+L)/\mu}\,2^{-j}\) in stage \(j\).
States restart at the current output every 2,530 calls, with \(k=1\).
\texttt{SAAS} uses \((\theta,\theta_{\mathrm m})=(0.2,0.1)\),
\(\vartheta_{\mathrm m}=0\), \(\hat\gamma_1=1/(101L)\),
\(\gamma_{\max}=2/L\), \((\nu_{\mathrm{inc}},\nu_{\mathrm{dec}})=(1.5,0.99)\),
and \(\epsilon_{\mathrm{rel}}=\tau_{\mathrm{des}}^{(t)}=0\).
Method-specific preliminary experiments use different tuning budgets;
the selected inputs are then fixed across all four directions.

The convex accelerated searches initialize
\(\gamma_0=\nu_{\mathrm{inc}}\hat\gamma_1\) and
\(\alpha_0=\min\{[\sqrt{2\theta_{\mathrm m}\mu\gamma_0}
+\sqrt{\gamma_0/\gamma_{\max}}]/2,
0.99(1+\nu_{\mathrm{inc}})/(2+\nu_{\mathrm{inc}}-\theta_{\mathrm m})\}\).
In the geometry experiment, Armijo comparisons allow a \(10^{-15}\)
numerical margin. Every 100
calls, a method sets its current and query-forming states to zero if
all have normalized quadratic size below \(10^{-24}\). This is below
the \(10^{-12}\) display floor and the target in the table.
\paragraph{Function-value costs.}
Every convex search trial uses one first-order call and two scalar
function evaluations; the fixed-step and clipped baselines use only
the first-order call. If a function evaluation costs \(\eta\) times
a first-order call, their costs are therefore
\(C_\eta=(1+2\eta)N_{\mathrm{FO}}\) and
\(C_\eta=N_{\mathrm{FO}}\), respectively, excluding monitoring.
\Cref{tab:exp_weighted_target_costs} supplements the first-order-call
comparison with the representative choice \(\eta=0.5\).
All quadratic ablation variants have the same call ratio.

\begin{table}[H]
\centering
\caption{Calls to the median-trajectory target \(10^{-6}\).
The last column is the weighted cost of \texttt{SAAS} relative to
\texttt{R-clipped-SSTM} at \(\eta=0.5\); values below one favor
\texttt{SAAS}. Dashes denote an unreached target.}
\label{tab:exp_weighted_target_costs}

\small
\setlength{\tabcolsep}{5pt}
\begin{tabular*}{\linewidth}{@{\extracolsep{\fill}}lrrrrr@{}}
\toprule
Setting & \texttt{SAAS} & \texttt{R-clipped-SSTM} & \texttt{AGNES} & \texttt{NAG} & Cost ratio\\
\midrule
Collinear & 1,768 & 1,524 & 44,617 & -- & 2.32\\
Isotropic & 2,051 & 6,285 & 83,429 & -- & 0.65\\
Orthogonal & 4,447 & 3,417 & 93,624 & 16,884 & 2.60\\
Fixed direction & 2,017 & 6,285 & 79,407 & -- & 0.64\\
\bottomrule
\end{tabular*}
\end{table}

\subsection{Power-law least squares}
\label{app:exp_powerlaw_ls}

The final problems in \Cref{subsec:exp_powerlaw_ls} have \(n=128\),
\(m=4n\), and \(A^\top A/m=H\). We randomly permute the diagonal of
\(H\) and rescale a Gaussian initial direction to unit objective gap.
An error coordinate has the form \(cS(U^{-1/\widetilde q}-1)\), with
independent Rademacher \(S\) and uniform \(U\in(0,1)\).
The multiplier \(c\) sets the median error norm to
\(0.005\|\bm g_0\|\), where \(\bm g_0=\nabla\phi(\bm x_0)\) and
\(\widetilde q\in\{1.8,1.5,1.2\}\). These laws have finite moments
of orders below \(\widetilde q\) and infinite variance.
The biased setting at \(\widetilde q=1.5\) adds a fixed vector
\(\bm b\) opposite to \(\bm g_0\), scaled so that the mean oracle's
stationary point has normalized gap \(10^{-5}\).

\paragraph{Parameter selection.}
Each tuned method evaluates 32 candidates on two 96-dimensional
problems with condition numbers 500 and 2000, using two paths per
noise setting and 1,500 calls per path. Its best four proceed to two
112-dimensional problems with condition numbers 700 and 2500, four
paths per setting, and 3,000 calls. Both stages minimize the equally
weighted average of the eight problem--noise median late scores.
The selected configuration is evaluated on the three final problems
using 24 new paired paths per setting and 4,500 calls per path.
Problems and oracle streams are disjoint across these stages.
Removing condition~\eqref{Check_2} from the original search roster
left the final trajectories and confirmation winners unchanged.

\paragraph{Method inputs.}
Both searches use \((\nu_{\mathrm{inc}},\nu_{\mathrm{dec}})=(1.2,0.95)\),
\(\gamma_{\max}=2/L\), and the empirical calibration
\(\widetilde\epsilon_{\mathrm g}
=1.05(\|\bm b\|+0.005\|\bm g_0\|)\), with \(\bm b=0\) for centered errors.
This uses a median error norm, rather than the conditional-mean bound
in the theory. \texttt{SAAS} takes
\((\theta,\theta_{\mathrm m},\vartheta_{\mathrm m})=(0.3,0.05,0)\),
\(\hat\gamma_1=1/L\),
\(\epsilon_{\mathrm{rel}}=0.25\widetilde\epsilon_{\mathrm g}/\sqrt{0.35}\), and
\(\tau_{\mathrm{des}}^{(t)}=5\hat\gamma_t\widetilde\epsilon_{\mathrm g}^2/28\).
\texttt{SASS} takes \(\theta=0.2\), \(\hat\gamma_1=0.5/L\),
\(\epsilon_{\mathrm{rel}}=0.03\widetilde\epsilon_{\mathrm g}/\sqrt{0.32}\), and
\(\tau_{\mathrm{des}}^{(t)}=5\hat\gamma_t\widetilde\epsilon_{\mathrm g}^2/16\).
Thresholds and relaxations are zero in the exact-gradient reference.

\texttt{AGNES} uses \(\eta_t=0.7t^{-0.2}/L\),
\(\bm y_t=\bm x_t+0.3\eta_t\bm v_t\),
\(\bm x_{t+1}=\bm y_t-\eta_t\mathbf G_t\), and
\(\bm v_{t+1}=0.92(\bm v_t-\mathbf G_t)\), starting at \(\bm v_0=0\).
For \texttt{R-clipped-SSTM}, \(\alpha_k=(k+1)/(4L)\), the clip norm
is \(0.03\|\bm g_0\|/\alpha_k\), and states restart every
\(\operatorname{round}(\sqrt\kappa)\) calls with the same clipping numerator.
Classical \texttt{NAG} uses step \(1/L\) and momentum
\((\sqrt\kappa-1)/(\sqrt\kappa+1)\).

\paragraph{Supplementary comparisons.}
\Cref{fig:app_powerlaw_ls_problemwise} separates the three problems pooled in \Cref{fig:exp_powerlaw_ls}. \Cref{tab:exp_powerlaw_scores}~reports the full and late scores together with paired late-score comparisons. The paired comparison takes differences within matched paths before computing the median; at $\widetilde q=1.2$, it favors SAAS even though SASS has the lower median late score. With exact gradients, NAG has the lowest full score, and both NAG and SAAS reach the late-score floor of $-16$.

\begin{table}[H]
\centering
\caption{Least-squares median scores over three exact-gradient problems or
72 pairs per noisy setting. Below: median paired late-score advantages
(positive favors \texttt{SAAS}) and wins out of 72.}
\label{tab:exp_powerlaw_scores}

\small
\setlength{\tabcolsep}{4pt}
\renewcommand{\arraystretch}{1.0}
\begin{tabularx}{\linewidth}{@{}ll*{5}{>{\raggedleft\arraybackslash}X}@{}}
\toprule
Method & Score & Exact & \(\widetilde q=1.8\) & \(\widetilde q=1.5\)
& \(\widetilde q=1.2\) & \shortstack{\(\widetilde q=1.5\)\\+ bias}\\
\midrule
\texttt{SAAS} & Full & \(-12.784\) & \(-4.740\) & \(-4.744\) & \(-4.821\) & \(-4.391\)\\
& Late & \(-16.000\) & \(-5.005\) & \(-4.966\) & \(-5.027\) & \(-4.485\)\\
\addlinespace[1pt]
\texttt{SASS} & Full & \(-8.026\) & \(-4.577\) & \(-4.601\) & \(-4.648\) & \(-4.335\)\\
& Late & \(-11.358\) & \(-4.902\) & \(-4.952\) & \(-5.035\) & \(-4.558\)\\
\addlinespace[1pt]
\texttt{R-clipped-SSTM} & Full & \(-8.321\) & \(-4.417\) & \(-4.387\) & \(-4.351\) & \(-4.248\)\\
& Late & \(-12.232\) & \(-4.759\) & \(-4.746\) & \(-4.745\) & \(-4.560\)\\
\addlinespace[1pt]
\texttt{AGNES} & Full & \(-5.281\) & \(-4.414\) & \(-4.145\) & \(-3.490\) & \(-4.027\)\\
& Late & \(-6.668\) & \(-4.802\) & \(-4.436\) & \(-3.743\) & \(-4.257\)\\
\addlinespace[1pt]
\texttt{NAG} & Full & \(-15.264\) & \(-3.690\) & \(-3.407\) & \(-2.931\) & \(-3.418\)\\
& Late & \(-16.000\) & \(-3.710\) & \(-3.471\) & \(-3.010\) & \(-3.428\)\\
\midrule
Comparator & Statistic & & & & &\\

\texttt{SASS} & Advantage & -- & 0.164 & 0.156 & 0.133 & \(-0.044\)\\
& Wins / 72 & -- & 72 & 63 & 57 & 24\\
\addlinespace[1pt]
\texttt{R-clipped-SSTM} & Advantage & -- & 0.298 & 0.289 & 0.368 & \(-0.029\)\\
& Wins / 72 & -- & 72 & 72 & 72 & 24\\
\bottomrule
\end{tabularx}
\end{table}

\begin{figure}[H]
\centering
\includegraphics[width=\linewidth]{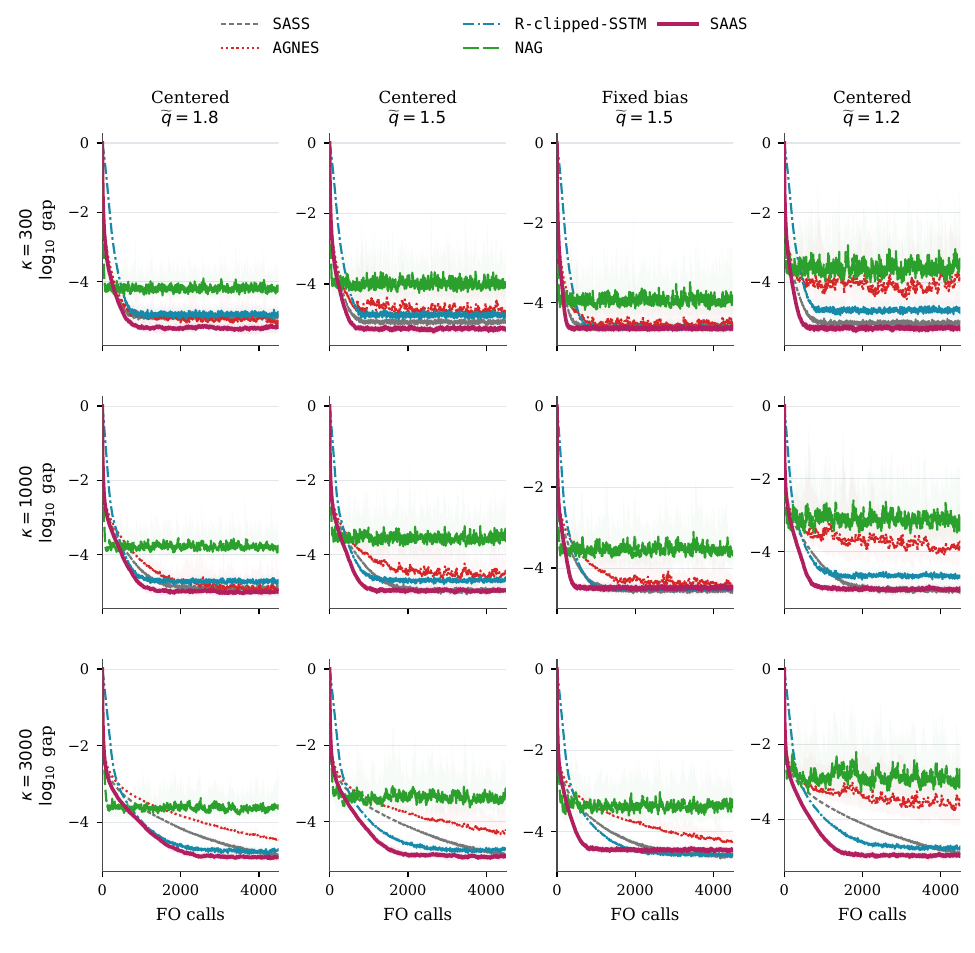}
\caption{\textbf{Least-squares trajectories by problem.}
Rows use \(\kappa=300,1000,3000\); columns use the four noisy settings.
Curves and bands summarize 24 paired paths per panel. The complete
4,500-call horizon includes the late-score window, calls 3,376--4,500.}
\label{fig:app_powerlaw_ls_problemwise}
\end{figure}
\subsection{Ridge-logistic objectives}
\label{app:exp_logistic_details}

We minimize \(\phi(\bm x)=m^{-1}\sum_{i=1}^m
\log(1+\exp(-b_i\bm a_i^\top\bm x))+(\lambda/2)\|\bm x\|^2\)
on the LIBSVM training sets \texttt{a9a}, \texttt{ijcnn1}, and
\texttt{w8a}~\cite{chang2011libsvm}, with dimensions
\((m,n)=(32561,123),(49990,22),(49749,300)\).
Nonzero feature columns are divided by their root mean square and
the design by \(\sqrt n\). With
\(L_{\mathrm{data}}=\lambda_{\max}(A^\top A/m)/4\), we set
\(\lambda=L_{\mathrm{data}}/(10^4-1)\), giving
\(\mu=\lambda\), \(L=L_{\mathrm{data}}+\lambda\), and \(L/\mu=10^4\).
All runs start at zero and use exact function values. \texttt{L-BFGS} followed
by Newton refinement gives reference points with gradient norms
\(3.63\times10^{-14}\), \(4.86\times10^{-17}\), and
\(1.27\times10^{-15}\), respectively. Strong convexity bounds their
exact-arithmetic objective error by
\(\|\nabla\phi(\bm x_{\mathrm{ref}})\|^2/(2\mu)\);
objective subtraction limits reported gaps near machine precision.

\paragraph{Oracle profiles.}
Writing \(\bm g_0=\nabla\phi(\bm x_0)\), the \texttt{a9a} oracle is
\(\mathbf G=\nabla\phi(\bm y)+B
\min\{0.5\|\nabla\phi(\bm y)\|,\|\bm g_0\|\}\bm D\), where
\(B\sim\operatorname{Bernoulli}(0.02)\) and \(\bm D\) is a fresh
isotropic unit vector. This gives \textsc{cFO} with
\((\epsilon_{\mathrm g},\kappa_{\mathrm g},\delta_{\mathrm g},\epsilon_{\mathrm c})
=(0,0,0.02,\|\bm g_0\|)\). On \texttt{ijcnn1} and \texttt{w8a},
independent coordinates are proportional to \(SU^{-1/1.5}\), with
Rademacher \(S\), uniform \(U\in(0,1)\), and median error norm
\(0.0135\|\bm g_0\|\). The \texttt{w8a} oracle additionally uses a
pathwise-fixed isotropic bias of norm \(0.00135\|\bm g_0\|\).
Both Pareto profiles have infinite variance. Oracle draws are paired
across methods and independent of their step sizes.

\paragraph{Selection protocol.}
For each data set, we expand the existing configurations through a
common additional tuning stage. Each of the five tuned methods receives
96 candidate configurations on four paired development paths; the best
six are compared on eight independent confirmation paths. Candidates
combine changes around the existing inputs with broader joint parameter
choices. All paths use 1,200 first-order calls. The selected inputs are
then fixed for 32 new paired evaluation paths. Classical \texttt{NAG}
uses its fixed parameters throughout.

The same selection score weights the median full and late scores by
0.4 and 0.6, adds 0.2 times the 90th percentile of rebound and 0.1 times
the 90th-percentile--median late-score difference, and adds 1,000 times
the failure fraction. The late window is calls 901--1,200;
rebound is its positive score increase over the best of the nine
consecutive 100-call blocks after call 300. Scores use
\(\log_{10}\max\{r_t,10^{-12}\}\).
A nonfinite or incomplete path, or an iterate norm above \(10^{12}\),
counts as a failure, with remaining gaps set to \(10^{12}\).

\begin{table}[H]
\centering
\caption{Selected ridge-logistic inputs. For \texttt{SAAS},
\(\vartheta_{\mathrm m}=0\) and
\(\tau_{\mathrm{des}}^{(t)}=r_{\mathrm{des}}\hat\gamma_t\|\bm g_0\|^2\).
Thresholds and relaxation multipliers are selected empirically and fixed within a run.}
\label{tab:exp_logistic_saas_parameters}

\small
\setlength{\tabcolsep}{5pt}
\renewcommand{\arraystretch}{1.0}
\begin{tabularx}{\linewidth}{@{}ll*{3}{>{\raggedleft\arraybackslash}X}@{}}
\toprule
Method & Input & \texttt{a9a} & \texttt{ijcnn1} & \texttt{w8a}\\
\midrule
\texttt{SAAS} & \(\theta\) & 0.300 & 0.261 & 0.150\\
& \(\theta_{\mathrm m}\) & 0.2985 & 0.260 & 0.120\\
& \(L\hat\gamma_1\) & 2.066 & 1.603 & 1.281\\
& \(L\gamma_{\max}\) & 3.038 & 1.742 & 2.468\\
& \((\nu_{\mathrm{inc}},\nu_{\mathrm{dec}})\) & \((1.05,0.99)\) & \((1.0865,0.99662)\) & \((1.03,0.99)\)\\
& \(\epsilon_{\mathrm{rel}}/\|\bm g_0\|\) & 0 & 0.01551 & 0.010\\
& \(r_{\mathrm{des}}\) & 0 & \(1.76\times10^{-5}\) & \(1.75\times10^{-5}\)\\
\midrule
\texttt{R-clipped-SSTM} & \(a_{\mathrm S}\) & 0.3967 & 2.2893 & 5.4558\\
& \(c\) & 26.236 & 1.3578 & 6.2957\\
& \(R\) & 232 & 35 & 266\\
\addlinespace[1pt]
\texttt{FISTA-SS} & \(s\) & 1.114922 & 0.924625 & 1.541259\\
& \(\nu_{\mathrm{dec}}\) & 0.998330 & 0.999000 & 0.995968\\
\addlinespace[1pt]
\texttt{Adam} & \(s\) & 0.3302 & 2.0477 & 5.2050\\
& \((\beta_1,\beta_2)\) & \((0.95,0.95555)\) & \((0.95,0.92566)\) & \((0,0.90824)\)\\
& \(\varepsilon\) & \(4.34\times10^{-6}\) & \(3.77\times10^{-10}\) & \(3.08\times10^{-6}\)\\
& \(p\) & 0.25 & 0.75 & 0.75\\
\addlinespace[1pt]
\texttt{AGNES} & \(s\) & 2.1027 & 0.2676 & 0.4425\\
& \(\rho\) & 0.97056 & 0.91281 & 0.97266\\
& \(a\) & 1.2643 & 0.1940 & 0.1548\\
\bottomrule
\end{tabularx}
\end{table}

\paragraph{Selected methods.}
\Cref{tab:exp_logistic_saas_parameters} gives the inputs.
\texttt{SAAS} uses condition~\textnormal{(III)} to regulate the next
trial step after acceptance and leaves condition~\textnormal{(II)} disabled.
For \texttt{R-clipped-SSTM}, the local weight is
\(\alpha_j=(j+1)/(2a_{\mathrm S}L)\), the clip norm is
\(c\|\bm g_0\|/\alpha_j\), and both states and the local counter
restart every \(R\) calls. \texttt{AGNES} uses step \(\eta=s/L\),
\(\bm y_t=\bm x_t+a\eta\bm v_t\), and
\(\bm v_{t+1}=\rho(\bm v_t-\mathbf G_t)\), starting at zero.
\texttt{Adam} uses its standard bias-corrected moments and step
\(s\|\bm g_0\|t^{-p}/(L\sqrt n)\), adding \(\varepsilon\) to the
root second moment. Smooth \texttt{FISTA-SS}~\cite{stoFISTA} uses
its general-convex momentum recursion, the quadratic-model test with
zero relaxation, initial step \(s/L\), and reciprocal step factors
\(\nu_{\mathrm{dec}}\) and \(1/\nu_{\mathrm{dec}}\).
Its cap \(1/\mu\) is inactive. Classical \texttt{NAG} uses step
\(1/L\) and momentum \((\sqrt{L/\mu}-1)/(\sqrt{L/\mu}+1)\).

\Cref{fig:exp_logistic_stress_tests,tab:exp_logistic_results}
show the trajectories and scores. \texttt{SAAS} has the lowest median
full score on all three data sets. On \texttt{ijcnn1}, its early decrease
is followed by a small rebound, and \texttt{Adam} has the lower late
score. On \texttt{w8a}, \texttt{SAAS} reaches the low-gap region earlier;
the late scores of \texttt{SAAS}, \texttt{Adam}, and
\texttt{R-clipped-SSTM} differ by less than 0.02 decades.

\begin{figure}[H]
\centering
\includegraphics[width=0.94\linewidth]{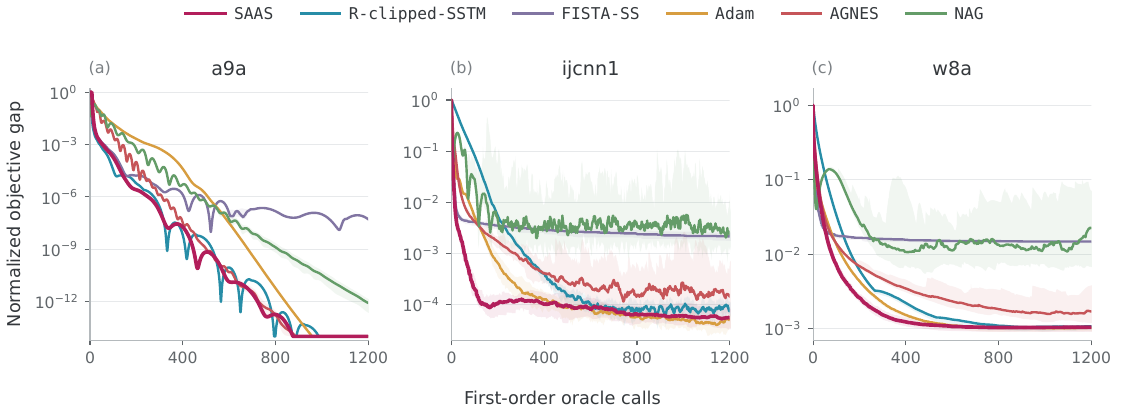}
\caption{\textbf{Ridge-logistic objectives.}
Bounded relative corruption on \texttt{a9a}, centered Pareto errors on
\texttt{ijcnn1}, and Pareto errors with persistent bias on \texttt{w8a}.
Lines and bands show pointwise medians and 10th--90th percentiles over
32 paired paths for each method. The \texttt{a9a} display floor is \(10^{-14}\).}
\label{fig:exp_logistic_stress_tests}

\medskip
\begin{minipage}{\linewidth}
\centering
\captionof{table}{Median full and late ridge-logistic scores over 32 paired paths.
Smaller values are better; gaps below \(10^{-12}\) receive the same score.}
\label{tab:exp_logistic_results}
\small
\begin{tabular*}{\linewidth}{@{\extracolsep{\fill}}lrrrrrr@{}}
\toprule
& \multicolumn{2}{c}{\texttt{a9a}} & \multicolumn{2}{c}{\texttt{ijcnn1}}
& \multicolumn{2}{c}{\texttt{w8a}}\\
Method & Full & Late & Full & Late & Full & Late\\
\midrule
\texttt{SAAS} & \(-9.239\) & \(-12.000\) & \(-3.987\) & \(-4.241\) & \(-2.770\) & \(-2.983\)\\
\texttt{R-clipped-SSTM} & \(-9.160\) & \(-12.000\) & \(-3.349\) & \(-4.099\) & \(-2.590\) & \(-2.976\)\\
\texttt{FISTA-SS} & \(-6.074\) & \(-7.187\) & \(-2.544\) & \(-2.654\) & \(-1.788\) & \(-1.829\)\\
\texttt{Adam} & \(-7.382\) & \(-12.000\) & \(-3.789\) & \(-4.327\) & \(-2.710\) & \(-2.993\)\\
\texttt{AGNES} & \(-8.766\) & \(-12.000\) & \(-3.267\) & \(-3.676\) & \(-2.439\) & \(-2.788\)\\
\texttt{NAG} & \(-6.955\) & \(-10.898\) & \(-2.210\) & \(-2.375\) & \(-1.672\) & \(-1.800\)\\
\bottomrule
\end{tabular*}
\end{minipage}
\end{figure}

\subsection{Neural-network setup and comparisons}
\label{app:exp_nn}

\paragraph{Network and data.}
The digits data~\cite{pedregosa2011scikit} use a fixed stratified split
of 1,347 training and 450 test examples. Pixels are divided by 16 and
centered using the training mean. Each network has two ReLU hidden
layers and a ten-class softmax output; all weights and biases minimize
the mean training cross-entropy. Weights are independent centered
Gaussians with variance \(2/\text{fan-in}\), and biases start at zero.

The searches use \(\mu=\vartheta_{\mathrm m}=0\),
\(\alpha_0=\sqrt{0.1}/2\), and the momentum updates in
\eqref{para_alpha} and \eqref{eq:momentum_parameterization}.
A fresh uniform minibatch, sampled without replacement within the
batch, supplies the gradient and both Armijo losses. \texttt{SASS} sets \(\hat{\bm y}_t=\bm x_t\). The method and scale comparisons
use \(\epsilon_{\mathrm{rel}}=0\); the safeguard comparison varies it.

\paragraph{Measurements and work.}
Computations use float64 and stable softmax cross-entropy.
Full-training loss is recorded at the main iterate every ten calls
and at the endpoints. For scores, normalized losses are floored at
\(10^{-16}\), transformed to \(\log_{10}\), and linearly interpolated
at integer work units, carrying the last observation to the horizon.
A nonfinite training loss or one above \(10^6\) receives log-score 6
thereafter. Training- and test-set monitoring is excluded from work
counts. A search trial costs two minibatch forward passes and one
backward pass; an optimizer update costs one of each.

\subsubsection{Gradient and function-value errors}
\label{app:exp_nn_comparison}

The method comparison uses widths \(32\)--\(16\), batch size 128,
and a budget of 2,400 work units. Given minibatch gradient \(\bm g_B\),
the oracle returns
\begin{equation}\label{eq:exp_nn_flip}
\mathbf G=\frac{S}{1-2\delta_{\mathrm g}}\bm g_B,
\qquad S\in\{-1,1\},\qquad
\mathbb P\{S=-1\}=\delta_{\mathrm g}\in\{0.1,0.2,0.3,0.4\}.
\end{equation}
Signs are independent of the minibatches and across calls.
The rescaling preserves \(\mathbb E_S\mathbf G=\bm g_B\), with
\(\mathbb E_S\|\mathbf G\|^2=\|\bm g_B\|^2/(1-2\delta_{\mathrm g})^2\).
Here \(\delta_{\mathrm g}\) is the imposed reversal probability;
\textsc{cFO} concerns the total error relative to the full gradient.
Paired runs share initialization, minibatches, and sign draws.

Both searches additionally receive independent uniform errors on
\([-a,a]\) in the query and candidate minibatch losses, paired across
methods and amplitudes. For
\(a/\phi(\bm x_0)\in\{0,10^{-4},10^{-3}\}\), the Armijo relaxation is
\(\tau_{\mathrm{des}}^{(t)}=2\epsilon_{\mathrm f}'\), following the
additive form in \cite{Xie2024}. The amplitude is fixed within a run;
the added error in a loss difference is at most \(2a\), in addition
to minibatch error. Thus these errors do not have the step-scaled
form of Assumption~\ref{BZO}. \texttt{Adam}, clipped \texttt{RMSprop},
and \texttt{AGNES} reuse their paths across amplitudes.
Clipped \texttt{RMSprop} uses \(\mathbf G\min\{1,c/\|\mathbf G\|\}\).

\paragraph{Parameter selection.}
Configurations minimize mean full score on pilot paths. At
\(\delta_{\mathrm g}=0.2,0.3\), each search screens 96 candidates on
one 1,200-unit path and confirms its best four on two new 2,400-unit
paths. Each optimizer screens 96 additional candidates on a
2,400-unit path and confirms its best three on two new paths.
These stages extend earlier tuning, so cumulative tuning budgets
differ across methods. At \(\delta_{\mathrm g}=0.1,0.4\), all methods
screen 32 candidates on one path and confirm the best three on two
new paths, using 2,400 units throughout. Search candidates satisfy
\(\theta<1-2\delta_{\mathrm g}\).

Base inputs are selected at \(a=0\). At each positive amplitude, both
searches then test \(\epsilon_{\mathrm f}'/a\in
\{0,0.125,0.25,0.5,1,2,4\}\) on two pilot paths and confirm the best
three on two new paths, with all other inputs fixed. Evaluation uses
eight paired paths excluded from selection. Full and late scores
cover units 1--2,400 and 1,801--2,400.
\Cref{tab:exp_nn_search_inputs,tab:exp_nn_optimizer_inputs} give the
selected configurations; \Cref{tab:exp_nn_comparison_summary} summarizes
all methods at the largest function-value error amplitude.

\begin{table}[H]
\centering
\caption{Neural-network search inputs and relaxation ratios.
Here \(\gamma_0=\gamma_{\max}\),
\(\hat\gamma_1=\gamma_0/\nu_{\mathrm{inc}}\), and
\(\epsilon_{\mathrm{rel}}=0\). The final two columns give
\(\epsilon_{\mathrm f}'/a\); at \(a=0\), the relaxation is zero.}
\label{tab:exp_nn_search_inputs}

\small
\setlength{\tabcolsep}{4pt}
\begin{tabular*}{\linewidth}{@{\extracolsep{\fill}}lcccccccc@{}}
\toprule
& & & & & & & \multicolumn{2}{c}{\(a/\phi(\bm x_0)\)}\\
\cmidrule(l){8-9}
Method & \(\delta_{\mathrm g}\) & \(\theta\) & \(\theta_{\mathrm m}\) & \(\gamma_{\max}\)
& \(\nu_{\mathrm{inc}}\) & \(\nu_{\mathrm{dec}}\) & \(10^{-4}\) & \(10^{-3}\)\\
\midrule
\texttt{SAAS} & 0.1 & 0.2 & 0.1 & 1 & 1.5 & 0.8 & 4 & 0.125\\
 & 0.2 & 0.2 & 0.1 & 0.5 & 2 & 0.5 & 4 & 0.25\\
 & 0.3 & 0.2 & 0.1 & 0.25 & 1.5 & 0.8 & 1 & 0.25\\
 & 0.4 & 0.1 & 0.05 & 0.125 & 1.5 & 0.8 & 0.5 & 0.5\\
\addlinespace[2pt]
\texttt{SASS} & 0.1 & 0.1 & -- & 2 & 2 & 0.5 & 4 & 4\\
 & 0.2 & 0.1 & -- & 0.5 & 1.5 & 0.8 & 0.5 & 0.5\\
 & 0.3 & 0.1 & -- & 0.5 & 1.5 & 0.8 & 0.5 & 0.5\\
 & 0.4 & 0.025 & -- & 0.5 & 1.5 & 0.8 & 0.25 & 2\\
\bottomrule
\end{tabular*}
\end{table}

\begin{table}[H]
\centering
\caption{Neural-network optimizer inputs. The schedules use the completed
fraction \(u\) of the prescribed updates. \texttt{Adam} and
\texttt{RMSprop} use denominator constant \(10^{-8}\).
The \texttt{AGNES} update index satisfies \(k\ge1\).}
\label{tab:exp_nn_optimizer_inputs}
\small
\setlength{\tabcolsep}{5pt}
\renewcommand{\arraystretch}{1.10}
\begin{tabularx}{\linewidth}{@{}c>{\centering\arraybackslash}Xcc@{}}
\toprule
\multicolumn{4}{@{}l}{\texttt{Adam}}\\
\addlinespace[2pt]
\(\delta_{\mathrm g}\) & Learning rate & \(\beta_1\) & \(\beta_2\)\\
\midrule
0.1 & \(0.03/(1+9u)\) & 0.5 & 0.9\\
0.2 & \(0.09/\sqrt{1+99u}\) & 0.5 & 0.999\\
0.3 & \(0.02[1+\cos(\pi u)]/2\) & 0.5 & 0.9\\
0.4 & \(0.01[1+\cos(\pi u)]/2\) & 0.5 & 0.999\\
\bottomrule
\end{tabularx}

\medskip
\begin{tabularx}{\linewidth}{@{}c>{\centering\arraybackslash}Xccc@{}}
\toprule
\multicolumn{5}{@{}l}{Clipped \texttt{RMSprop}}\\
\addlinespace[2pt]
\(\delta_{\mathrm g}\) & Learning rate & Second-moment decay & Momentum & Clip norm\\
\midrule
0.1 & 0.01 & 0.9 & 0 & 0.03\\
0.2 & \(0.001/\sqrt{1+9u}\) & 0.99 & 0.9 & 0.1\\
0.3, 0.4 & 0.003 & 0.999 & 0 & 0.03\\
\bottomrule
\end{tabularx}

\medskip
\begin{tabularx}{\linewidth}{@{}cc>{\centering\arraybackslash}X>{\centering\arraybackslash}X@{}}
\toprule
\multicolumn{4}{@{}l}{\texttt{AGNES}}\\
\addlinespace[2pt]
\(\delta_{\mathrm g}\) & Step size & Extrapolation/step ratio & Momentum multiplier\\
\midrule
0.1 & 0.3 & 0.1 & 0.99\\
0.2 & 0.3 & 0.1 & \((k-1)/(k+4)\)\\
0.3 & 0.1 & 0.1 & 0.99\\
0.4 & 0.03 & 0.04 & 0.99\\
\bottomrule
\end{tabularx}
\end{table}

\begin{table}[H]
\centering
\caption{Neural-network results at \(a/\phi(\bm x_0)=10^{-3}\).
Scores and final normalized losses are medians over eight paired paths.
The last column counts paths on which \texttt{SAAS} has the lower full
score, out of eight. The main figure shows all three amplitudes.}
\label{tab:exp_nn_comparison_summary}

\small
\setlength{\tabcolsep}{5pt}
\renewcommand{\arraystretch}{1.08}
\begin{tabularx}{\linewidth}{@{}cXrrrr@{}}
\toprule
\(\delta_{\mathrm g}\) & Method & Full & Late & Final loss & \texttt{SAAS} lower\\
\midrule
0.1 & \texttt{SAAS} & \(-3.847\) & \(-5.090\) & \(5.98\times10^{-6}\) & --\\
& \texttt{SASS} & \(-2.694\) & \(-3.559\) & \(2.29\times10^{-4}\) & 8\\
& \texttt{Adam} & \(-6.117\) & \(-9.103\) & \(4.43\times10^{-10}\) & 0\\
& Clipped \texttt{RMSprop} & \(-4.685\) & \(-6.965\) & \(1.28\times10^{-8}\) & 0\\
& \texttt{AGNES} & \(-2.571\) & \(-2.878\) & \(4.26\times10^{-4}\) & 6\\
\midrule
0.2 & \texttt{SAAS} & \(-3.587\) & \(-4.738\) & \(1.34\times10^{-5}\) & --\\
& \texttt{SASS} & \(-2.223\) & \(-2.811\) & \(1.28\times10^{-3}\) & 8\\
& \texttt{Adam} & \(-3.041\) & \(-3.707\) & \(1.60\times10^{-4}\) & 8\\
& Clipped \texttt{RMSprop} & \(-3.429\) & \(-4.701\) & \(7.82\times10^{-6}\) & 7\\
& \texttt{AGNES} & \(-3.331\) & \(-4.391\) & \(3.07\times10^{-5}\) & 8\\
\midrule
0.3 & \texttt{SAAS} & \(-3.246\) & \(-4.292\) & \(3.24\times10^{-5}\) & --\\
& \texttt{SASS} & \(-2.198\) & \(-2.862\) & \(1.11\times10^{-3}\) & 8\\
& \texttt{Adam} & \(-2.821\) & \(-4.623\) & \(1.03\times10^{-5}\) & 7\\
& Clipped \texttt{RMSprop} & \(-2.624\) & \(-3.980\) & \(4.74\times10^{-5}\) & 8\\
& \texttt{AGNES} & \(-2.646\) & \(-3.918\) & \(9.06\times10^{-5}\) & 8\\
\midrule
0.4 & \texttt{SAAS} & \(-2.583\) & \(-3.516\) & \(2.21\times10^{-4}\) & --\\
& \texttt{SASS} & \(-2.183\) & \(-3.010\) & \(7.33\times10^{-4}\) & 7\\
& \texttt{Adam} & \(-1.311\) & \(-1.770\) & \(1.66\times10^{-2}\) & 8\\
& Clipped \texttt{RMSprop} & \(-1.174\) & \(-1.730\) & \(1.19\times10^{-2}\) & 8\\
& \texttt{AGNES} & \(-1.240\) & \(-1.781\) & \(1.06\times10^{-2}\) & 8\\
\bottomrule
\end{tabularx}
\end{table}
\subsubsection{Gradient-norm safeguard}
\label{app:exp_nn_search_diagnostics}

For \Cref{fig:exp_nn_search_diagnostics}, both searches use
\(\theta=0.2\), \(\gamma_0=0.25\), the independently specified
\(\hat\gamma_1=1/6\),
\((\nu_{\mathrm{inc}},\nu_{\mathrm{dec}})=(1.05,0.9)\), and
\(\tau_{\mathrm{des}}^{(t)}=0.5a\), without a step-size cap.
\texttt{SAAS} uses \(\theta_{\mathrm m}=0.1\).
The network and oracle follow \Cref{app:exp_nn_comparison}, with
\(\delta_{\mathrm g}=0.3\) and \(a/\phi(\bm x_0)=0.007\).
The only change between rows is the threshold shared by both methods:
\(\epsilon_{\mathrm{rel}}=0\) or 0.012.

Common step factors were selected on pilot and confirmation paths
from \(\nu_{\mathrm{inc}}\in\{1.05,1.1,1.2,1.35,1.5\}\) and
\(\nu_{\mathrm{dec}}\in\{0.5,0.65,0.8,0.9\}\), considering both searches.
The error level and positive threshold were selected on further paths.
Twelve new evaluation paths pair initialization and all oracle draws
across methods and thresholds. Each runs for 4,800 work units;
the first 3,200 units were chosen for display after viewing the full runs.

The displayed endpoint comprises 1,066 trials and 3,198 charged units.
Over this window, the positive threshold reduces the next step after
a mean 8.1\% of accepted \texttt{SAAS} trials and never activates for
\texttt{SASS}. Reversed-gradient acceptance in the last 300
displayed trials averages 0.611 and 0.717 at zero and positive thresholds.
Over the complete runs, median endpoint losses are
\(1.13\times10^{-3}\) and \(2.47\times10^{-4}\), respectively;
the positive threshold has a lower late score on 11 of 12 paths,
using work units 3,601--4,800.

\subsection{Search-scale and momentum-allocation ablation}
\label{app:exp_allocation_ablation}

\subsubsection{Strongly convex quadratics}
\label{app:exp_selected_ablation_details}

We use \(\phi(\bm x)=\bm x^\top Q\Lambda Q^\top\bm x/2\) in
dimension 32, with random orthogonal \(Q\), geometric spectrum at
\(\kappa=500\), and linear spectrum at \(\kappa=4000\).
Initial gaps are one. The four variants are the classical pair
\((\theta,\theta_{\mathrm m})=(0.5,0.5)\), the fixed pair \((0.5,0.1)\),
a selected pair, and \texttt{SASS} at \(\theta=0.5\).
Their remaining inputs are shared within each problem.

\paragraph{Errors and search inputs.}
Centered errors have isotropic directions and a Student-\(t\) radius
normalized to mean norm \(0.15\|\bm g_0\|\).
The three settings use centered \(t_{1.2}\), the same error plus
a fixed bias of norm \(0.03\|\bm g_0\|\), and \(t_{1.5}\) plus
gradient-opposed bias of norm \(0.06\|\bm g_0\|\), active for 12 of
every 64 calls. We set
\(\widetilde\epsilon_{\mathrm g}=1.05(0.15+b)\|\bm g_0\|\),
with \(b=0,0.03,0.06\), and
\(\tau_{\mathrm{des}}^{(t)}
=\hat\gamma_t\widetilde\epsilon_{\mathrm g}^2/[2(1-\theta)]\).
Thus changing \(\theta\) changes both the required decrease and this
relaxation. All variants use \(\vartheta_{\mathrm m}=0\),
\(\gamma_{\max}=2/L\), and a numerical step floor of
\(\gamma_{\max}\) times machine precision. Exact-gradient runs have
zero thresholds and relaxations. No clipping, averaging, or momentum
restarts are used.

\begin{table}[H]
\centering
\caption{Shared quadratic search inputs and selected scale pairs.}
\label{tab:exp_ablation_controls}
\small
\setlength{\tabcolsep}{6pt}
\begin{tabular*}{\linewidth}{@{\extracolsep{\fill}}cccccc@{}}
\toprule
\(\kappa\) & \(\nu_{\mathrm{inc}}\) & \(\nu_{\mathrm{dec}}\) & \(L\hat\gamma_1\)
& \(\epsilon_{\mathrm{rel}}/\widetilde\epsilon_{\mathrm g}\) & Selected \((\theta,\theta_{\mathrm m})\)\\
\midrule
500 & 1.1 & 0.8 & \(20/11\) & 0.2 & \((0.5,0.05)\)\\
4000 & 2 & \(2^{-13/7}\) & 1 & 0.2 & \((0.5,0.05)\)\\
\bottomrule
\end{tabular*}
\end{table}

\paragraph{Selection and evaluation.}
We vary \(\theta\) over \(\{0.2,0.3,\ldots,0.7\}\) and
\(\theta_{\mathrm m}\) over \(\{0.05,0.1,0.2,0.3,0.4,0.5\}\), retaining
the 27 pairs with \(\theta_{\mathrm m}<\theta\) or
\((\theta,\theta_{\mathrm m})=(0.5,0.5)\).
We select the pair with the largest average of the three noise-specific median
paired late-score advantages over the classical pair. At
\(\kappa=500,4000\), selection uses 32 and 16 paths per setting;
trajectory evaluation uses 64 and 32 new paths, respectively, with
1,400 calls per path. Objectives and initial points remain fixed.

The common inputs were developed using smaller allocations.
At \(\kappa=500\), candidate inputs were screened on four paths per
setting, with exact-gradient restrictions on the fixed and selected
pairs: calls to \(10^{-4},10^{-6},10^{-10}\) could increase by at most
10\% relative to their earlier configurations, and full score by at
most 0.1 decades. After viewing selection results, we additionally
required a winner distinct from \((0.5,0.1)\) and a nonnegative
advantage over \texttt{SASS} in at least one noisy setting.
At \(\kappa=4000\), three-path screening emphasized noisy late scores
without exact-speed restrictions. The common inputs and error scales
were retained after inspecting evaluation outcomes; the classical
variant and \texttt{SASS} received no separate tuning.

\paragraph{Observed effects.}
The classical pair reaches \(10^{-6}\) first with exact gradients.
Under noise, the selected pair improves the paired late score over
the classical pair in all six problem--noise combinations, averaging
0.438 and 1.244 decades on the two problems. Against \texttt{SASS}, it has a positive median paired advantage only for centered
\(t_{1.2}\) at \(\kappa=500\).
Several plateaus coincide with the numerical step floor; for the
selected pair at \(\kappa=4000\), median first arrivals are about
177 and 150 calls in the first two noisy settings.

\begin{figure}[H]
\centering
\includegraphics[width=\linewidth]{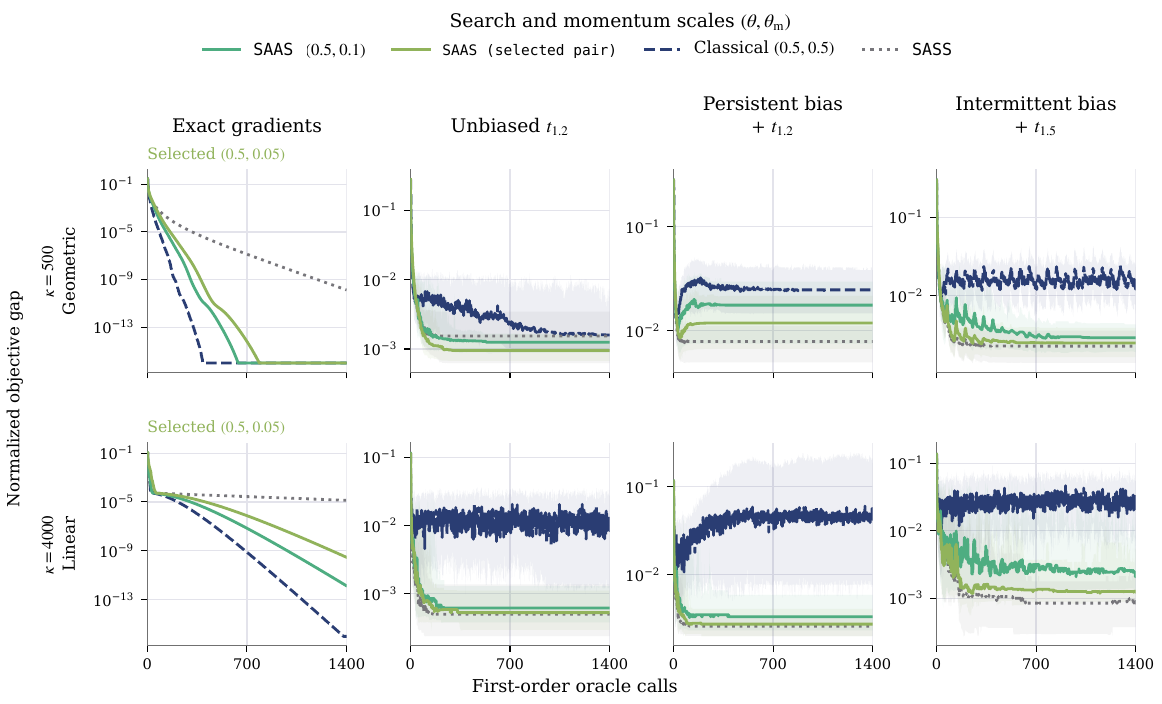}
\caption{\textbf{Strongly convex quadratic ablation.}
Rows use \(\kappa=500,4000\). Blue dashed, mint-green, and yellow-green curves use \((\theta,\theta_{\mathrm m})=(0.5,0.5),(0.5,0.1),(0.5,0.05)\),
respectively; gray is \texttt{SASS} at \(\theta=0.5\).
Noisy panels summarize 64 and 32 paired paths in the two rows.}
\label{fig:exp_selected_ablation}
\end{figure}

\begin{figure}[!htbp]
\centering
\includegraphics[width=0.92\linewidth]{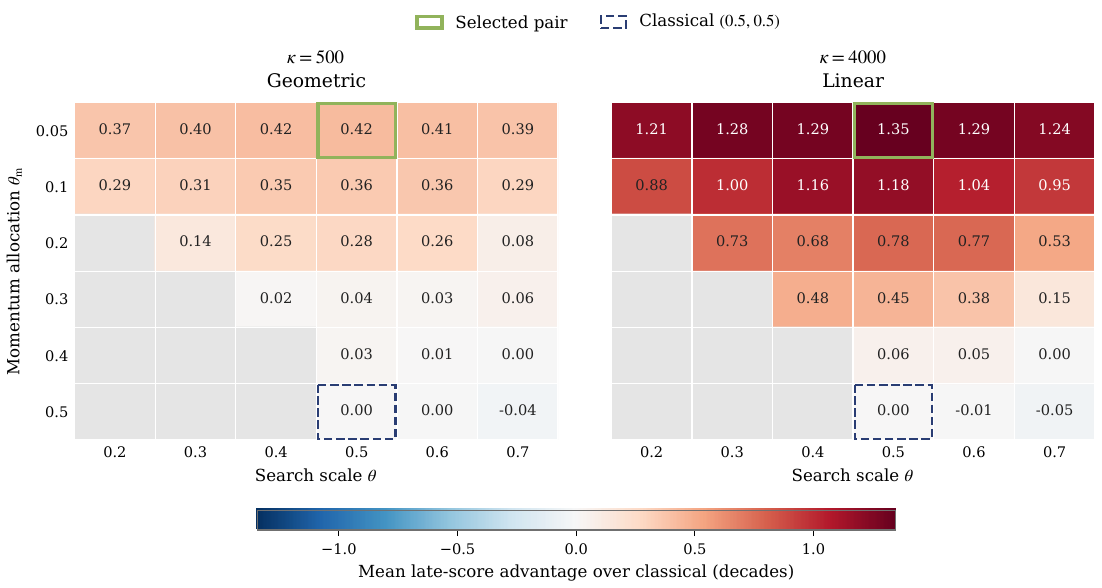}
\caption{\textbf{Quadratic scale sensitivity.}
Entries average the three noise-specific median paired late-score
advantages over \((0.5,0.5)\); positive values favor the cell.
Solid and dashed boxes mark the selected and classical pairs.}
\label{fig:exp_selected_scale_heatmap}
\end{figure}

\subsubsection{Neural-network training}
\label{app:exp_nn_ablation}

Using the setup in \Cref{app:exp_nn}, we cross widths
\(32\)--\(16\), \(512\)--\(256\), and \(1024\)--\(512\) with batch
sizes 32, 64, 128, and 512. The grid tests
\(\theta\in\{0.1,0.2,\ldots,0.7\}\) and
\(\theta_{\mathrm m}\in\{0.025,0.05,0.1,0.2,0.3,0.4,0.5\}\):
39 pairs with \(\theta_{\mathrm m}\le\theta\), plus \texttt{SASS} at each of the seven \(\theta\) values. Each configuration
uses five paired paths and 800 calls. A trial uses the same minibatch
for gradient and function evaluations, with no added corruption,
\(\tau_{\mathrm{des}}^{(t)}=0\), and no step-size cap.

For the three widths, respectively,
\((\nu_{\mathrm{inc}},\nu_{\mathrm{dec}})=(1.05,0.5),(1.1,0.3),(1.05,0.3)\)
and \(\gamma_0=0.05,0.05,0.1\). The initial trial step is independently
fixed at \(\gamma_0/1.03\). The factors were selected from 28
combinations on separate pilot and confirmation paths at
\(\theta=0.5\), using \(\theta_{\mathrm m}=0.05,0.1,0.5\) and \texttt{SASS}.
Selection combined full and late scores with rebound and failure
penalties; the factors were then fixed across the complete scale grid
and all batch sizes. Scores use units 1--2,400 and 1,801--2,400.
Larger batches process more examples over this common call horizon.

\Cref{fig:exp_nn_curves_05,fig:exp_nn_batch_allocation} show the
\(\theta=0.5\) slice in the main text. The late-score grids in
\Cref{fig:exp_nn_late_heatmap,fig:exp_nn_batch64_grid} show how changing
\(\theta\) alters this comparison. Batch size 64 also favors smaller allocations: within
the \(\theta=0.5\) slice, the best paired gains are 3.37, 2.81, and
4.21 decades for the three widths. At batch size 512, no tested
configuration has a positive median paired full- or late-score
advantage over the classical pair; several late scores tie at the
\(10^{-16}\) scoring floor.

\begin{figure}[H]
\centering
{\small\noindent\makebox[\linewidth][l]{%
\hspace{0.095\linewidth}%
\makebox[0.265\linewidth]{\(32\)--\(16\)}%
\makebox[0.265\linewidth]{\(512\)--\(256\)}%
\makebox[0.265\linewidth]{\(1024\)--\(512\)}}\par}
\includegraphics[width=0.89\linewidth,trim={0bp 0bp 49.69bp 253bp},clip]
{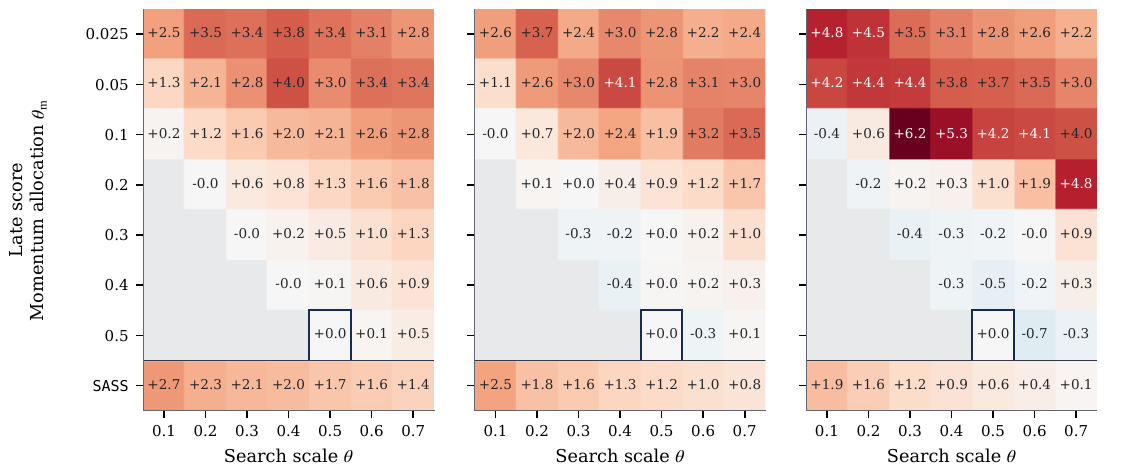}%
\hfill
\includegraphics[height=0.35\linewidth,trim={540bp 97.40bp 0bp 89bp},clip]
{figures/experiments/nn_batch64_full_late_grid.pdf}
\caption{\textbf{Scale sensitivity at batch size 64.}
Median paired late-score advantages over \((0.5,0.5)\) across five paths.
Panels show the three network widths. Boxes mark the classical pair;
the bottom row is \texttt{SASS}.}
\label{fig:exp_nn_batch64_grid}
\end{figure}

\begin{figure}[H]
\centering
\includegraphics[width=\linewidth]{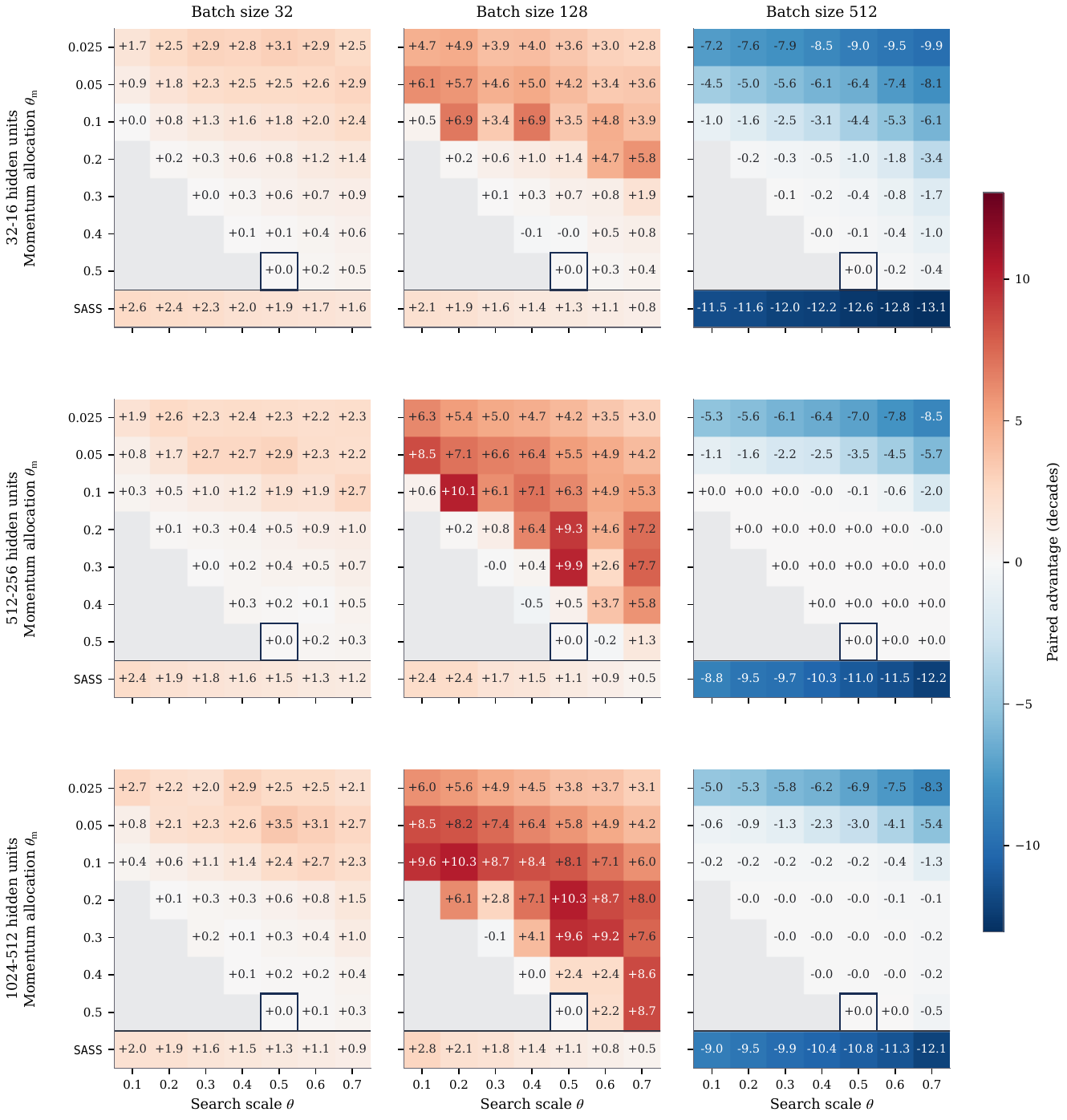}
\caption{\textbf{Neural-network scale sensitivity.}
Median paired late-score advantages over \((\theta,\theta_{\mathrm m})=(0.5,0.5)\)
across five paths. Rows show network widths; columns show batch sizes
32, 128, and 512. Positive values favor the cell, boxes mark the
classical pair, and gray cells have \(\theta_{\mathrm m}>\theta\).
The bottom row of each panel is \texttt{SASS}.}
\label{fig:exp_nn_late_heatmap}
\end{figure}
\FloatBarrier
\endgroup

\end{document}